\documentclass{article}
\usepackage[margin=1in]{geometry}
\usepackage{amsmath, amssymb}
\usepackage{amsthm}
\usepackage{amsmath}
\usepackage{bm}
\usepackage[hidelinks]{hyperref}
\usepackage{algorithm}
\usepackage{algpseudocode}
\usepackage{setspace}
\usepackage{graphicx}
\usepackage{subfig}
\newcommand{\bu}{{\bf u}}

\newcommand{\bW}{{\bf W}}

\newcommand{\bH}{\mathbf{H}}

\newcommand{\bU}{{\bf U}}
\newcommand{\bR}{{\bf R}}
\newcommand{\bQ}{{\bf Q}}
\newcommand{\bP}{{\bf P}}
\newcommand{\bZ}{{\bf Z}}
\newcommand{\bS}{{\bf S}}
\newcommand{\bg}{{\bf g}}
\newcommand{\bh}{{\bf h}}
\newcommand{\bV}{{\bf V}}
\newcommand{\x}{{\bf x}}

\newcommand{\R}{\mathbb{R}}
\newcommand{\Div}{{\rm div}}

\newcommand{\bff}{\mathbf{f}}

\newcommand{\bGamma}{\bm{\Gamma}}

\newtheorem{Theorem}{Theorem}
\newtheorem{Proposition}{Proposition}
\newtheorem{Corollary}{Corollary}
\newtheorem{Problem}{Problem}
\newtheorem{Lemma}{Lemma}
\newtheorem{Remark}{Remark}

\numberwithin{equation}{section}

\title{A Carleman contraction method for inverse initial data recovery in the Navier--Stokes equations with unknown body force}
\date{}

 \author{
       Phuong M. Nguyen\thanks{Department of Mathematics and Statistics, University of North Carolina at Charlotte, Charlotte, NC, 28223, USA, \texttt{pnguye45@charlotte.edu}, corresponding author.}  
     \and
        Loc H. Nguyen\thanks{Department of Mathematics and Statistics, University of North Carolina at Charlotte, Charlotte, NC, 28223, USA,  \texttt{loc.nguyen@charlotte.edu}.}
    }

\begin{document}

\maketitle

\begin{abstract}
We solve an inverse initial data problem for the incompressible Navier--Stokes system. The objective is to recover the initial velocity and pressure from lateral boundary observations, without assuming that the time-independent body force is known.
 To eliminate this unknown force, we differentiate the momentum equation with respect to time and then apply a Legendre polynomial--exponential time-dimensional reduction. This procedure yields a coupled system of elliptic equations for the expansion coefficients. We then construct a contractive map for this reduced system on a suitable admissible set equipped with a Carleman-weighted norm. Its fixed point yields an approximate solution of the time-dimensional reduction model, and the contraction property gives rise to a globally convergent Picard iteration. Finally, we present a numerical algorithm based on this framework and numerical experiments showing accurate reconstructions of the initial velocity and pressure from synthetic boundary data.
\end{abstract}


\medskip


\section{Introduction}

Let $d \ge 2$ denote the spatial dimension, let $\Omega \subset \R^d$ be a domain, and let $T>0$ be a final time. Let $\bu : \Omega \times (0,T) \to \R^d$ be the velocity field and $p : \Omega \times (0,T) \to \R$ be the pressure. We consider the incompressible Navier--Stokes system
\begin{equation}
\begin{cases}
\bu_t + (\bu \cdot \nabla)\bu
= -\nabla p + \mu \Delta \bu + \mathbf{f} & (\x, t) \in \Omega \times (0,T),\\
\Div \,\bu = 0 & (\x, t) \in \Omega \times (0,T),\\
\bu(\x, t) = 0 & (\x, t) \in \partial \Omega \times (0,T),\\
\bu(\x, 0) = \bu_0 & \x \in \Omega,
\end{cases}
\label{Navier-Stokes}
\end{equation}
where $\bu_0$ is the initial velocity, $\mu > 0$ is the viscosity coefficient, and $\mathbf{f} : \R^d \to \R^d$ is a time-independent body force. The first equation in \eqref{Navier-Stokes} expresses the momentum balance, while the second enforces incompressibility. We study the following nonlinear inverse problem.

\begin{Problem}[Initial data inverse problem for the Navier--Stokes system]\label{ISP}
Let $\nu$ denote the outward unit normal vector on $\partial\Omega$. Assume that $\bff$ is unknown. Given the lateral boundary data
\begin{equation}\label{data}
  \bg(\x,t):=\partial_{\nu}\bu(\x,t), 
  \qquad
  \bh(\x,t):=
  \begin{bmatrix}
  p(\x,t)\\
    \partial_{\nu} p(\x,t)
  \end{bmatrix}
  \quad \text{for } (\x,t)\in \partial \Omega \times (0,T),
\end{equation}
determine the initial data $\bu_0=\bu(\x,0)$ and $p_0=p(\x,0)$ in $\Omega$.
\end{Problem}

It is useful to clarify the practical meaning of the data in \eqref{data}. Although Problem~\ref{ISP} is formulated in terms of the normal derivatives $\partial_\nu \bu$ and $\partial_\nu p$, these quantities can be interpreted as idealized limits of finite-difference quotients computed from measurements near the boundary. For example, for $\x\in \partial\Omega$ and $\delta>0$ sufficiently small,
\[
\partial_\nu p(\x,t)\approx \frac{p(\x,t)-p(\x-\delta \nu(\x),t)}{\delta},
\]
and similarly for each component of $\bu$. Thus, the data $(\partial_\nu \bu,p,\partial_\nu p)$ on $\partial\Omega\times(0,T)$ should be viewed as a mathematical idealization of high-resolution measurements taken on the boundary and on a nearby interior layer parallel to it.
The inverse initial data problem for the incompressible Navier--Stokes system is relevant in a variety of applications, including aerospace and automotive engineering, hydraulic systems, oceanography, atmospheric modeling, and environmental fluid transport~\cite{Lions1998, Stokes1851, Tsai2018, White1991}. In such settings, direct access to the initial velocity and pressure fields is often infeasible, while accurate knowledge of these quantities is important for initializing predictive simulations and understanding the subsequent evolution of the flow. Recovering the initial state from lateral boundary measurements is therefore a meaningful inverse problem.

Since the analysis and numerical treatment of inverse problems depend on the underlying forward model, it is useful to briefly recall the status of the forward Navier--Stokes system. While the global regularity problem for the three-dimensional incompressible Navier--Stokes equations remains open, substantial progress has been made in several related settings; see, for example, \cite{Feireisl2004, Ladyzhenskaya1969, Lions1998, NovotnyStraskraba2004, Temam1979}. In particular, the existence of global weak solutions for stationary problems with homogeneous Dirichlet boundary conditions was established in \cite{Lions1998} and further developed in \cite{FrehseGojSteinhauer2005, NovotnyStraskraba2004, PlotnikovSokolowski2004, PlotnikovSokolowski2005}. Additional results include uniqueness and differentiability for small-data stationary problems \cite{PlotnikovRubanSokolowski2008}, well-posedness for coupled Navier--Stokes--Cahn--Hilliard and Boussinesq-type systems \cite{ AntontsevKhompysh2023, FangNeiGuo2025, HeWu2021}, and existence and asymptotic behavior for stationary Navier--Stokes equations in higher-dimensional half-spaces \cite{Fujii2025}. These works provide a broad analytical background for inverse problems associated with viscous fluid models.

Inverse problems for the Navier--Stokes equations have been studied in several directions, including the recovery of coefficients, geometries, and initial data from boundary or interior observations. In \cite{FouresteyMoubachir2005}, a numerical approximation scheme based on a Lagrange--Galerkin method was developed for inverse problems related to drag reduction and velocity identification. A regularized Gauss--Newton method for shape reconstruction in two-dimensional Navier--Stokes equations was proposed in \cite{YanHeMa2010}, where differentiability with respect to the boundary curve was established. Global identifiability of the viscosity in two-dimensional Stokes and Navier--Stokes equations from Cauchy force measurements was proved in \cite{LaiUhlmannWang2015}. Global uniqueness results for inverse boundary value problems for the two-dimensional Navier--Stokes equations and the Lam\'e system were obtained in \cite{ImanuvilovYamamoto2015a} by means of complex geometric optics solutions. More recently, iterative procedures for retrospective inverse problems, including a simple backward integration method, were introduced in \cite{OConnorEtAl2024}. In addition, global uniqueness for viscosity determination in Stokes and Navier--Stokes equations in both two and three dimensions was established in \cite{Liu2024}. These works demonstrate the breadth of inverse problems for fluid equations, but they primarily focus on coefficient, viscosity, or shape identification. By contrast, the present paper focuses on recovering the initial velocity and pressure from lateral boundary data when the body force is unknown.
A distinctive feature of Problem~\ref{ISP} is that the body force $\bff$ is not prescribed. In most inverse problems, all model ingredients except the target unknown are assumed known a priori. For example, the formulation in \cite{VanLeNguyen} requires prior knowledge of both the body force $\bff$ and the pressure $p$. In contrast, here neither the body force nor the pressure is assumed known in advance; rather, the goal is to recover the initial velocity and pressure $(\bu_0,p_0)$ solely from the lateral boundary measurements \eqref{data}. This formulation is closer to realistic situations, where external forcing may be unavailable, only partially observed, or contaminated by uncertainty.

From a mathematical perspective, Problem~\ref{ISP} is highly challenging. The governing system is nonlinear, the observations are restricted to the lateral boundary, and one seeks to determine interior initial data without knowing all terms in the model. A standard approach for nonlinear inverse problems is to formulate the reconstruction as an optimization problem. For instance, one may minimize an output least-squares functional measuring the discrepancy between the observed boundary data and the boundary data generated by a trial pair $(\bu_0,p_0)$, and then solve the resulting minimization problem by an adjoint-based gradient method. Because the forward operator is nonlinear, however, this cost functional is typically nonconvex and may admit multiple local minima. As a result, such methods depend strongly on the availability of a good initial guess, which is often unavailable in practice.

To overcome these difficulties, we follow a different strategy. We first eliminate the unknown body force $\bff$ by differentiating the momentum equation with respect to time. We then apply a Legendre polynomial--exponential time-dimensional reduction, inspired by the exponentially weighted Legendre basis introduced in \cite{TrongElastic} and later used in \cite{LeVanDangNguyen}. This reduction transforms the original inverse problem into a finite-dimensional coupled elliptic system for the coefficient vectors of $\bu$ and $p$. We solve the reduced system by a Carleman--Picard method, which avoids the need for a good initial guess by constructing a contraction map on a suitable admissible set and proving that the associated operator has a unique fixed point. At both the analytical and numerical levels, this yields a systematic reconstruction framework for Problem~\ref{ISP}. We also note the related work \cite{VanLeNguyen}, where a damped Picard iteration was introduced for recovering the initial velocity in an anisotropic Navier--Stokes setting. That approach, however, requires prior knowledge of both the pressure $p$ and the body force $\bff$, and a convergence proof for the damped Picard iteration was not established there.

The combination of time-dimensional reduction and the Carleman contraction principle was used in  \cite{LeNguyen:jiip2022} for an inverse initial-value problem for a quasilinear parabolic equation.
 Later, \cite{Nguyen:AVM2023} showed that this approach can be interpreted as the construction of a contraction mapping whose fixed point is the desired solution, thereby yielding global convergence of the associated Picard iteration even when the initial guess is far from the true solution. Since then, this framework has been extended to several inverse problems; see, for example, \cite{AbneyLeNguyenPeters, LeCON2023, LeKlibanov:ip2022, LeNguyenNguyenPark, NguyenKlibanov:ip2022, NguyenNguyenVu2026,  DangNguyenVu, VanLeNguyen}. However, these earlier methods are not directly applicable to the present problem due to the complex structure of the incompressible Navier--Stokes system. In particular, the nonlinear terms in the reduced system involve coupled transport-type interactions between the velocity and pressure variables, and they do not satisfy the structural assumptions imposed on the nonlinearity in \cite{LeNguyen:jiip2022, Nguyen:AVM2023}. For this reason, the main contribution of the present paper is to develop a Carleman contraction method specifically adapted to this inverse initial data problem by combining force elimination with a Legendre-polynomial exponential-time dimensional reduction.

% The main contributions of this paper are as follows.
% \begin{enumerate}
%     \item We formulate an inverse initial data problem for the incompressible Navier--Stokes system without assuming that the time-independent body force is known, with the aim of recovering both the initial velocity and the initial pressure from lateral boundary data.
%     \item We eliminate the unknown body force by differentiating the momentum equation with respect to time and derive a Legendre polynomial--exponential time-dimensional reduction, which leads to a finite-dimensional coupled elliptic system for the expansion coefficients.
%     \item We develop a Carleman contraction framework for the reduced system and prove that the associated operator is contractive on a suitable admissible set, which yields the existence and uniqueness of its fixed point.
%     \item We further prove that this fixed point is consistent with a solution of the reduced system as the regularization parameter tends to zero.
%     \item We propose a numerical reconstruction algorithm based on this framework and demonstrate its effectiveness through numerical experiments with synthetic boundary data.
% \end{enumerate}

The remainder of the paper is organized as follows. In Section~\ref{sec:LPexp_basis}, we recall the Legendre polynomial--exponential basis and establish the auxiliary properties needed for the time-dimensional reduction. In Section~\ref{sec:time_reduction_system}, we derive the reduced system for the coefficient vectors together with its Cauchy data. Section~\ref{sec:contraction} is devoted to the Carleman contraction framework, where we prove the contractive property of the reduced operator and establish the existence and uniqueness of its fixed point. In Section~\ref{sec:numerics}, we describe the reconstruction procedure, the numerical implementation, and several numerical examples illustrating the performance of the proposed method. Finally, Section~\ref{sec:conclusion} contains concluding remarks and possible directions for future work.

\section{The Legendre polynomial--exponential basis}\label{sec:LPexp_basis}

In this section, we recall the Legendre polynomial--exponential basis introduced in \cite{TrongElastic} and collect the two properties needed later: the structure of the matrix $S_N$ and the commutation of time derivatives with the corresponding basis expansion.
For each $n\ge 0$, let $P_n$ be the Legendre polynomial of degree $n$ on $(-1,1)$, defined by Rodrigues' formula
\begin{equation*}%\label{eq:Pn_Rodrigues}
P_n(x)=\frac{1}{2^n n!}\frac{{\rm d}^n}{{\rm d}x^n}(x^2-1)^n,
\qquad x\in(-1,1).
\end{equation*}
Mapping $t\in(0,T)$ to $x=\frac{2t}{T}-1\in(-1,1)$, we define the shifted and normalized Legendre functions on $(0,T)$ by
\begin{equation}\label{eq:Qn_def}
Q_n(t)=\sqrt{\frac{2n+1}{T}}\,P_n\!\left(\frac{2t}{T}-1\right),
\qquad t\in(0,T).
\end{equation}
Then $\{Q_n\}_{n\ge 0}$ is an orthonormal basis of $L^2(0,T)$, that is,
\[
\int_0^T Q_n(t)Q_m(t)\,{\rm d}t=\delta_{mn},
\qquad m,n\ge 0.
\]

On $L^2(0,T)$, we use the exponentially weighted inner product
\[
\langle u,v\rangle_{e^{-2t}}
:=\int_0^T e^{-2t}u(t)v(t)\,{\rm d}t,
\qquad
\|u\|_{e^{-2t}}
:=\langle u,u\rangle_{e^{-2t}}^{1/2}.
\]
Since $e^{-2t}\in [e^{-2T},1]$ for all $t\in[0,T]$, the norm $\|\cdot\|_{e^{-2t}}$ is equivalent to the standard $L^2(0,T)$ norm.  
We define the \textit{Legendre polynomial--exponential} functions by
\begin{equation}\label{eq:Psin_def}
\Psi_n(t):=e^tQ_n(t),
\qquad t\in(0,T),\ n\ge 0.
\end{equation}
Then $\{\Psi_n\}_{n\ge 0}$ is an orthonormal basis of $L^2(0,T)$ with respect to the weighted inner product $\langle \cdot,\cdot\rangle_{e^{-2t}}$. Indeed,
\[
\langle \Psi_n,\Psi_m\rangle_{e^{-2t}}
=\int_0^T e^{-2t}\Psi_n(t)\Psi_m(t)\,{\rm d}t
=\int_0^T Q_n(t)Q_m(t)\,{\rm d}t
=\delta_{mn}.
\]

The following proposition describes the structure of the matrix generated by the time derivatives of the basis functions.

\begin{Proposition}\label{prop:SN_structure}
Let $\{\Psi_n\}_{n\ge 0}$ be defined by \eqref{eq:Psin_def}. For $m,n\ge 0$, define
\begin{equation}\label{eq:smn_def}
s_{mn}:=\int_0^T e^{-2t}\,\Psi_n'(t)\Psi_m(t)\,dt,
\end{equation}
and, for each $N\ge 0$, let
\[
S_N:=\big[s_{mn}\big]_{m,n=0}^N.
\]
Then
\[
s_{mn}=
\begin{cases}
1, & m=n,\\
0, & m>n.
\end{cases}
\]
Consequently, $S_N$ is an upper triangular matrix with ones on the diagonal. In particular, $S_N$ is invertible for every $N\ge 0$.
\end{Proposition}

\begin{proof}
The proof is a modification of the argument in \cite{Klibanov:jiip2017}, where an analogous result was proved for another exponential-based basis.
By \eqref{eq:Psin_def}, we have
\[
\Psi_n(t)=e^tQ_n(t),
\qquad
\Psi_n'(t)=e^t\big(Q_n'(t)+Q_n(t)\big).
\]
Substituting this into \eqref{eq:smn_def}, we obtain
\[
s_{mn}
=\int_0^T e^{-2t}\Psi_n'(t)\Psi_m(t)\,dt
=\int_0^T \big(Q_n'(t)+Q_n(t)\big)Q_m(t)\,dt
=\int_0^T Q_n'(t)Q_m(t)\,dt+\delta_{mn},
\]
where we used the orthonormality of $\{Q_n\}_{n\ge 0}$ in $L^2(0,T)$.

Since $Q_n$ is a polynomial of degree $n$, the derivative $Q_n'$ is a polynomial of degree at most $n-1$. On the other hand, $Q_m$ is orthogonal in $L^2(0,T)$ to every polynomial of degree at most $m-1$. Therefore, whenever $m>n$, we have
\[
\deg(Q_n')\le n-1\le m-1,
\]
and hence
\[
\int_0^T Q_n'(t)Q_m(t)\,dt=0.
\]
It follows that
\[
s_{mn}=0,
\qquad m>n.
\]

Next, for the diagonal entries,
\[
\int_0^T Q_n'(t)Q_n(t)\,dt
=\frac12\big[Q_n(t)^2\big]_0^T.
\]
Using \eqref{eq:Qn_def} together with the identities $P_n(-1)=(-1)^n$ and $P_n(1)=1$, we obtain
\[
Q_n(0)=(-1)^n\sqrt{\frac{2n+1}{T}},
\qquad
Q_n(T)=\sqrt{\frac{2n+1}{T}}.
\]
Hence $Q_n(0)^2=Q_n(T)^2$, so that
\[
\big[Q_n(t)^2\big]_0^T=0.
\]
Therefore,
\[
\int_0^T Q_n'(t)Q_n(t)\,dt=0,
\]
and thus
\[
s_{nn}=\int_0^T Q_n'(t)Q_n(t)\,dt+1=1.
\]
The proof is complete.
\end{proof}

Next, let $s\ge 0$. For any $w\in L^2((0,T);H^s(\Omega))$, we define its $\{\Psi_n\}$--coefficients by
\begin{equation}\label{eq:wn_def}
w_n(\x):=\int_0^T e^{-2t}\,w(\x,t)\Psi_n(t)\,dt,
\qquad n\ge 0.
\end{equation}

The following proposition shows that the first and second time derivatives commute with the $\{\Psi_n\}$--expansion. We refer to \cite[Theorem 1]{VanLeNguyen} and \cite[Theorem 1]{TrongElastic} for the proofs.

\begin{Proposition}[Commutation of time derivatives with the $\{\Psi_n\}$--expansion]\label{prop2}
Let $s\ge 0$, and let $\{\Psi_n\}_{n\ge 0}$ be defined by \eqref{eq:Psin_def}. Assume that
\[
w\in H^k((0,T);H^s(\Omega))
\]
for some integer $k\ge 5$, and let $\{w_n\}_{n\ge 0}$ be the coefficients of $w$ defined in \eqref{eq:wn_def}. Then
\[
\partial_t w=\sum_{n=0}^\infty w_n\Psi_n'
\quad\text{in } L^2((0,T);H^s(\Omega)),
\]
and
\[
\partial_t^2 w=\sum_{n=0}^\infty w_n\Psi_n''
\quad\text{in } L^2((0,T);H^s(\Omega)).
\]
\end{Proposition}

\section{The time-dimensional reduction model} \label{sec:time_reduction_system}

We now derive a time-dimensional reduction model for Problem~\ref{ISP}. The derivation proceeds in four steps. First, we eliminate the time-independent body force $\mathbf f$ by differentiating the Navier--Stokes system with respect to time. Next, we expand the unknowns in the Legendre polynomial--exponential basis. We then truncate the resulting series and project the equations onto the finite-dimensional basis. Finally, we derive the Cauchy data for the reduced unknowns from the measured lateral boundary data.

A major difficulty in Problem~\ref{ISP} is that one aims to recover the interior behavior of $2d+1$ unknown functions, namely the $d$ components of $\bu$, the pressure $p$, and the $d$ components of the body force $\mathbf f$, while only the $d+1$ coupled equations in \eqref{Navier-Stokes} are available. A first key step is therefore to eliminate $\mathbf f$. Since $\mathbf f$ is independent of time, we differentiate the momentum equation in \eqref{Navier-Stokes} with respect to $t$. This removes $\mathbf f$ from the system and yields
\begin{equation}\label{eq:NS_time_diff}
\bu_{tt}+(\bu_t\cdot\nabla)\bu+(\bu\cdot\nabla)\bu_t
=-\nabla p_t+\mu\Delta \bu_t
\qquad \text{in } \Omega \times (0, T).
\end{equation}
Taking the divergence of \eqref{eq:NS_time_diff}
\[
    \Delta p_t = \Div(-\bu_{tt} - (\bu_t\cdot\nabla)\bu -(\bu\cdot\nabla)\bu_t + \mu \Delta \bu_t) \quad \mbox{in } \Omega \times (0, T).
\]
Using $\Div\,\bu=0$ and hence $\Div\,\bu_{tt}=0$ and $\mu \Div \Delta \bu_t = \mu  \Delta \Div(\bu_t) = 0$, we obtain
\begin{align*}
    \Delta p_t &= \Div( - (\bu_t\cdot\nabla)\bu-(\bu\cdot\nabla)\bu_t ) \\
    &=  -\sum_{i = 1}^d \partial_{x_i}[\sum_{j = 1}^d \partial_t u_j \partial_{x_j} u_i] 
    - \sum_{i = 1}^d \partial_{x_i}[\sum_{j = 1}^d  u_j  \partial_{x_j} \partial_t u_i]\\
    &= -\sum_{i, j = 1}^d \partial_{x_i} (\partial_t u_j) \partial_{x_j} u_i - \sum_{i, j = 1}^d  \partial_t u_j \partial_{x_i x_j}  u_i
    - \sum_{i, j = 1}^d \partial_{x_i}  u_j  \partial_{x_j} (\partial_t u_i) - \sum_{i, j = 1}^d   u_j  \partial_{x_ ix_j} \partial_t u_i.    
\end{align*}
Note that
\begin{align*}
    \sum_{i, j = 1}^d  \partial_t u_j \partial_{x_i x_j}  u_i    &= \sum_{j = 1}^d \partial_t u_j \partial_{x_j} \sum_{i = 1}^d \partial_{x_i} u_i = \sum_{j = 1}^d \partial_t u_j \partial_{x_j} \Div \bu = 0,\\
    \sum_{i, j = 1}^d   u_j  \partial_{x_ ix_j} \partial_t u_i
    &= \sum_{j = 1}^d u_j \partial_{x_j} \sum_{i = 1}^d  \partial_{x_i}(\partial_t  u_i)
    = \sum_{j = 1}^d u_j \partial_{x_j} \Div \bu_t = 0.
\end{align*}
Hence,
\begin{equation}\label{eq:poisson_q_simplified}
\Delta p_t
=-\nabla \bu_t:(\nabla \bu)^\top-\nabla \bu:(\nabla \bu_t)^\top
\qquad \text{in } \Omega \times (0, T),
\end{equation}
where ${}^\top$ denotes matrix transpose.

We next expand $\bu$ and $p$ in the Legendre polynomial--exponential basis $\{\Psi_n\}_{n\ge 0}$. Specifically,
\[
\bu(\x,t)=\sum_{n=0}^{\infty}\bu_n(\x)\Psi_n(t),
\qquad
p(\x,t)=\sum_{n=0}^{\infty}p_n(\x)\Psi_n(t),
\]
where
\[
\bu_n(\x)=\int_0^T e^{-2t}\,\bu(\x,t)\Psi_n(t)\,dt,
\qquad
p_n(\x)=\int_0^T e^{-2t}\,p(\x,t)\Psi_n(t)\,dt.
\]
By Proposition~\ref{prop2}, the first and second time derivatives admit the expansions
\[
\bu_t(\x,t)=\sum_{n=0}^{\infty}\bu_n(\x)\Psi_n'(t),
\qquad
p_t(\x,t)=\sum_{n=0}^{\infty}p_n(\x)\Psi_n'(t),
\qquad
\bu_{tt}(\x,t)=\sum_{n=0}^{\infty}\bu_n(\x)\Psi_n''(t).
\]

Substituting these expressions into \eqref{eq:NS_time_diff} and \eqref{eq:poisson_q_simplified}, we obtain
\begin{multline}\label{eq:NS_time_diff_series}
\sum_{n=0}^{\infty}\bu_n(\x)\Psi_n''(t)
+\Big(\Big(\sum_{n=0}^{\infty}\bu_n(\x)\Psi_n'(t)\Big)\cdot\nabla\Big)
\Big(\sum_{k=0}^{\infty}\bu_k(\x)\Psi_k(t)\Big) \\
+\Big(\Big(\sum_{k=0}^{\infty}\bu_k(\x)\Psi_k(t)\Big)\cdot\nabla\Big)
\Big(\sum_{n=0}^{\infty}\bu_n(\x)\Psi_n'(t)\Big)
=-\nabla\Big(\sum_{n=0}^{\infty}p_n(\x)\Psi_n'(t)\Big)
+\mu\Delta\Big(\sum_{n=0}^{\infty}\bu_n(\x)\Psi_n'(t)\Big),
\end{multline}
and
\begin{multline}\label{eq:poisson_q_simplified_series}
\Delta\Big(\sum_{n=0}^{\infty}p_n(\x)\Psi_n'(t)\Big)
=
-\nabla\Big(\sum_{n=0}^{\infty}\bu_n(\x)\Psi_n'(t)\Big):
\Big(\nabla\Big(\sum_{k=0}^{\infty}\bu_k(\x)\Psi_k(t)\Big)\Big)^\top \\
-\nabla\Big(\sum_{k=0}^{\infty}\bu_k(\x)\Psi_k(t)\Big):
\Big(\nabla\Big(\sum_{n=0}^{\infty}\bu_n(\x)\Psi_n'(t)\Big)\Big)^\top.
\end{multline}

To obtain a finite-dimensional reduced model, we truncate the above series at level $N\in\mathbb N$. This truncation yields the following projected approximation of \eqref{eq:NS_time_diff_series} and \eqref{eq:poisson_q_simplified_series}:
%For numerical implementation, we truncate the above series at level $N\in\mathbb N$. Then \eqref{eq:NS_time_diff_series} and \eqref{eq:poisson_q_simplified_series} are approximated by
\begin{multline}\label{eq:NS_time_diff_series_N}
\sum_{n=0}^{N}\bu_n(\x)\Psi_n''(t)
+\Big(\Big(\sum_{n=0}^{N}\bu_n(\x)\Psi_n'(t)\Big)\cdot\nabla\Big)
\Big(\sum_{k=0}^{N}\bu_k(\x)\Psi_k(t)\Big) \\
+\Big(\Big(\sum_{k=0}^{N}\bu_k(\x)\Psi_k(t)\Big)\cdot\nabla\Big)
\Big(\sum_{n=0}^{N}\bu_n(\x)\Psi_n'(t)\Big)
=-\nabla\Big(\sum_{n=0}^{N}p_n(\x)\Psi_n'(t)\Big)
+\mu\Delta\Big(\sum_{n=0}^{N}\bu_n(\x)\Psi_n'(t)\Big),
\end{multline}
and
\begin{multline}\label{eq:poisson_q_simplified_series_N}
\Delta\Big(\sum_{n=0}^{N}p_n(\x)\Psi_n'(t)\Big)
=
-\nabla\Big(\sum_{n=0}^{N}\bu_n(\x)\Psi_n'(t)\Big):
\Big(\nabla\Big(\sum_{k=0}^{N}\bu_k(\x)\Psi_k(t)\Big)\Big)^\top \\
-\nabla\Big(\sum_{k=0}^{N}\bu_k(\x)\Psi_k(t)\Big):
\Big(\nabla\Big(\sum_{n=0}^{N}\bu_n(\x)\Psi_n'(t)\Big)\Big)^\top.
\end{multline}
In the remainder of this paper, we regard the truncated projected system as the reduced model associated with Problem~\ref{ISP}. The analysis in the following sections is carried out for this finite-dimensional coefficient system, which approximates the original inverse problem.

 For each $m=0,1,\dots,N$, we multiply \eqref{eq:NS_time_diff_series_N} and \eqref{eq:poisson_q_simplified_series_N} by $e^{-2t}\Psi_m(t)$ and integrate over $(0,T)$. Introduce the coefficients
\[
r_{mn}:=\int_0^T e^{-2t}\Psi_n''(t)\Psi_m(t)\,dt,
\qquad
s_{mn}:=\int_0^T e^{-2t}\Psi_n'(t)\Psi_m(t)\,dt,
\]
and
\[
c_{mkn}:=\int_0^T e^{-2t}\Psi_n'(t)\Psi_k(t)\Psi_m(t)\,dt,
\qquad m,k,n=0,1,\dots,N.
\]
Then, for each $m=0,1,\dots,N$, we obtain
\begin{multline*}%\label{eq:projected_NS_N_m}
\sum_{n=0}^N r_{mn}\,\bu_n(\x)
+\sum_{k,n=0}^N c_{mkn}\,(\bu_n(\x)\cdot\nabla)\bu_k(\x)
+\sum_{k,n=0}^N c_{mkn}\,(\bu_k(\x)\cdot\nabla)\bu_n(\x) \\
=-\sum_{n=0}^N s_{mn}\,\nabla p_n(\x)
+\mu\sum_{n=0}^N s_{mn}\,\Delta \bu_n(\x),
\qquad \x\in\Omega,
\end{multline*}
and
\begin{equation*}%\label{eq:projected_Poisson_N_m}
\sum_{n=0}^N s_{mn}\,\Delta p_n(\x)
=
-\sum_{k,n=0}^N c_{mkn}\,\nabla \bu_n(\x):\big(\nabla \bu_k(\x)\big)^\top
-\sum_{k,n=0}^N c_{mkn}\,\nabla \bu_k(\x):\big(\nabla \bu_n(\x)\big)^\top,
\qquad \x\in\Omega.
\end{equation*}

Rearranging the highest-order spatial terms to the left-hand side gives
\begin{multline}\label{eq:projected_NS_N_m_rearranged}
\mu\sum_{n=0}^N s_{mn}\,\Delta \bu_n(\x)
=
\sum_{n=0}^N r_{mn}\,\bu_n(\x)
+\sum_{k,n=0}^N c_{mkn}\,(\bu_n(\x)\cdot\nabla)\bu_k(\x) \\
+\sum_{k,n=0}^N c_{mkn}\,(\bu_k(\x)\cdot\nabla)\bu_n(\x)
+\sum_{n=0}^N s_{mn}\,\nabla p_n(\x),
\qquad \x\in\Omega,
\end{multline}
and
\begin{equation}\label{eq:projected_Poisson_N_m_rearranged}
\sum_{n=0}^N s_{mn}\,\Delta p_n(\x)
=
-\sum_{k,n=0}^N c_{mkn}\,\nabla \bu_n(\x):\big(\nabla \bu_k(\x)\big)^\top
-\sum_{k,n=0}^N c_{mkn}\,\nabla \bu_k(\x):\big(\nabla \bu_n(\x)\big)^\top,
\qquad \x\in\Omega.
\end{equation}

To write the projected system more compactly, define
\[
\bU^N(\x):=
\begin{bmatrix}
\bu_0(\x)\\
\bu_1(\x)\\
\vdots\\
\bu_N(\x)
\end{bmatrix},
\qquad
P^N(\x):=
\begin{bmatrix}
p_0(\x)\\
p_1(\x)\\
\vdots\\
p_N(\x)
\end{bmatrix},
\]
and let
\[
S_N:=[s_{mn}]_{m,n=0}^N,
\qquad
R_N:=[r_{mn}]_{m,n=0}^N.
\]
For each $m=0,1,\dots,N$, set
\begin{align*}
\mathbf F_m(\x,\bU^N,\nabla\bU^N)
&:=
\sum_{k,n=0}^N c_{mkn}\,(\bu_n(\x)\cdot\nabla)\bu_k(\x)
+\sum_{k,n=0}^N c_{mkn}\,(\bu_k(\x)\cdot\nabla)\bu_n(\x),
%\label{eq:Fm_def_short}
\\
G_m(\x,\bU^N,\nabla\bU^N)
&:=
-\sum_{k,n=0}^N c_{mkn}\,\nabla \bu_n(\x):\big(\nabla \bu_k(\x)\big)^\top
-\sum_{k,n=0}^N c_{mkn}\,\nabla \bu_k(\x):\big(\nabla \bu_n(\x)\big)^\top.
%\label{eq:Gm_def_short}
\end{align*}
We also define the stacked nonlinear terms
\[
\mathbf F^N(\x,\bU^N,\nabla\bU^N):=
\begin{bmatrix}
\mathbf F_0(\x,\bU^N,\nabla\bU^N)\\
\vdots\\
\mathbf F_N(\x,\bU^N,\nabla\bU^N)
\end{bmatrix},
\qquad
G^N(\x,\bU^N,\nabla\bU^N):=
\begin{bmatrix}
G_0(\x,\bU^N,\nabla\bU^N)\\
\vdots\\
G_N(\x,\bU^N,\nabla\bU^N)
\end{bmatrix}.
\]
Then \eqref{eq:projected_NS_N_m_rearranged}--\eqref{eq:projected_Poisson_N_m_rearranged} can be written as
\begin{equation*}%\label{eq:UNPN_system_1}
\mu\,S_N(\Delta \bU^N)
=
R_N\bU^N+\mathbf F^N(\x,\bU^N,\nabla \bU^N)+S_N(\nabla P^N)
\qquad \text{in } \Omega,
\end{equation*}
and
\begin{equation*}%\label{eq:UNPN_system_2}
S_N(\Delta P^N)
=
G^N(\x,\bU^N,\nabla \bU^N)
\qquad \text{in } \Omega,
\end{equation*}
where the matrices $S_N$ and $R_N$ act on the coefficient index $n=0,\dots,N$.

It remains to determine the boundary data for the reduced unknowns. These are obtained by projecting the measured lateral data in \eqref{data} onto the basis $\{\Psi_n\}_{n=0}^N$. Recall that $\bu=0$ on $\partial\Omega\times(0,T)$ and
\[
\partial_\nu \bu(\x,t)=\bg(\x,t),
\qquad
\bh(\x,t)=
\begin{bmatrix}
h_1(\x,t)\\
h_2(\x,t)
\end{bmatrix}
=
\begin{bmatrix}
p(\x,t)\\
\partial_\nu p(\x,t)
\end{bmatrix},
\qquad (\x,t)\in\partial\Omega\times(0,T).
\]
For each $n=0,1,\dots,N$, the boundary traces of the coefficients are
\begin{equation}\label{eq:Utrace_on_Gamma}
\bu_n(\x)=\int_0^T e^{-2t}\,\bu(\x,t)\Psi_n(t)\,dt=0,
\qquad \x\in\partial\Omega,
\end{equation}
and
\begin{equation}\label{eq:Ptrace_on_Gamma}
p_n(\x)=\int_0^T e^{-2t}\,p(\x,t)\Psi_n(t)\,dt
=\int_0^T e^{-2t}\,h_1(\x,t)\Psi_n(t)\,dt,
\qquad \x\in\partial\Omega.
\end{equation}
Similarly, their normal derivatives satisfy
\begin{equation}\label{eq:Udn_on_Gamma}
\partial_\nu \bu_n(\x)
=\int_0^T e^{-2t}\,\partial_\nu\bu(\x,t)\Psi_n(t)\,dt
=\int_0^T e^{-2t}\,\bg(\x,t)\Psi_n(t)\,dt,
\qquad \x\in\partial\Omega,
\end{equation}
and
\begin{equation}\label{eq:Pdn_on_Gamma}
\partial_\nu p_n(\x)
=\int_0^T e^{-2t}\,\partial_\nu p(\x,t)\Psi_n(t)\,dt
=\int_0^T e^{-2t}\,h_2(\x,t)\Psi_n(t)\,dt,
\qquad \x\in\partial\Omega.
\end{equation}

For later reference, we introduce
\[
\bGamma^N(\x):=
\begin{bmatrix}
\displaystyle \int_0^T e^{-2t}\,\bg(\x,t)\Psi_0(t)\,dt\\
\vdots\\
\displaystyle \int_0^T e^{-2t}\,\bg(\x,t)\Psi_N(t)\,dt
\end{bmatrix},
\qquad
\bH_1^N(\x):=
\begin{bmatrix}
\displaystyle \int_0^T e^{-2t}\,h_1(\x,t)\Psi_0(t)\,dt\\
\vdots\\
\displaystyle \int_0^T e^{-2t}\,h_1(\x,t)\Psi_N(t)\,dt
\end{bmatrix},
\]
and
\[
\bH_2^N(\x):=
\begin{bmatrix}
\displaystyle \int_0^T e^{-2t}\,h_2(\x,t)\Psi_0(t)\,dt\\
\vdots\\
\displaystyle \int_0^T e^{-2t}\,h_2(\x,t)\Psi_N(t)\,dt
\end{bmatrix},
\qquad \x\in\partial\Omega.
\]
Consequently, the Cauchy data for $\bU^N$ and $P^N$ on $\partial\Omega$ are
\begin{equation*}%\label{eq:Cauchy_UP_stacked}
\bU^N|_{\partial\Omega}(\x)=\mathbf 0,
\qquad
\partial_\nu \bU^N|_{\partial\Omega}(\x)=\bGamma^N(\x),
\qquad \x\in\partial\Omega,
\end{equation*}
and
\begin{equation*}%\label{eq:Cauchy_P_stacked}
P^N|_{\partial\Omega}(\x)=\bH_1^N(\x),
\qquad
\partial_\nu P^N|_{\partial\Omega}(\x)=\bH_2^N(\x),
\qquad \x\in\partial\Omega,
\end{equation*}
where $\mathbf 0$ denotes the $(N+1)$-vector whose entries are the zero vector in $\mathbb R^d$.

In summary, we have derived the following time-dimensional reduction system:
\begin{equation}\label{eq:time_reduction_system_final}
\begin{cases}
\mu\,S_N(\Delta \bU^N)
=
R_N\bU^N+\mathbf F^N(\x,\bU^N,\nabla \bU^N)+S_N(\nabla P^N),
& \x \in \Omega,\\[1ex]
S_N(\Delta P^N)
=
G^N(\x,\bU^N,\nabla \bU^N),
& \x \in \Omega,\\[1ex]
\bU^N|_{\partial\Omega}(\x)=\mathbf 0,
\qquad
\partial_\nu \bU^N|_{\partial\Omega}(\x)=\bGamma^N(\x),
& \x \in \partial\Omega,\\[1ex]
P^N|_{\partial\Omega}(\x)=\bH_1^N(\x),
\qquad
\partial_\nu P^N|_{\partial\Omega}(\x)=\bH_2^N(\x),
& \x \in \partial\Omega.
\end{cases}
\end{equation}

\section{The Carleman contraction framework for the reduced system}\label{sec:contraction}

In this section, we develop a Carleman contraction framework for solving the reduced system derived in Section~\ref{sec:time_reduction_system}. We first recall a Carleman estimate. We then introduce a regularized operator obtained by freezing the nonlinear terms in the reduced system. Next, we prove that this operator is contractive for sufficiently large Carleman parameter and deduce the existence and uniqueness of its fixed point. Finally, we show that this fixed point is consistent with a solution of the reduced system as the regularization parameter tends to zero.

Carleman estimates are fundamental tools in the study of partial differential equations. They were originally introduced to prove unique continuation properties; see, for instance, \cite{Carleman:1933, Protter:1960AMS}. Since then, they have become powerful tools in many areas of partial differential equations, especially in both theoretical and numerical studies of inverse problems; see, e.g., \cite{BeilinaKlibanovBook, BukhgeimKlibanov:smd1981, VoKlibanovNguyen:IP2020, KhoaKlibanovLoc:SIAMImaging2020, KlibanovLiBook, KlibanovNguyen:ip2019, LeNguyenNguyenPowell:JOSC2021, LocNguyen:ip2019}. They have also found applications in cloaking \cite{MinhLoc:tams2015} and in the computation of solutions to Hamilton--Jacobi equations \cite{KlibanovNguyenTran:JCP2022, LeNguyenTran:CAMWA2022}.

We now recall the Carleman estimate that will be used throughout this section. Let $\x_0\in \mathbb{R}^d\setminus \overline\Omega$ be such that
\[
r(\x)=|\x-\x_0|>1
\quad \text{for all } \x\in \Omega.
\]
For each $\beta>0$, define
\begin{equation*}%\label{mu}
\mu_\beta(\x)=r^{-\beta}(\x)=|\x-\x_0|^{-\beta},
\qquad \x\in \overline\Omega.
\end{equation*}

\begin{Lemma}[Carleman estimate]\label{lem:Carleman}
There exist positive constants $\beta_0$, depending only on $\x_0$, $\Omega$, and $d$, such that for every function $v\in C^2(\overline\Omega)$ satisfying
\begin{equation*}%\label{3.1}
v(\x)=\partial_\nu v(\x)=0
\quad \text{for all } \x\in \partial\Omega,
\end{equation*}
the following estimate holds:
\begin{equation*}%\label{Car est}
\int_\Omega e^{2\lambda \mu_\beta(\x)}|\Delta v|^2\,d\x
\ge
C\lambda \int_\Omega e^{2\lambda \mu_\beta(\x)}|\nabla v(\x)|^2\,d\x
+
C\lambda^3 \int_\Omega e^{2\lambda \mu_\beta(\x)}|v(\x)|^2\,d\x
\end{equation*}
for all $\beta\ge \beta_0$ and $\lambda\ge \lambda_0$. Here, $\lambda_0=\lambda_0(\x_0,\Omega,d,\beta)$ and $C=C(\x_0,\Omega,d,\beta)>0$ depend only on the listed parameters.
\end{Lemma}

Lemma~\ref{lem:Carleman} is a direct consequence of \cite[Lemma 5]{MinhLoc:tams2015}; see also \cite[Lemma 2.1]{LeNguyenTran:CAMWA2022} for details. An alternative derivation, with a different Carleman weight function, can be obtained from \cite[Chapter 4, Section 1, Lemma 3]{Lavrentiev:AMS1986}; see also \cite[Section 3]{LeNguyenNguyenPowell:JOSC2021}. We also refer the reader to \cite{BeilinaKlibanovBook, BukhgeimKlibanov:smd1981, KlibanovLiBook, LocNguyen:ip2019} for various forms of Carleman estimates and their applications. Variants under partial boundary conditions are available as well; see, for example, \cite{KlibanovNguyenTran:JCP2022, NguyenLiKlibanov:2019}.

We now apply Lemma~\ref{lem:Carleman} to the reduced system \eqref{eq:time_reduction_system_final}. Since $S_N$ is invertible by Proposition~\ref{prop:SN_structure}, system \eqref{eq:time_reduction_system_final} can be rewritten as
\begin{equation}\label{4.4}
\begin{cases}
\Delta \bU^N
=
(\mu S_N)^{-1}\big[R_N\bU^N+\mathbf F^N(\x,\bU^N,\nabla \bU^N)+S_N(\nabla P^N)\big],
& \x \in \Omega,\\[1ex]
\Delta P^N
=
S_N^{-1}G^N(\x,\bU^N,\nabla \bU^N),
& \x \in \Omega,\\[1ex]
\bU^N|_{\partial\Omega}(\x)=\mathbf 0,
\qquad
\partial_\nu \bU^N|_{\partial\Omega}(\x)=\bGamma^N(\x),
& \x \in \partial\Omega,\\[1ex]
P^N|_{\partial\Omega}(\x)=\bH_1^N(\x),
\qquad
\partial_\nu P^N|_{\partial\Omega}(\x)=\bH_2^N(\x),
& \x \in \partial\Omega.
\end{cases}
\end{equation}

We seek a solution $(\bU^N,P^N)$ of \eqref{4.4} in the admissible set
\begin{multline*}
\mathcal{A}
=
\Big\{
(\bV,\bQ)\in [H^s(\Omega)]^{d(N+1)}\times [H^s(\Omega)]^{N+1}:\ 
\|\bV\|_{[L^\infty(\Omega)]^{d(N+1)}}\le M,\ 
\|\bQ\|_{[L^\infty(\Omega)]^{N+1}}\le M,\\
\bV|_{\partial\Omega}=\mathbf 0,\ 
\partial_\nu \bV|_{\partial\Omega}=\bGamma^N,\ 
\bQ|_{\partial\Omega}=\bH_1^N,\ 
\partial_\nu \bQ|_{\partial\Omega}=\bH_2^N
\Big\},
\end{multline*}
where $M>0$ is a sufficiently large constant, and $s>\frac d2+2$ is chosen so that $H^s(\Omega)$ is continuously embedded into $C^2(\overline\Omega)$. Throughout the paper, we assume that $\mathcal A$ is nonempty. It is straightforward to verify that $\mathcal A$ is a closed and convex subset of
\begin{equation*}%\label{H}
H := [H^s(\Omega)]^{d(N+1)}\times [H^s(\Omega)]^{N+1}.
\end{equation*}

Fix $\beta>\beta_0$, let $\lambda_0$ be as in Lemma~\ref{lem:Carleman}, and let $\epsilon>0$. For each $\lambda>\lambda_0$, we define an operator
\[
\mathcal T_{\lambda,\epsilon}:\mathcal A\to\mathcal A
\]
as follows. Given any $(\bV,\bQ)\in\mathcal A$, let
\[
\mathcal T_{\lambda,\epsilon}(\bV,\bQ)
=
\underset{(\bW,\bR)\in\mathcal A}{\operatorname{argmin}}
\,J_{\lambda,\epsilon}(\bW,\bR;\bV,\bQ),
\]
where
\begin{multline*}
J_{\lambda,\epsilon}(\bW,\bR;\bV,\bQ)
:=
\int_\Omega e^{2\lambda\mu_\beta(\x)}
\left|
\Delta \bW
-(\mu S_N)^{-1}\big[
R_N\bW+\mathbf F^N(\x,\bV,\nabla \bV)+S_N(\nabla \bQ)
\big]
\right|^2
\,d\x
\\
+
\int_\Omega e^{2\lambda\mu_\beta(\x)}
\left|
\Delta \bR
-S_N^{-1}G^N(\x,\bV,\nabla \bV)
\right|^2
\,d\x
+
\epsilon \|(\bW,\bR)\|_H^2.
%\label{eq:J_frozen}
\end{multline*}

\begin{Remark}[Well-posedness of the minimization problem]%\label{rem:T_welldefined}
For each fixed $(\bV,\bQ)\in\mathcal A$, the functional $J_{\lambda,\epsilon}(\bW,\bR;\bV,\bQ)$ is convex in the variables $(\bW,\bR)$ because the nonlinear terms are frozen at $(\bV,\bQ)$. Moreover, the regularization term with $\epsilon>0$ makes this functional strictly convex. Since $\mathcal A$ is a closed and convex subset of $H$, the minimization problem defining $\mathcal T_{\lambda,\epsilon}(\bV,\bQ)$ has a unique solution. Hence, $\mathcal T_{\lambda,\epsilon}$ is well defined.
\end{Remark}

Having defined $\mathcal T_{\lambda,\epsilon}$, we next show that it is contractive on $\mathcal A$ with respect to the norm
\begin{equation}\label{norm}
\|(\bV,\bQ)\|_{\lambda,\epsilon}^2
:=
\int_\Omega e^{2\lambda\mu_\beta(\x)}
\Big(
|\bV(\x)|^2+|\nabla \bV(\x)|^2+|\bQ(\x)|^2+|\nabla \bQ(\x)|^2
\Big)\,d\x
+\frac{\epsilon}{\lambda}\|(\bV,\bQ)\|_H^2.
\end{equation}

\begin{Theorem}[Contractive property of $\mathcal T_{\lambda,\epsilon}$]\label{thm:contractive}
Fix $\beta\ge \beta_0$, and let $\lambda_0$ be as in Lemma~\ref{lem:Carleman}. Then there exists a constant $C>0$, depending only on $\x_0$, $\Omega$, $\beta$, $d$, $\mu$, $N$, $S_N$, $R_N$, $\{c_{mkn}\}_{m,k,n=0}^N$, and $M$, such that for every $\epsilon>0$, every $\lambda\ge \lambda_0$, and every
\[
(\bV_1,\bQ_1),(\bV_2,\bQ_2)\in\mathcal A,
\]
we have
\begin{equation}\label{eq:contractive_estimate}
\|\mathcal T_{\lambda,\epsilon}(\bV_1,\bQ_1)-\mathcal T_{\lambda,\epsilon}(\bV_2,\bQ_2)\|_{\lambda,\epsilon}
\le
\frac{C}{\sqrt{\lambda}}\,
\|(\bV_1,\bQ_1)-(\bV_2,\bQ_2)\|_{\lambda,\epsilon}.
\end{equation}
In particular, if $\lambda$ is chosen sufficiently large so that
\[
\theta:=\frac{C}{\sqrt{\lambda}}\in(0,1),
\]
then $\mathcal T_{\lambda,\epsilon}$ is a contraction on $\mathcal A$ with respect to the norm $\|\cdot\|_{\lambda,\epsilon}$.
\end{Theorem}

\begin{proof}
Take any two pairs $(\bV_1,\bQ_1),(\bV_2,\bQ_2)\in\mathcal A$, and define
\[
(\bW_1,\bR_1):=\mathcal T_{\lambda,\epsilon}(\bV_1,\bQ_1),
\qquad
(\bW_2,\bR_2):=\mathcal T_{\lambda,\epsilon}(\bV_2,\bQ_2).
\]
Set
\[
\widetilde{\bV}:=\bV_1-\bV_2,\qquad
\widetilde{\bQ}:=\bQ_1-\bQ_2,\qquad
\widetilde{\bW}:=\bW_1-\bW_2,\qquad
\widetilde{\bR}:=\bR_1-\bR_2.
\]
Since all pairs in $\mathcal A$ satisfy the same Cauchy data on $\partial\Omega$, we have
\[
\widetilde{\bV}|_{\partial\Omega}
=
\partial_\nu \widetilde{\bV}|_{\partial\Omega}
=
\widetilde{\bQ}|_{\partial\Omega}
=
\partial_\nu \widetilde{\bQ}|_{\partial\Omega}
=0,
\]
and
\[
\widetilde{\bW}|_{\partial\Omega}
=
\partial_\nu \widetilde{\bW}|_{\partial\Omega}
=
\widetilde{\bR}|_{\partial\Omega}
=
\partial_\nu \widetilde{\bR}|_{\partial\Omega}
=0.
\]

Let
\[
B:=(\mu S_N)^{-1}R_N,
\]
and define
\[
\mathbf f_i(\x)
:=
(\mu S_N)^{-1}\Big[\mathbf F^N(\x,\bV_i,\nabla\bV_i)+S_N(\nabla \bQ_i)\Big],
\qquad
\mathbf g_i(\x)
:=
S_N^{-1}G^N(\x,\bV_i,\nabla\bV_i),
\qquad i=1,2.
\]
Then, for each $i=1,2$, the pair $(\bW_i,\bR_i)$ is the unique minimizer over $\mathcal A$ of the functional
\[
(\bW,\bR)\mapsto
\int_\Omega e^{2\lambda\mu_\beta(\x)}
\big|\Delta \bW-B\bW-\mathbf f_i(\x)\big|^2\,d\x
+
\int_\Omega e^{2\lambda\mu_\beta(\x)}
\big|\Delta \bR-\mathbf g_i(\x)\big|^2\,d\x
+
\epsilon \|(\bW,\bR)\|_H^2.
\]

Since $\mathcal A$ is convex, the standard variational inequality for constrained minimizers implies that, for every $(\bZ,\bS)\in\mathcal A$,
\begin{multline}\label{eq:VI_i_contract_clean}
\int_\Omega e^{2\lambda\mu_\beta(\x)}
\big(\Delta \bW_i-B\bW_i-\mathbf f_i(\x)\big)\cdot
\big(\Delta(\bZ-\bW_i)-B(\bZ-\bW_i)\big)\,d\x \\
+
\int_\Omega e^{2\lambda\mu_\beta(\x)}
\big(\Delta \bR_i-\mathbf g_i(\x)\big)\cdot
\Delta(\bS-\bR_i)\,d\x
+
\epsilon \big\langle (\bW_i,\bR_i),(\bZ,\bS)-(\bW_i,\bR_i)\big\rangle_H
\ge 0.
\end{multline}

Choosing $(\bZ,\bS)=(\bW_2,\bR_2)$ in \eqref{eq:VI_i_contract_clean} with $i=1$, and using
\[
\bW_2-\bW_1=-\widetilde{\bW},
\qquad
\bR_2-\bR_1=-\widetilde{\bR},
\]
we obtain
\begin{multline}\label{eq:VI_1_rewritten_clean}
\int_\Omega e^{2\lambda\mu_\beta(\x)}
\big(\Delta \bW_1-B\bW_1-\mathbf f_1(\x)\big)\cdot
\big(\Delta\widetilde{\bW}-B\widetilde{\bW}\big)\,d\x \\
+
\int_\Omega e^{2\lambda\mu_\beta(\x)}
\big(\Delta \bR_1-\mathbf g_1(\x)\big)\cdot
\Delta\widetilde{\bR}\,d\x
+
\epsilon \big\langle (\bW_1,\bR_1),(\widetilde{\bW},\widetilde{\bR})\big\rangle_H
\le 0.
\end{multline}
Similarly, choosing $(\bZ,\bS)=(\bW_1,\bR_1)$ in \eqref{eq:VI_i_contract_clean} with $i=2$, we get
\begin{multline}\label{eq:VI_2_rewritten_clean}
\int_\Omega e^{2\lambda\mu_\beta(\x)}
\big(\Delta \bW_2-B\bW_2-\mathbf f_2(\x)\big)\cdot
\big(\Delta\widetilde{\bW}-B\widetilde{\bW}\big)\,d\x \\
+
\int_\Omega e^{2\lambda\mu_\beta(\x)}
\big(\Delta \bR_2-\mathbf g_2(\x)\big)\cdot
\Delta\widetilde{\bR}\,d\x
+
\epsilon \big\langle (\bW_2,\bR_2),(\widetilde{\bW},\widetilde{\bR})\big\rangle_H
\ge 0.
\end{multline}
Subtracting \eqref{eq:VI_1_rewritten_clean} from \eqref{eq:VI_2_rewritten_clean}, we obtain
\begin{multline}\label{eq:basic_contract_est_clean}
\int_\Omega e^{2\lambda\mu_\beta(\x)}
\big|\Delta \widetilde{\bW}-B\widetilde{\bW}\big|^2\,d\x
+
\int_\Omega e^{2\lambda\mu_\beta(\x)}
\big|\Delta \widetilde{\bR}\big|^2\,d\x
+
\epsilon \|(\widetilde{\bW},\widetilde{\bR})\|_H^2 \\
\le
\int_\Omega e^{2\lambda\mu_\beta(\x)}
(\mathbf f_1(\x)-\mathbf f_2(\x))\cdot
\big(\Delta \widetilde{\bW}-B\widetilde{\bW}\big)\,d\x
+
\int_\Omega e^{2\lambda\mu_\beta(\x)}
(\mathbf g_1(\x)-\mathbf g_2(\x))\cdot
\Delta \widetilde{\bR}\,d\x.
\end{multline}

Using $2ab\le  a^2+b^2$, we deduce from \eqref{eq:basic_contract_est_clean} that
\begin{multline}\label{eq:basic_contract_est_2_clean}
\int_\Omega e^{2\lambda\mu_\beta(\x)}
\big|\Delta \widetilde{\bW}-B\widetilde{\bW}\big|^2\,d\x
+
\int_\Omega e^{2\lambda\mu_\beta(\x)}
\big|\Delta \widetilde{\bR}\big|^2\,d\x
+
2\epsilon \|(\widetilde{\bW},\widetilde{\bR})\|_H^2 \\
\le
\int_\Omega e^{2\lambda\mu_\beta(\x)}
\big|\mathbf f_1(\x)-\mathbf f_2(\x)\big|^2\,d\x
+
\int_\Omega e^{2\lambda\mu_\beta(\x)}
\big|\mathbf g_1(\x)-\mathbf g_2(\x)\big|^2\,d\x.
\end{multline}

Because the iteration is restricted to $\mathcal A$ and $s>\frac d2+2$ implies the continuous embedding $H^s(\Omega)\hookrightarrow C^2(\overline\Omega)$, all admissible pairs are uniformly bounded together with their first derivatives. Therefore, the nonlinear mappings
\[
(\bV,\bQ)\mapsto \mathbf F^N(\x,\bV,\nabla \bV)+S_N(\nabla \bQ)
\quad\text{and}\quad
\bV\mapsto G^N(\x,\bV,\nabla \bV)
\]
are Lipschitz continuous on $\mathcal A$.
Hence, there exists a constant $C>0$, depending only on $M$, $d$, $\mu$, $N$, $S_N$, $R_N$, $\{c_{mkn}\}_{m,k,n=0}^N$, and $\Omega$, such that
\begin{equation*}%\label{eq:lipschitz_rhs_clean}
\big|\mathbf f_1(\x)-\mathbf f_2(\x)\big|
+
\big|\mathbf g_1(\x)-\mathbf g_2(\x)\big|
\le
C \Big(
|\widetilde{\bV}(\x)|
+
|\nabla\widetilde{\bV}(\x)|
+
|\widetilde{\bQ}(\x)|
+
|\nabla\widetilde{\bQ}(\x)|
\Big)
\end{equation*}
for all $\x\in\Omega$. Consequently,
\begin{multline}\label{eq:rhs_estimate_contract_clean}
\int_\Omega e^{2\lambda\mu_\beta(\x)}
\big|\mathbf f_1(\x)-\mathbf f_2(\x)\big|^2\,d\x
+
\int_\Omega e^{2\lambda\mu_\beta(\x)}
\big|\mathbf g_1(\x)-\mathbf g_2(\x)\big|^2\,d\x \\
\le
C
\int_\Omega e^{2\lambda\mu_\beta(\x)}
\Big(
|\widetilde{\bV}(\x)|^2
+
|\nabla\widetilde{\bV}(\x)|^2
+
|\widetilde{\bQ}(\x)|^2
+
|\nabla\widetilde{\bQ}(\x)|^2
\Big)\,d\x.
\end{multline}

Here and in what follows, $C$ denotes a generic positive constant, depending on $M$, $d$, $\mu$, $N$, $S_N$, $R_N$, $\{c_{mkn}\}_{m,k,n=0}^N$ and independent of the unknown functions, that may vary from estimate to estimate.
To estimate the left-hand side of \eqref{eq:basic_contract_est_2_clean}, using the inequality $|a - b|^2 \geq \frac 12 a^2 - b^2$ gives
\[
\big|\Delta \widetilde{\bW}-B\widetilde{\bW}\big|^2
\ge
\frac12 |\Delta \widetilde{\bW}|^2 - |B\widetilde{\bW}|^2.
\]
Hence, since $B$ is a constant matrix
\[
\int_\Omega e^{2\lambda\mu_\beta(\x)}
\big|\Delta \widetilde{\bW}-B\widetilde{\bW}\big|^2\,d\x
\ge
\frac12
\int_\Omega e^{2\lambda\mu_\beta(\x)}
|\Delta \widetilde{\bW}|^2\,d\x
-
C
\int_\Omega e^{2\lambda\mu_\beta(\x)}
|\widetilde{\bW}|^2\,d\x.
\]

Since each component of $\widetilde{\bW}$ and $\widetilde{\bR}$ satisfies homogeneous Dirichlet and Neumann boundary conditions, Lemma~\ref{lem:Carleman} applies componentwise and yields
\[
\int_\Omega e^{2\lambda\mu_\beta(\x)}
|\Delta \widetilde{\bW}|^2\,d\x
\ge
C\lambda
\int_\Omega e^{2\lambda\mu_\beta(\x)}
|\nabla \widetilde{\bW}|^2\,d\x
+
C\lambda^3
\int_\Omega e^{2\lambda\mu_\beta(\x)}
|\widetilde{\bW}|^2\,d\x
\]
and
\[
\int_\Omega e^{2\lambda\mu_\beta(\x)}
|\Delta \widetilde{\bR}|^2\,d\x
\ge
C\lambda
\int_\Omega e^{2\lambda\mu_\beta(\x)}
|\nabla \widetilde{\bR}|^2\,d\x
+
C\lambda^3
\int_\Omega e^{2\lambda\mu_\beta(\x)}
|\widetilde{\bR}|^2\,d\x.
\]
Therefore, after enlarging $\lambda_0$ if necessary, such that for all $\lambda\ge\lambda_0$,
\begin{multline}\label{eq:lhs_estimate_contract_clean}
\int_\Omega e^{2\lambda\mu_\beta(\x)}
\big|\Delta \widetilde{\bW}-B\widetilde{\bW}\big|^2\,d\x
+
\int_\Omega e^{2\lambda\mu_\beta(\x)}
\big|\Delta \widetilde{\bR}\big|^2\,d\x \\
\ge
C\lambda
\int_\Omega e^{2\lambda\mu_\beta(\x)}
\Big(
|\widetilde{\bW}(\x)|^2
+
|\nabla\widetilde{\bW}(\x)|^2
+
|\widetilde{\bR}(\x)|^2
+
|\nabla\widetilde{\bR}(\x)|^2
\Big)\,d\x.
\end{multline}

Combining \eqref{eq:basic_contract_est_2_clean}, \eqref{eq:rhs_estimate_contract_clean}, and \eqref{eq:lhs_estimate_contract_clean}, we obtain
\begin{multline}\label{eq:before_cstar}
\lambda
\int_\Omega e^{2\lambda\mu_\beta(\x)}
\Big(
|\widetilde{\bW}(\x)|^2
+
|\nabla\widetilde{\bW}(\x)|^2
+
|\widetilde{\bR}(\x)|^2
+
|\nabla\widetilde{\bR}(\x)|^2
\Big)\,d\x
+
\epsilon \|(\widetilde{\bW},\widetilde{\bR})\|_H^2 \\
\le
C
\int_\Omega e^{2\lambda\mu_\beta(\x)}
\Big(
|\widetilde{\bV}(\x)|^2
+
|\nabla\widetilde{\bV}(\x)|^2
+
|\widetilde{\bQ}(\x)|^2
+
|\nabla\widetilde{\bQ}(\x)|^2
\Big)\,d\x.
\end{multline}
Dividing by $\lambda$ and using \eqref{norm}, we obtain
\[
\|(\widetilde{\bW},\widetilde{\bR})\|_{\lambda,\epsilon}^2
\le
\frac{C}{\lambda}
\int_\Omega e^{2\lambda\mu_\beta(\x)}
\Big(
|\widetilde{\bV}(\x)|^2
+
|\nabla\widetilde{\bV}(\x)|^2
+
|\widetilde{\bQ}(\x)|^2
+
|\nabla\widetilde{\bQ}(\x)|^2
\Big)\,d\x.
\]
Since the integral on the right-hand side is bounded above by $\|(\widetilde{\bV},\widetilde{\bQ})\|_{\lambda,\epsilon}^2$, we find that
\[
\|(\widetilde{\bW},\widetilde{\bR})\|_{\lambda,\epsilon}^2
\le
\frac{C}{\lambda}
\|(\widetilde{\bV},\widetilde{\bQ})\|_{\lambda,\epsilon}^2.
\]
Therefore,
\[
\|\mathcal T_{\lambda,\epsilon}(\bV_1,\bQ_1)-\mathcal T_{\lambda,\epsilon}(\bV_2,\bQ_2)\|_{\lambda,\epsilon}
\le
\sqrt{\frac{C}{\lambda}}\,
\|(\bV_1,\bQ_1)-(\bV_2,\bQ_2)\|_{\lambda,\epsilon}.
\]
This proves \eqref{eq:contractive_estimate}. 
In particular, if $\lambda$ is sufficiently large so that $C/\sqrt{\lambda}<1$, then $\mathcal T_{\lambda,\epsilon}$ is a contraction on $\mathcal A$.
\end{proof}

An immediate consequence of Theorem~\ref{thm:contractive} is the existence and uniqueness of a fixed point of $\mathcal T_{\lambda,\epsilon}$.

\begin{Corollary}[Existence and uniqueness of the fixed point]\label{thm:fixed_point}
Assume that $\lambda$ is chosen sufficiently large so that $\mathcal T_{\lambda,\epsilon}$ is a contraction on $\mathcal A$ with respect to the norm $\|\cdot\|_{\lambda,\epsilon}$. Then $\mathcal T_{\lambda,\epsilon}$ admits a unique fixed point $(\bU^\ast,\bP^\ast)\in\mathcal A$, that is,
\[
\mathcal T_{\lambda,\epsilon}(\bU^\ast,\bP^\ast)=(\bU^\ast,\bP^\ast).
\]
Moreover, for any initial guess $(\bU^{(0)},\bP^{(0)})\in\mathcal A$, the sequence
\[
(\bU^{(n+1)},\bP^{(n+1)})=\mathcal T_{\lambda,\epsilon}(\bU^{(n)},\bP^{(n)}),
\qquad n\ge 0,
\]
converges in the norm $\|\cdot\|_{\lambda,\epsilon}$ to $(\bU^\ast,\bP^\ast)$ as $n\to\infty$.
\end{Corollary}

\begin{proof}
Since $\epsilon>0$, the norm $\|\cdot\|_{\lambda,\epsilon}$ is equivalent to the norm of $H$. Indeed, for every $(\bV,\bQ)\in H$,
\[
\frac{\epsilon}{\lambda}\|(\bV,\bQ)\|_H^2
\le
\|(\bV,\bQ)\|_{\lambda,\epsilon}^2
\le
C_{\lambda,\epsilon}\|(\bV,\bQ)\|_H^2
\]
for some constant $C_{\lambda,\epsilon}>0$. Hence, because $\mathcal A$ is closed in $H$, the metric space $(\mathcal A,\|\cdot\|_{\lambda,\epsilon})$ is complete. The conclusion therefore follows from Banach's fixed-point theorem and Theorem~\ref{thm:contractive}.
\end{proof}

We next compare this fixed point with a solution of the reduced system \eqref{4.4}.

\begin{Theorem}[Consistency of the fixed point]\label{thm:consistency}
Fix $\beta\ge \beta_0$, and let $\lambda_0$ be as in Lemma~\ref{lem:Carleman}. Assume that the reduced system \eqref{4.4} has a solution
\[
(\bU^\dagger,P^\dagger)\in \mathcal A.
\]
Then there exist $\lambda_1\ge \lambda_0$ and a constant $C>0$, depending only on $\x_0$, $\Omega$, $\beta$, $d$, $\mu$, $N$, $S_N$, $R_N$, $\{c_{mkn}\}_{m,k,n=0}^N$, and $M$, such that for every $\epsilon>0$ and every $\lambda\ge \lambda_1$, the operator $\mathcal T_{\lambda,\epsilon}$ admits a unique fixed point
\[
(\bU_{\lambda,\epsilon},P_{\lambda,\epsilon})\in\mathcal A
\]
and
\begin{equation}\label{eq:consistency_norm_estimate}
\|(\bU_{\lambda,\epsilon}-\bU^\dagger,\;P_{\lambda,\epsilon}-P^\dagger)\|_{\lambda,\epsilon}^2
\le
\frac{C\epsilon}{\lambda}\|(\bU^\dagger,P^\dagger)\|_H^2.
\end{equation}
In particular, for each fixed $\lambda\ge \lambda_1$,
\[
\|(\bU_{\lambda,\epsilon}-\bU^\dagger,\;P_{\lambda,\epsilon}-P^\dagger)\|_{\lambda,\epsilon}\to 0
\qquad \text{as } \epsilon\to 0.
\]
\end{Theorem}

\begin{proof}
For brevity, define
\[
\mathcal F(\x,\bV,\bQ)
:=
(\mu S_N)^{-1}\Big[\mathbf F^N(\x,\bV,\nabla\bV)+S_N(\nabla \bQ)\Big],
\qquad
\mathcal G(\x,\bV)
:=
S_N^{-1}G^N(\x,\bV,\nabla\bV),
\]
and recall that
\[
B:=(\mu S_N)^{-1}R_N.
\]
Then \eqref{4.4} can be rewritten as
\[
\Delta \bU^N-B\bU^N=\mathcal F(\x,\bU^N,P^N),
\qquad
\Delta P^N=\mathcal G(\x,\bU^N).
\]

Fix $\epsilon>0$ and $\lambda\ge \lambda_1$, where $\lambda_1$ is chosen large enough so that $\mathcal T_{\lambda,\epsilon}$ is contractive and the absorption argument below is valid. Let $(\bU_{\lambda,\epsilon},P_{\lambda,\epsilon})$ be the unique fixed point of $\mathcal T_{\lambda,\epsilon}$.

Since $(\bU_{\lambda,\epsilon},P_{\lambda,\epsilon})$ is a fixed point, it is the unique minimizer over $\mathcal A$ of the functional
\begin{multline*}
(\bW,\bR)\mapsto
\int_\Omega e^{2\lambda\mu_\beta(\x)}
\big|\Delta \bW-B\bW-\mathcal F(\x,\bU_{\lambda,\epsilon},P_{\lambda,\epsilon})\big|^2\,d\x \\
+
\int_\Omega e^{2\lambda\mu_\beta(\x)}
\big|\Delta \bR-\mathcal G(\x,\bU_{\lambda,\epsilon})\big|^2\,d\x
+
\epsilon \|(\bW,\bR)\|_H^2.
\end{multline*}
Therefore, for every $(\bZ,\bS)\in\mathcal A$, the corresponding variational inequality gives
\begin{multline}\label{eq:VI_consistency}
\int_\Omega e^{2\lambda\mu_\beta(\x)}
\big(\Delta \bU_{\lambda,\epsilon}-B\bU_{\lambda,\epsilon}-\mathcal F(\x,\bU_{\lambda,\epsilon},P_{\lambda,\epsilon})\big)\cdot
\big(\Delta(\bZ-\bU_{\lambda,\epsilon})-B(\bZ-\bU_{\lambda,\epsilon})\big)\,d\x \\
+
\int_\Omega e^{2\lambda\mu_\beta(\x)}
\big(\Delta P_{\lambda,\epsilon}-\mathcal G(\x,\bU_{\lambda,\epsilon})\big)\cdot
\Delta(\bS-P_{\lambda,\epsilon})\,d\x
+
\epsilon\big\langle (\bU_{\lambda,\epsilon},P_{\lambda,\epsilon}),(\bZ,\bS)-(\bU_{\lambda,\epsilon},P_{\lambda,\epsilon})\big\rangle_H
\ge 0.
\end{multline}

We now choose $(\bZ,\bS)=(\bU^\dagger,P^\dagger)\in\mathcal A$ and define
\[
\widetilde{\bU}:=\bU_{\lambda,\epsilon}-\bU^\dagger,
\qquad
\widetilde{P}:=P_{\lambda,\epsilon}-P^\dagger.
\]
Since both $(\bU_{\lambda,\epsilon},P_{\lambda,\epsilon})$ and $(\bU^\dagger,P^\dagger)$ belong to $\mathcal A$, they satisfy the same Cauchy data on $\partial\Omega$. Hence,
\[
\widetilde{\bU}|_{\partial\Omega}
=
\partial_\nu \widetilde{\bU}|_{\partial\Omega}
=
\widetilde{P}|_{\partial\Omega}
=
\partial_\nu \widetilde{P}|_{\partial\Omega}
=0.
\]
Using $(\bZ,\bS)=(\bU^\dagger,P^\dagger)$ in \eqref{eq:VI_consistency}, and recalling that
\[
\bU^\dagger-\bU_{\lambda,\epsilon}=-\widetilde{\bU},
\qquad
P^\dagger-P_{\lambda,\epsilon}=-\widetilde{P},
\]
we obtain
\begin{multline}\label{eq:VI_true_solution}
\int_\Omega e^{2\lambda\mu_\beta(\x)}
\big(\Delta \bU_{\lambda,\epsilon}-B\bU_{\lambda,\epsilon}-\mathcal F(\x,\bU_{\lambda,\epsilon},P_{\lambda,\epsilon})\big)\cdot
\big(\Delta\widetilde{\bU}-B\widetilde{\bU}\big)\,d\x \\
+
\int_\Omega e^{2\lambda\mu_\beta(\x)}
\big(\Delta P_{\lambda,\epsilon}-\mathcal G(\x,\bU_{\lambda,\epsilon})\big)\cdot
\Delta\widetilde{P}\,d\x
+
\epsilon\big\langle (\bU_{\lambda,\epsilon},P_{\lambda,\epsilon}),(\widetilde{\bU},\widetilde{P})\big\rangle_H
\le 0.
\end{multline}

Since $(\bU^\dagger,P^\dagger)$ solves \eqref{4.4}, we have
\[
\Delta \bU^\dagger-B\bU^\dagger=\mathcal F(\x,\bU^\dagger,P^\dagger),
\qquad
\Delta P^\dagger=\mathcal G(\x,\bU^\dagger).
\]
Subtracting these identities from the corresponding expressions for $(\bU_{\lambda,\epsilon},P_{\lambda,\epsilon})$, we get
\[
\Delta \widetilde{\bU}-B\widetilde{\bU}
=
\mathcal F(\x,\bU_{\lambda,\epsilon},P_{\lambda,\epsilon})
-
\mathcal F(\x,\bU^\dagger,P^\dagger)
+
\mathbf r_{\bU},
\]
and
\[
\Delta \widetilde{P}
=
\mathcal G(\x,\bU_{\lambda,\epsilon})
-
\mathcal G(\x,\bU^\dagger)
+
r_P,
\]
where
\[
\mathbf r_{\bU}
:=
\Delta \bU_{\lambda,\epsilon}-B\bU_{\lambda,\epsilon}-\mathcal F(\x,\bU_{\lambda,\epsilon},P_{\lambda,\epsilon}),
\qquad
r_P
:=
\Delta P_{\lambda,\epsilon}-\mathcal G(\x,\bU_{\lambda,\epsilon}).
\]
Substituting these two identities into \eqref{eq:VI_true_solution}, and using
\[
\big\langle (\bU_{\lambda,\epsilon},P_{\lambda,\epsilon}),(\widetilde{\bU},\widetilde{P})\big\rangle_H
=
\|(\widetilde{\bU},\widetilde{P})\|_H^2
+
\big\langle (\bU^\dagger,P^\dagger),(\widetilde{\bU},\widetilde{P})\big\rangle_H,
\]
we obtain
\begin{multline*}
\int_\Omega e^{2\lambda\mu_\beta(\x)}|\mathbf r_{\bU}|^2\,d\x
+
\int_\Omega e^{2\lambda\mu_\beta(\x)}|r_P|^2\,d\x
+
\epsilon\|(\widetilde{\bU},\widetilde{P})\|_H^2 \\
\le
-\int_\Omega e^{2\lambda\mu_\beta(\x)}
\mathbf r_{\bU}\cdot
\Big(
\mathcal F(\x,\bU_{\lambda,\epsilon},P_{\lambda,\epsilon})
-
\mathcal F(\x,\bU^\dagger,P^\dagger)
\Big)\,d\x \\
-
\int_\Omega e^{2\lambda\mu_\beta(\x)}
r_P\cdot
\Big(
\mathcal G(\x,\bU_{\lambda,\epsilon})
-
\mathcal G(\x,\bU^\dagger)
\Big)\,d\x
-
\epsilon\big\langle (\bU^\dagger,P^\dagger),(\widetilde{\bU},\widetilde{P})\big\rangle_H.
\end{multline*}
Using $2ab\le \frac12 a^2+2b^2$, we infer that
\begin{multline}\label{eq:residual_estimate_consistency}
\int_\Omega e^{2\lambda\mu_\beta(\x)}|\mathbf r_{\bU}|^2\,d\x
+
\int_\Omega e^{2\lambda\mu_\beta(\x)}|r_P|^2\,d\x
+
\epsilon\|(\widetilde{\bU},\widetilde{P})\|_H^2 \\
\le
C
\int_\Omega e^{2\lambda\mu_\beta(\x)}
\Big|
\mathcal F(\x,\bU_{\lambda,\epsilon},P_{\lambda,\epsilon})
-
\mathcal F(\x,\bU^\dagger,P^\dagger)
\Big|^2\,d\x \\
+
C
\int_\Omega e^{2\lambda\mu_\beta(\x)}
\Big|
\mathcal G(\x,\bU_{\lambda,\epsilon})
-
\mathcal G(\x,\bU^\dagger)
\Big|^2\,d\x
+
C\epsilon \|(\bU^\dagger,P^\dagger)\|_H^2.
\end{multline}

Because the iteration remains in $\mathcal A$, and because $s>\frac d2+2$ implies $H^s(\Omega)\hookrightarrow C^2(\overline\Omega)$, all admissible pairs are uniformly bounded together with their first derivatives. Hence, the mappings $\mathcal F$ and $\mathcal G$ are Lipschitz continuous on $\mathcal A$. Therefore, there exists a constant $C>0$, depending only on the structural parameters and $M$, such that
\begin{multline}\label{eq:lipschitz_consistency_clean}
\int_\Omega e^{2\lambda\mu_\beta(\x)}
\Big|
\mathcal F(\x,\bU_{\lambda,\epsilon},P_{\lambda,\epsilon})
-
\mathcal F(\x,\bU^\dagger,P^\dagger)
\Big|^2\,d\x \\
+
\int_\Omega e^{2\lambda\mu_\beta(\x)}
\Big|
\mathcal G(\x,\bU_{\lambda,\epsilon})
-
\mathcal G(\x,\bU^\dagger)
\Big|^2\,d\x
\\
\le
C
\int_\Omega e^{2\lambda\mu_\beta(\x)}
\Big(
|\widetilde{\bU}(\x)|^2
+
|\nabla\widetilde{\bU}(\x)|^2
+
|\widetilde{P}(\x)|^2
+
|\nabla\widetilde{P}(\x)|^2
\Big)\,d\x.
\end{multline}

Combining \eqref{eq:residual_estimate_consistency} and \eqref{eq:lipschitz_consistency_clean}, we get
\begin{multline}\label{eq:residual_estimate_consistency_2}
\int_\Omega e^{2\lambda\mu_\beta(\x)}|\mathbf r_{\bU}|^2\,d\x
+
\int_\Omega e^{2\lambda\mu_\beta(\x)}|r_P|^2\,d\x
+
\epsilon\|(\widetilde{\bU},\widetilde{P})\|_H^2 \\
\le
C
\int_\Omega e^{2\lambda\mu_\beta(\x)}
\Big(
|\widetilde{\bU}(\x)|^2
+
|\nabla\widetilde{\bU}(\x)|^2
+
|\widetilde{P}(\x)|^2
+
|\nabla\widetilde{P}(\x)|^2
\Big)\,d\x
+
C\epsilon \|(\bU^\dagger,P^\dagger)\|_H^2.
\end{multline}

On the other hand, from the definitions of $\mathbf r_{\bU}$ and $r_P$, we have
\[
\Delta \widetilde{\bU}-B\widetilde{\bU}
=
\mathbf r_{\bU}
+
\mathcal F(\x,\bU_{\lambda,\epsilon},P_{\lambda,\epsilon})
-
\mathcal F(\x,\bU^\dagger,P^\dagger),
\]
and
\[
\Delta \widetilde{P}
=
r_P
+
\mathcal G(\x,\bU_{\lambda,\epsilon})
-
\mathcal G(\x,\bU^\dagger).
\]
Therefore,
\begin{multline*}
\int_\Omega e^{2\lambda\mu_\beta(\x)}
\big|\Delta \widetilde{\bU}-B\widetilde{\bU}\big|^2\,d\x
+
\int_\Omega e^{2\lambda\mu_\beta(\x)}
|\Delta \widetilde{P}|^2\,d\x \\
\le
C
\int_\Omega e^{2\lambda\mu_\beta(\x)}|\mathbf r_{\bU}|^2\,d\x
+
C
\int_\Omega e^{2\lambda\mu_\beta(\x)}|r_P|^2\,d\x \\
+
C
\int_\Omega e^{2\lambda\mu_\beta(\x)}
\Big|
\mathcal F(\x,\bU_{\lambda,\epsilon},P_{\lambda,\epsilon})
-
\mathcal F(\x,\bU^\dagger,P^\dagger)
\Big|^2\,d\x \\
+
C
\int_\Omega e^{2\lambda\mu_\beta(\x)}
\Big|
\mathcal G(\x,\bU_{\lambda,\epsilon})
-
\mathcal G(\x,\bU^\dagger)
\Big|^2\,d\x.
\end{multline*}
Using \eqref{eq:lipschitz_consistency_clean} and \eqref{eq:residual_estimate_consistency_2}, we conclude that
\begin{multline}\label{eq:laplace_bound_consistency}
\int_\Omega e^{2\lambda\mu_\beta(\x)}
\big|\Delta \widetilde{\bU}-B\widetilde{\bU}\big|^2\,d\x
+
\int_\Omega e^{2\lambda\mu_\beta(\x)}
|\Delta \widetilde{P}|^2\,d\x \\
\le
C
\int_\Omega e^{2\lambda\mu_\beta(\x)}
\Big(
|\widetilde{\bU}(\x)|^2
+
|\nabla\widetilde{\bU}(\x)|^2
+
|\widetilde{P}(\x)|^2
+
|\nabla\widetilde{P}(\x)|^2
\Big)\,d\x
+
C\epsilon \|(\bU^\dagger,P^\dagger)\|_H^2.
\end{multline}

Since $B$ is a constant matrix, the same argument as in the proof of Theorem~\ref{thm:contractive}, together with Lemma~\ref{lem:Carleman} applied componentwise to $\widetilde{\bU}$ and $\widetilde{P}$, yields
\begin{multline}\label{eq:carleman_consistency_clean}
C_1\lambda
\int_\Omega e^{2\lambda\mu_\beta(\x)}
\Big(
|\widetilde{\bU}(\x)|^2
+
|\nabla\widetilde{\bU}(\x)|^2
+
|\widetilde{P}(\x)|^2
+
|\nabla\widetilde{P}(\x)|^2
\Big)\,d\x \\
\le
\int_\Omega e^{2\lambda\mu_\beta(\x)}
\big|\Delta \widetilde{\bU}-B\widetilde{\bU}\big|^2\,d\x
+
\int_\Omega e^{2\lambda\mu_\beta(\x)}
|\Delta \widetilde{P}|^2\,d\x
\end{multline}
for all $\lambda\ge \lambda_1$, after enlarging $\lambda_1$ if necessary.

Combining \eqref{eq:laplace_bound_consistency} and \eqref{eq:carleman_consistency_clean}, and taking $\lambda_1$ sufficiently large so that the weighted $H^1$-term on the right-hand side can be absorbed into the left-hand side, we obtain
\begin{equation}\label{eq:weighted_part_bound}
\int_\Omega e^{2\lambda\mu_\beta(\x)}
\Big(
|\widetilde{\bU}(\x)|^2
+
|\nabla\widetilde{\bU}(\x)|^2
+
|\widetilde{P}(\x)|^2
+
|\nabla\widetilde{P}(\x)|^2
\Big)\,d\x
\le
\frac{C\epsilon}{\lambda}\|(\bU^\dagger,P^\dagger)\|_H^2.
\end{equation}

Dividing \eqref{eq:residual_estimate_consistency_2} by $\lambda$ and using \eqref{eq:weighted_part_bound}, we further obtain
\[
\frac{\epsilon}{\lambda}\|(\widetilde{\bU},\widetilde{P})\|_H^2
\le
\frac{C}{\lambda}
\int_\Omega e^{2\lambda\mu_\beta(\x)}
\Big(
|\widetilde{\bU}(\x)|^2
+
|\nabla\widetilde{\bU}(\x)|^2
+
|\widetilde{P}(\x)|^2
+
|\nabla\widetilde{P}(\x)|^2
\Big)\,d\x
+
\frac{C\epsilon}{\lambda}\|(\bU^\dagger,P^\dagger)\|_H^2,
\]
and hence
\begin{equation}\label{eq:H_part_bound}
\frac{\epsilon}{\lambda}\|(\widetilde{\bU},\widetilde{P})\|_H^2
\le
\frac{C\epsilon}{\lambda}\|(\bU^\dagger,P^\dagger)\|_H^2.
\end{equation}

Finally, combining \eqref{eq:weighted_part_bound} and \eqref{eq:H_part_bound}, and recalling the definition \eqref{norm} of $\|\cdot\|_{\lambda,\epsilon}$, we conclude that
\[
\|(\bU_{\lambda,\epsilon}-\bU^\dagger,\;P_{\lambda,\epsilon}-P^\dagger)\|_{\lambda,\epsilon}^2
\le
\frac{C\epsilon}{\lambda}\|(\bU^\dagger,P^\dagger)\|_H^2.
\]
This proves \eqref{eq:consistency_norm_estimate}. The convergence statement follows immediately.
\end{proof}

This section provides the theoretical justification for the reconstruction method developed in this paper. The contractive property of $\mathcal T_{\lambda,\epsilon}$ and the consistency of its fixed point show that the Carleman--Picard iteration is a well-founded approach for solving the reduced system. We now move to the numerical implementation of this method and to computational examples based on synthetic boundary measurements.

\section{Reconstruction, numerical implementation, and numerical examples} \label{sec:numerics}

Corollary \ref{thm:fixed_point} and Theorem \ref{thm:consistency} yield an approximation of the coefficient vectors in the time-dimensional reduction system. We next use these coefficients to reconstruct the space-time solution $(\bu,p)$ and then recover the unknown initial data in the inverse problem. More precisely, once approximations of the coefficients $\{\bu_n\}_{n=0}^N$ and $\{p_n\}_{n=0}^N$ are available, we reconstruct
\[
\bu(\x,t)\approx \sum_{n=0}^N \bu_n(\x)\Psi_n(t),
\qquad
p(\x,t)\approx \sum_{n=0}^N p_n(\x)\Psi_n(t).
\]
In particular, by evaluating these expressions at $t=0$, that is,
\[
\bu(\x,0)\approx \sum_{n=0}^N \bu_n(\x)\Psi_n(0),
\qquad
p(\x,0)\approx \sum_{n=0}^N p_n(\x)\Psi_n(0),
\]
we obtain an approximation of the initial data, and hence an approximate solution of the inverse problem under consideration. The procedure is summarized in Algorithm \ref{alg:CP_time_reduction}.

\begin{algorithm}[ht]
\caption{Carleman--Picard method for the time-dimensional reduction system and reconstruction of the initial data}
\label{alg:CP_time_reduction}
\begin{algorithmic}[1]

\State Compute the projected boundary data by \eqref{eq:Utrace_on_Gamma}, \eqref{eq:Ptrace_on_Gamma}, \eqref{eq:Udn_on_Gamma}, and \eqref{eq:Pdn_on_Gamma}.

\State Choose parameters $\beta$, $\lambda$, $\epsilon$, and a maximum number of iterations $K_{\max}$ \label{s2}.

\State Choose an initial pair $(\bU^{(0)},\bP^{(0)})\in\mathcal A$. \label{s3}

\For{$k=0,1,2,\dots,K_{\max}-1$}
    \State Compute \label{s5}
    \[
    (\bU^{(k+1)},\bP^{(k+1)})
    :=
    \mathcal T_{\lambda,\epsilon}(\bU^{(k)},\bP^{(k)}),
    \]
    i.e., let $(\bU^{(k+1)},\bP^{(k+1)})$ be the minimizer over $\mathcal A$ of the frozen least-squares functional
    \[
    J_{\lambda,\epsilon}(\bW,\bR;\bU^{(k)},\bP^{(k)}).
    \] 
\EndFor

\State Set
\[
\bU^{\rm rec}:=\bU^{(K_{\max})},
\qquad
P^{\rm rec}:=\bP^{(K_{\max})}.
\]

\State Reconstruct the space-time solution by
\[
\bu^{\rm rec}(\x,t):= \sum_{n=0}^N \bu_n(\x)\Psi_n(t),
\qquad
p^{\rm rec}(\x,t):= \sum_{n=0}^N p_n(\x)\Psi_n(t),
\]
where $\bU^{\rm rec}=(\bu_0,\dots,\bu_N)^\top$ and $P^{\rm rec}=(p_0,\dots,p_N)^\top$.

\State Recover the initial data by evaluating at $t=0$:
\[
\bu^{\rm rec}(\x,0)= \sum_{n=0}^N \bu_n(\x)\Psi_n(0),
\qquad
p^{\rm rec}(\x,0)= \sum_{n=0}^N p_n(\x)\Psi_n(0).
\]

\end{algorithmic}
\end{algorithm}

\subsection{Data generation}

To generate synthetic data for the inverse solver, we numerically solve the forward incompressible Navier--Stokes system on the square domain $\Omega = (-1,1)^2$ up to the final time $T=0.4$ for prescribed test pairs $(\bff,\bu_0)$. In all numerical examples, the initial velocity $\bu_0$ is chosen to be divergence-free, while the body force $\bff$ is time-independent and depends on the test case. For simplicity in the forward simulation, we take the viscosity coefficient to be $\mu=1$.

We discretize $\Omega$ on a uniform Cartesian grid with $N_x=N_y=31$ grid points in each direction, so that $dx=dy=\frac{2}{31-1}=\frac{1}{15}\approx 0.0667$. The time interval $(0,T)$ is discretized uniformly by $t_n=n\,\Delta t$ for $n=0,1,\dots,N_t$, with $\Delta t=10^{-4}$. This choice is used in the finite-difference forward solver, with a small time step to maintain stability because the convection term is treated explicitly.

At each time level, we first compute the pressure by solving a Poisson equation. This equation is obtained by taking the divergence of the momentum equation and using the incompressibility condition $\Div \bu=0$, which yields
\begin{equation}\label{eq:p_forward_numeric}
\Delta p^n
=
-\sum_{i,j=1}^2 \partial_{x_i}u_j^n\,\partial_{x_j}u_i^n
+\Div \bff
\qquad \text{in } \Omega.
\end{equation}
The corresponding boundary condition is obtained by taking the normal component of the momentum equation on $\partial\Omega$. Since $\bu=0$ on $\partial\Omega$, both the time derivative and the convection term vanish there, and we obtain
\begin{equation*}%\label{eq:p_forward_neumann_numeric}
\partial_\nu p^n
=
\nu\cdot(\Delta \bu^n+\bff)
\qquad \text{on } \partial\Omega.
\end{equation*}
To fix the additive constant in the pressure, we impose the normalization
\begin{equation}\label{eq:p_mean_zero_numeric}
\int_\Omega p^n(\x)\,d\x=0.
\end{equation}
The pressure equation \eqref{eq:p_forward_numeric}--\eqref{eq:p_mean_zero_numeric} is discretized by finite differences on the spatial grid.

Once $p^n$ has been computed, we update the velocity by the semi-implicit scheme
\begin{equation}\label{eq:u_forward_numeric}
\frac{\bu^{n+1}-\bu^n}{\Delta t}
=
-(\bu^n\cdot\nabla)\bu^n
-\nabla p^n
+\Delta \bu^{n+1}
+\bff
\qquad \text{in } \Omega,
\end{equation}
subject to the homogeneous Dirichlet boundary condition
\begin{equation*}%\label{eq:u_forward_numeric_bc}
\bu^{n+1}=0
\qquad \text{on } \partial\Omega.
\end{equation*}
Here, the convection term and the pressure gradient are treated explicitly, while the diffusion term is treated implicitly, so that each time step requires solving only a linear elliptic problem for $\bu^{n+1}$.

The pressure equation also enforces incompressibility at the next time step. Indeed, taking the divergence of \eqref{eq:u_forward_numeric} gives
\[
\frac{\Div \bu^{n+1}-\Div \bu^n}{\Delta t}
=
-\Div\big((\bu^n\cdot\nabla)\bu^n\big)
-\Delta p^n
+\Delta(\Div \bu^{n+1})
+\Div \bff.
\]
Assuming that $\Div \bu^n=0$ and using \eqref{eq:p_forward_numeric}, we obtain
\[
\Div \bu^{n+1}-\Delta t\,\Delta(\Div \bu^{n+1})=0.
\]
Thus, formally, $\Div \bu^{n+1}$ satisfies a homogeneous elliptic equation. With compatible boundary conditions, this yields $\Div \bu^{n+1}=0$. In the discrete computation, the velocity therefore remains divergence-free up to discretization and solver errors.

After computing the time-dependent fields $\bu(\x,t)$ and $p(\x,t)$, we extract the lateral boundary data required in Problem~\ref{ISP}, namely
\[
\bg(\x,t)=\partial_\nu \bu(\x,t),
\qquad
\bh(\x,t)=
\begin{bmatrix}
p(\x,t)\\
\partial_\nu p(\x,t)
\end{bmatrix},
\qquad (\x,t)\in \partial\Omega\times(0,T),
\]
where the normal derivatives are approximated by finite differences along the boundary.

To assess the stability of the reconstruction method, we add noise to the boundary data. Let $\delta>0$ denote the noise level. For each boundary data $\mathbf y$, we define the perturbed data by
\begin{equation*}%\label{eq:noisy_data_rule}
\mathbf y^{\rm noisy}
=
\mathbf y \Big(1+\delta(-2\,\mathrm{rand}+1)\Big),
\end{equation*}
where  $\mathrm{rand}$ is an array of independent random variables uniformly distributed in $(0,1)$. Equivalently, each component of $\mathbf y$ is multiplied by a random factor in the interval $[1-\delta,\,1+\delta]$, which corresponds to uniformly distributed relative noise of level $\delta$. In our numerical tests, $\delta = 10\%.$

\subsection{Remarks on the implementation}

In this subsection, we discuss several implementation details that are important in our numerical computations.

We begin with Step \ref{s2}, in which the parameters $\beta$, $\lambda$, $\epsilon$, and $K_{\max}$ are selected. In the present work, these parameters are chosen empirically. More precisely, we use Test~1 as a reference test and manually tune the parameters until the reconstructed solution is satisfactory. Once these parameters are chosen, we keep them fixed for all other numerical tests. In our computations, we take $N=35$, $\epsilon=10^{-11}$, and $K_{\max}=5$. The Carleman weight parameters are chosen as $\x_0=(0,-10)$, $\beta=20$, and $\lambda=6$.

We now turn to Step \ref{s3}, where the initial pair is chosen. In all numerical tests, we initialize the iteration by $(\bU^{(0)},\bP^{(0)})=(\mathbf 0,\mathbf 0)$. Although this pair does not necessarily belong to $\mathcal A$, the discrepancy is corrected automatically after one iteration, since $(\bU^{(1)},\bP^{(1)})$ is defined as the minimizer of the frozen least-squares functional over $\mathcal A$, and hence lies in $\mathcal A$. Numerically, this initialization is sufficient for all examples considered in this paper.

We now discuss the implementation in Step~\ref{s5}. We observe that although $S$ is invertible, the matrix $S^{-1}$ contains ``singular" entries with large numbers when $N = 35$ as in our choice. This might cause some unnecessary errors in computation. We therefore solve the equivalent system obtained by replacing the partial differential equations in \eqref{4.4} with \eqref{eq:time_reduction_system_final}.
The corresponding definition of $J_{\lambda,\epsilon}$ is
    \begin{multline*}
J_{\lambda,\epsilon}(\bW,\bR;\bV,\bQ)
:=
\int_\Omega e^{2\lambda\mu_\beta(\x)}
\left|
\mu S_N \Delta \bW
-\big[
R_N\bW+\mathbf F^N(\x,\bV,\nabla \bV)+S_N(\nabla \bQ)
\big]
\right|^2
\,d\x
\\
+
\int_\Omega e^{2\lambda\mu_\beta(\x)}
\left|
S_N\Delta \bR
-G^N(\x,\bV,\nabla \bV)
\right|^2
\,d\x
+
\epsilon \|(\bW,\bR)\|_H^2.
%\label{eq:J_frozen}
\end{multline*}
 Given the current iterate $(\bU^{(k)},\bP^{(k)})$, we first evaluate the nonlinear terms 
 \[\mathbf F^N(\x,\bU^{(k)},\nabla \bU^{(k)}) \quad \mbox{and } \quad G^N(\x,\bU^{(k)},\nabla \bU^{(k)}),\] 
 and then freeze them in the reduced system. This leads to two linear weighted least-squares problems: one for $\bP^{(k+1)}$ and one for $\bU^{(k+1)}$. More precisely, we first solve the least-squares problem associated with the equation for $P^N$, together with the projected Dirichlet and Neumann boundary data and the Sobolev regularization terms. After obtaining $\bP^{(k+1)}$, we evaluate the coupling term $S_N(\nabla \bP^{(k+1)})$ and then solve the second linear weighted least-squares problem for $\bU^{(k+1)}$. In the MATLAB implementation, each least-squares problem is assembled as an overdetermined linear system of the form $A\mathbf{x}\approx \mathbf{b}$, where $A$ denotes the system matrix, $\mathbf{x}$ is the vector of unknown discrete values, and $\mathbf{b}$ is the corresponding right-hand side vector. We then solve this system by the MATLAB command $\texttt{x = A{\char`\\}b}$, which returns the least-squares solution. Repeating this procedure for $k=0,1,2,\dots$ generates the sequence $(\bU^{(k)},\bP^{(k)})$, which serves as a numerical approximation of the fixed point of $\mathcal T_{\lambda,\epsilon}$.

All other steps in Algorithm~\ref{alg:CP_time_reduction}, including the projection of the boundary data and the reconstruction of the space-time solution and the initial data, are straightforward to implement once the coefficient vectors have been computed.

\subsection{Numerical examples}

In this subsection, we present some numerical tests obtained by Algorithm~\ref{alg:CP_time_reduction}.

{\bf Test 1.}
We next present the first numerical test. In this example, the body force $\bff=(f_1,f_2)$ is defined by
\[
f_1(x,y)=e^{-\frac{(x-0.05)^2+(y+0.05)^2}{0.08}}
\left[
-4y(1-y^2)(1-x^2)^2-\frac{2(y+0.05)}{0.08}(1-x^2)^2(1-y^2)^2
\right],
\]
and
\begin{equation*}
f_2(x,y)=
-e^{-\frac{(x-0.05)^2+(y+0.05)^2}{0.08}}
\left[
-4x(1-x^2)(1-y^2)^2-\frac{2(x-0.05)}{0.08}(1-x^2)^2(1-y^2)^2
\right].
\end{equation*}
The force field is normalized so that its maximum absolute component equals $1$.

The true initial velocity $\bu_0^{\rm true}(x,y)=\big(u_{0,1}^{\rm true}(x,y),u_{0,2}^{\rm true}(x,y)\big)$ is defined by
\[
u_{0,1}^{\rm true}(x,y)=0.10\,e^{-\frac{x^2+y^2}{0.12}}
\left[
-4y(1-y^2)(1-x^2)^2-\frac{2y}{0.12}(1-x^2)^2(1-y^2)^2
\right],
\]
and
\begin{equation*}
u_{0,2}^{\rm true}(x,y)=
-0.10\,e^{-\frac{x^2+y^2}{0.12}}
\left[
-4x(1-x^2)(1-y^2)^2-\frac{2x}{0.12}(1-x^2)^2(1-y^2)^2
\right].
\end{equation*}
This field is also normalized so that its maximum absolute component equals $1$.
A direct computation shows that $\bu_0^{\rm true}(x,y)=\bigl(u_{0,1}^{\rm true}(x,y),\,u_{0,2}^{\rm true}(x,y)\bigr)$ is divergence-free in $\Omega$, namely, $\nabla\cdot \bu_0^{\rm true}=0$.
The reconstruction of the initial velocity is displayed in Figure \ref{fig:case1_all}.

\begin{figure}[ht]
\centering

\subfloat[$u_{0,1}^{\rm true}$\label{fig:case1_true_u1}]{
    \includegraphics[width=0.25\textwidth]{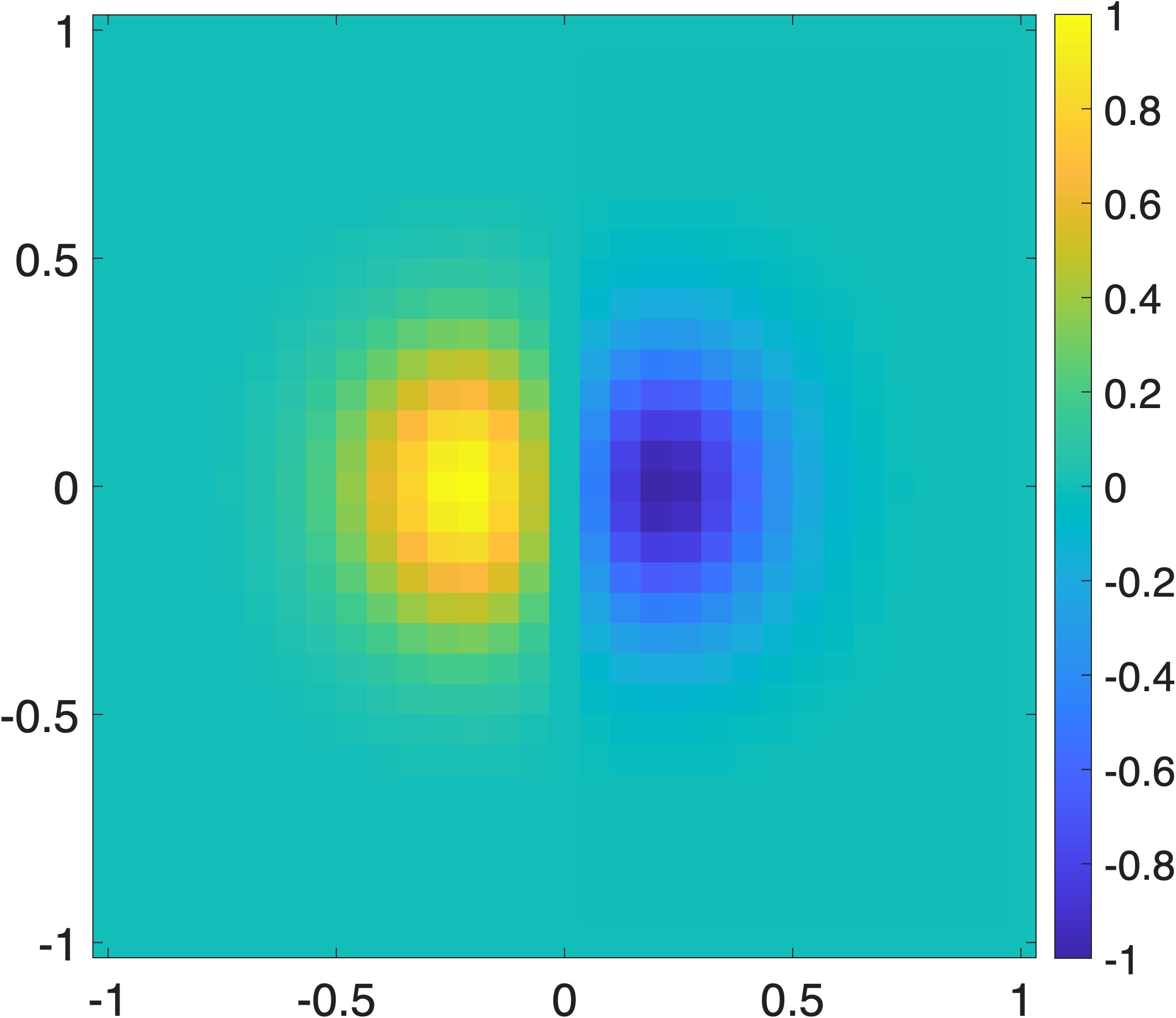}
}
\quad
\subfloat[$u_{0,2}^{\rm true}$\label{fig:case1_true_u2}]{
    \includegraphics[width=.25\textwidth]{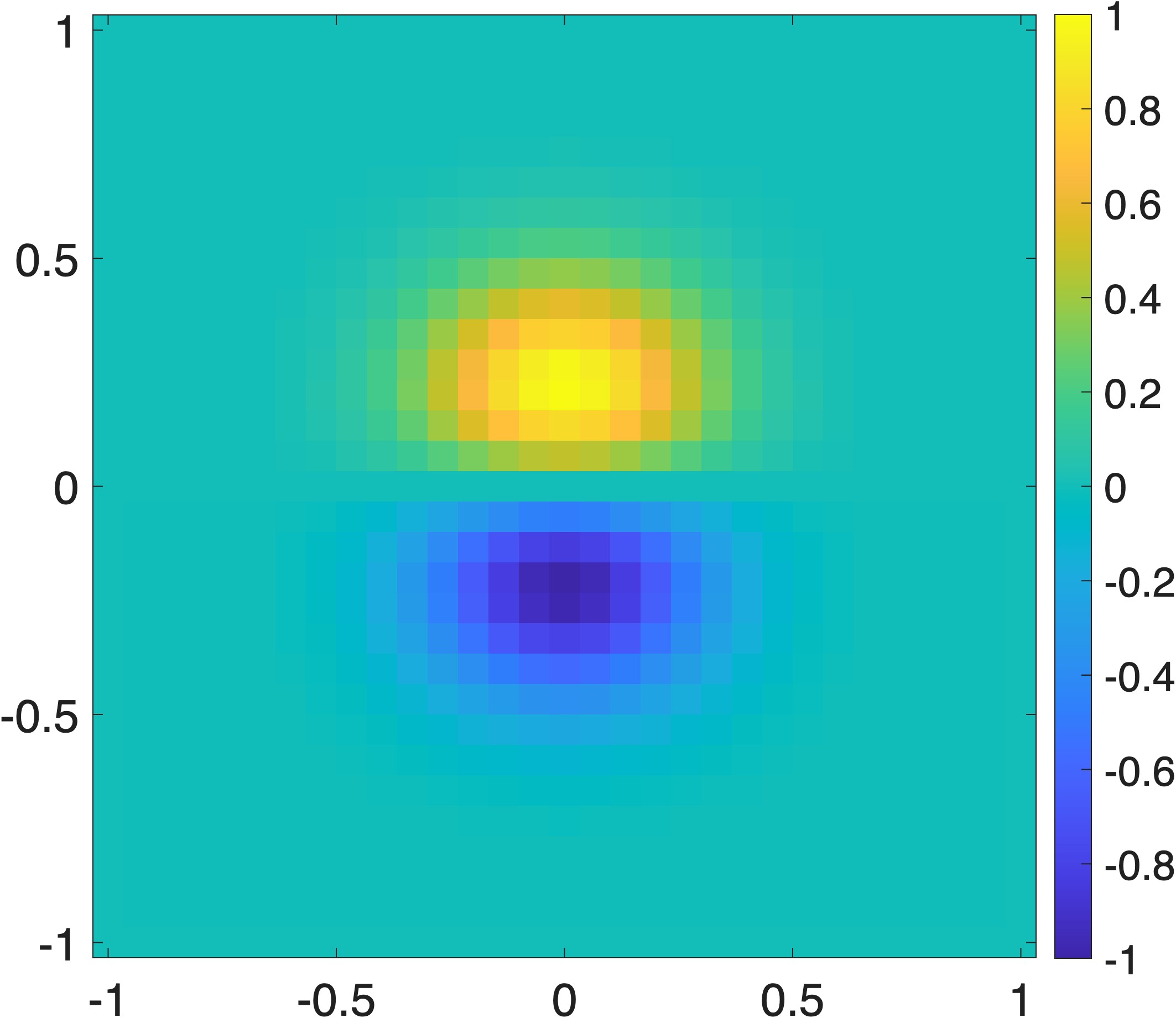}
}
\quad
\subfloat[$p_0^{\rm true}(\cdot,0)$\label{fig:case1_true_p0}]{
    \includegraphics[width=.25\textwidth]{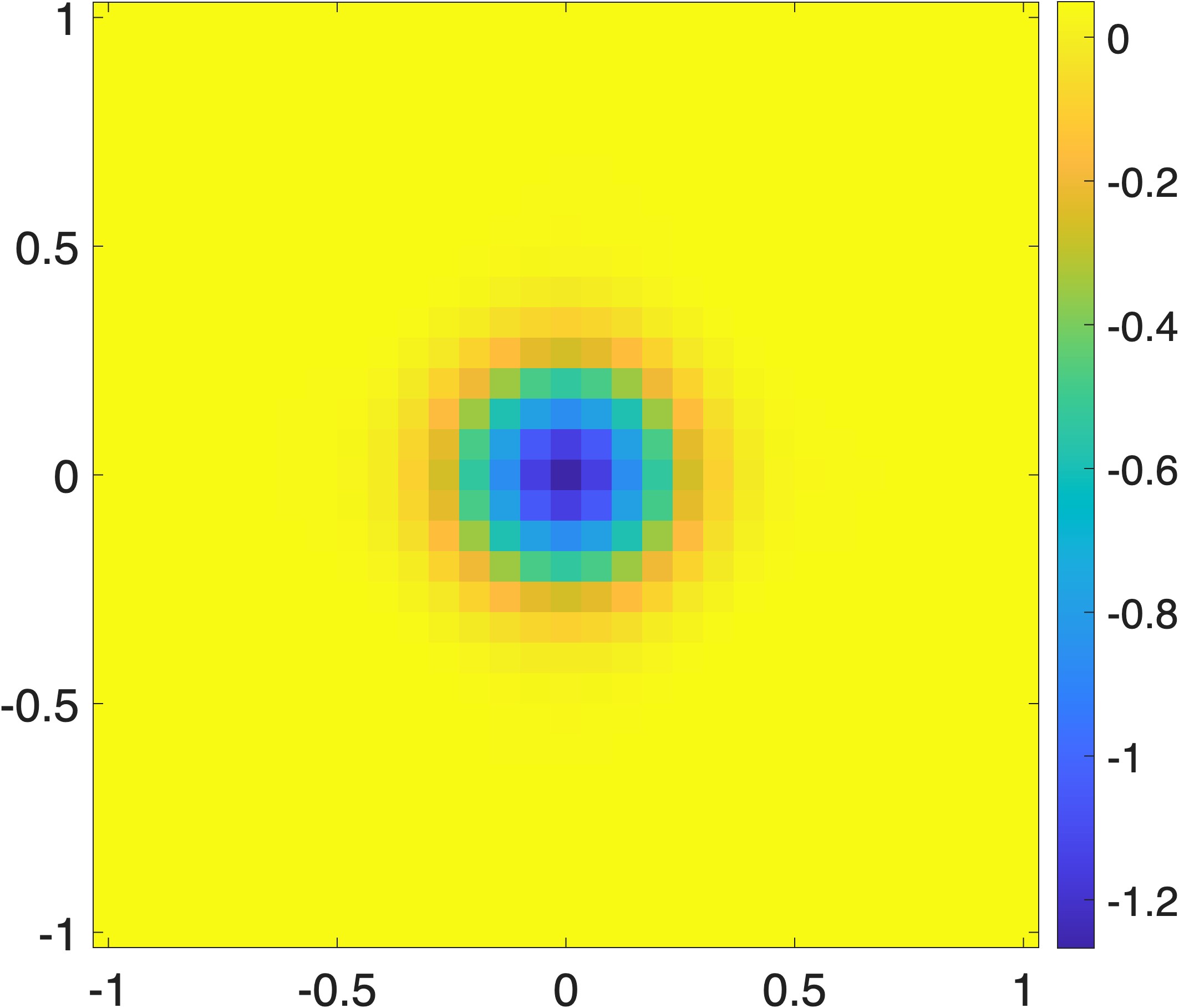}
}

\subfloat[$u_{0,1}^{\rm rec}$\label{fig:case1_rec_u1}]{
    \includegraphics[width=0.25\textwidth]{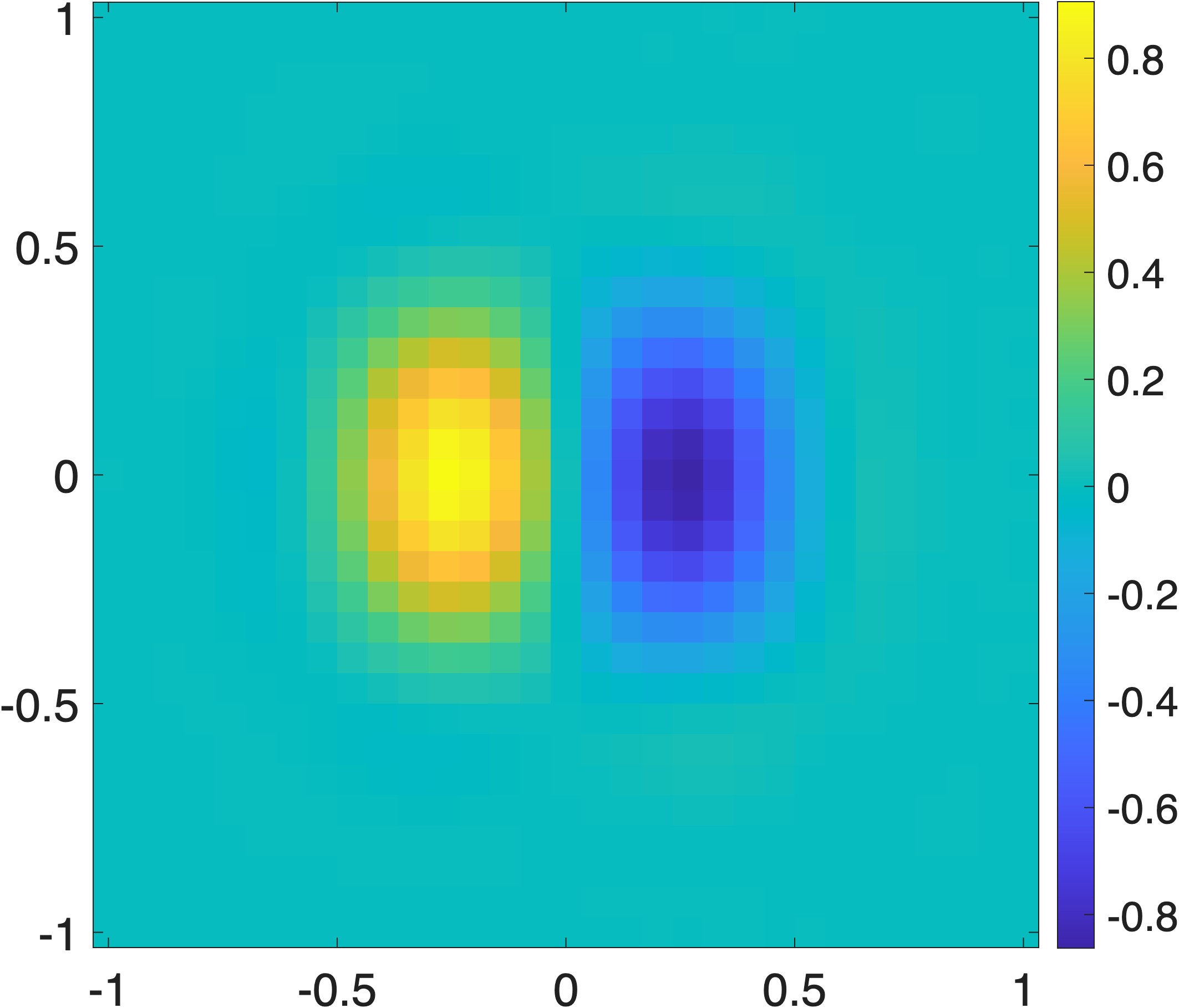}
}
\quad
\subfloat[$u_{0, 2}^{\rm rec}(\cdot,0)$\label{fig:case1_rec_u2}]{
    \includegraphics[width=0.25\textwidth]{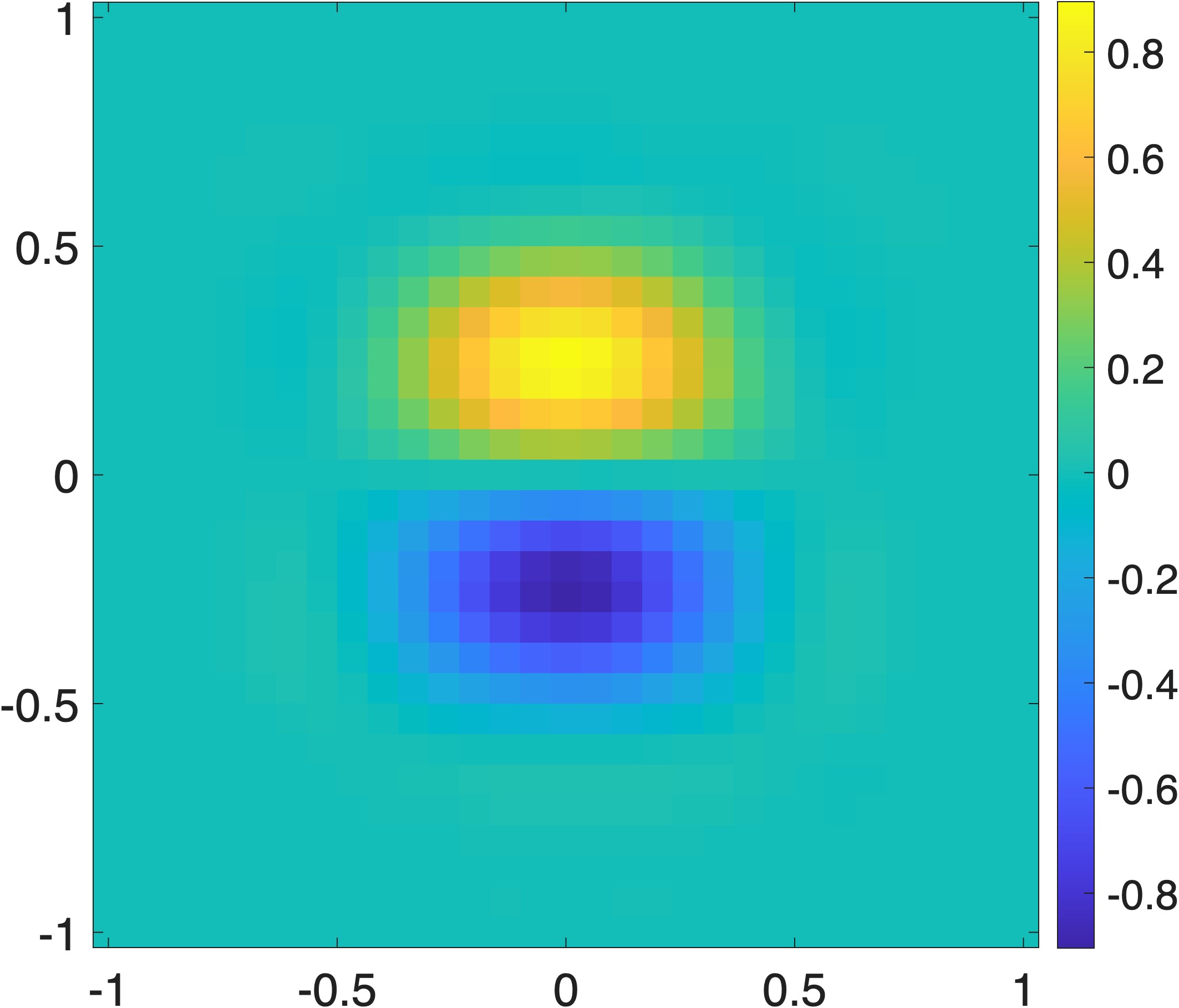}
}
\quad
\subfloat[$p_0^{\rm rec}$\label{fig:case1_rec_p0}]{
    \includegraphics[width=0.25\textwidth]{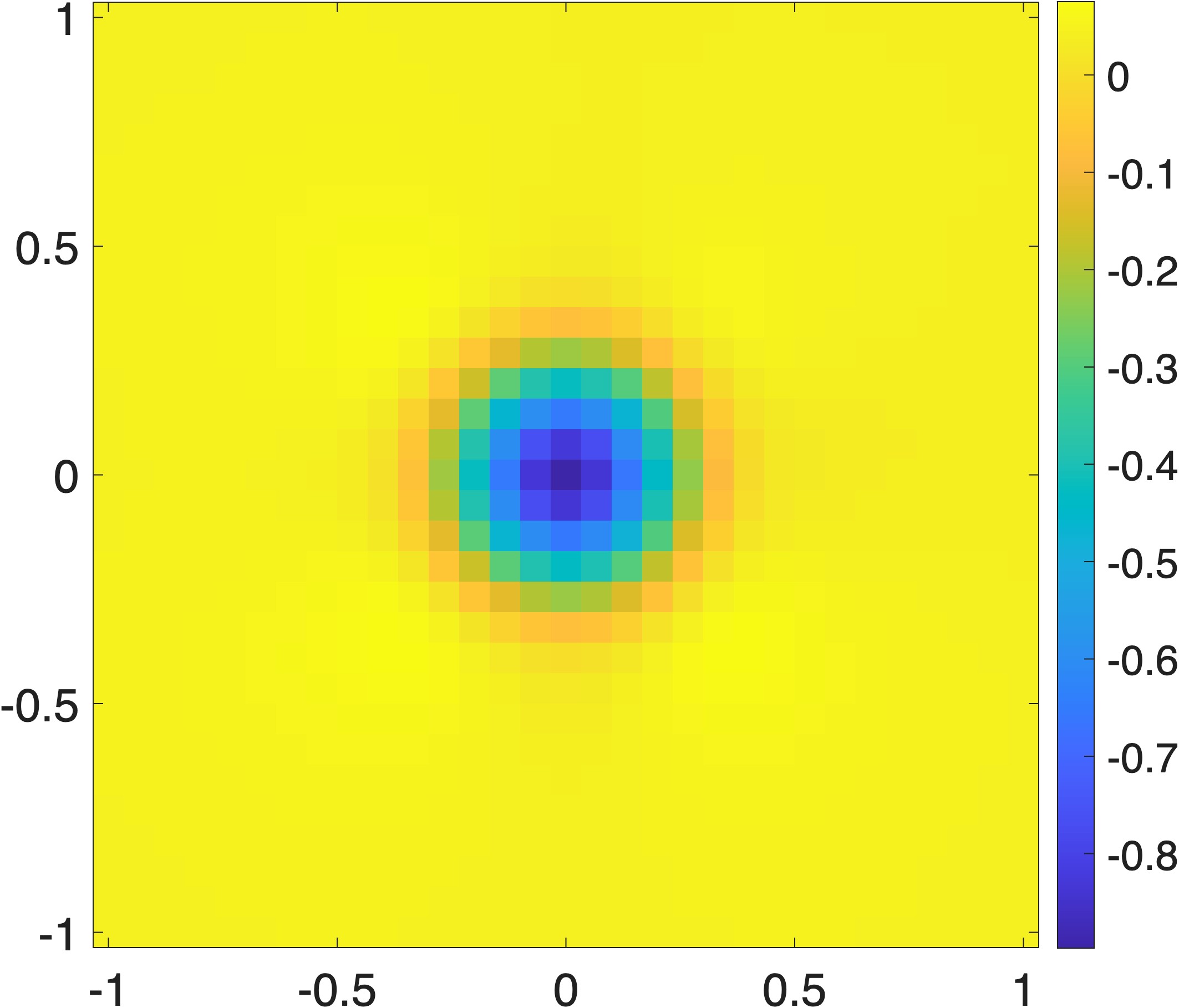}
}

\subfloat[$\frac{|u_{0,1}^{\rm true} - u_{0,1}^{\rm rec}|}{\|u_{0,1}^{\rm true}\|_{L^{\infty}(\Omega)}}$\label{fig:case1_err_u1}]{
    \includegraphics[width=0.25\textwidth]{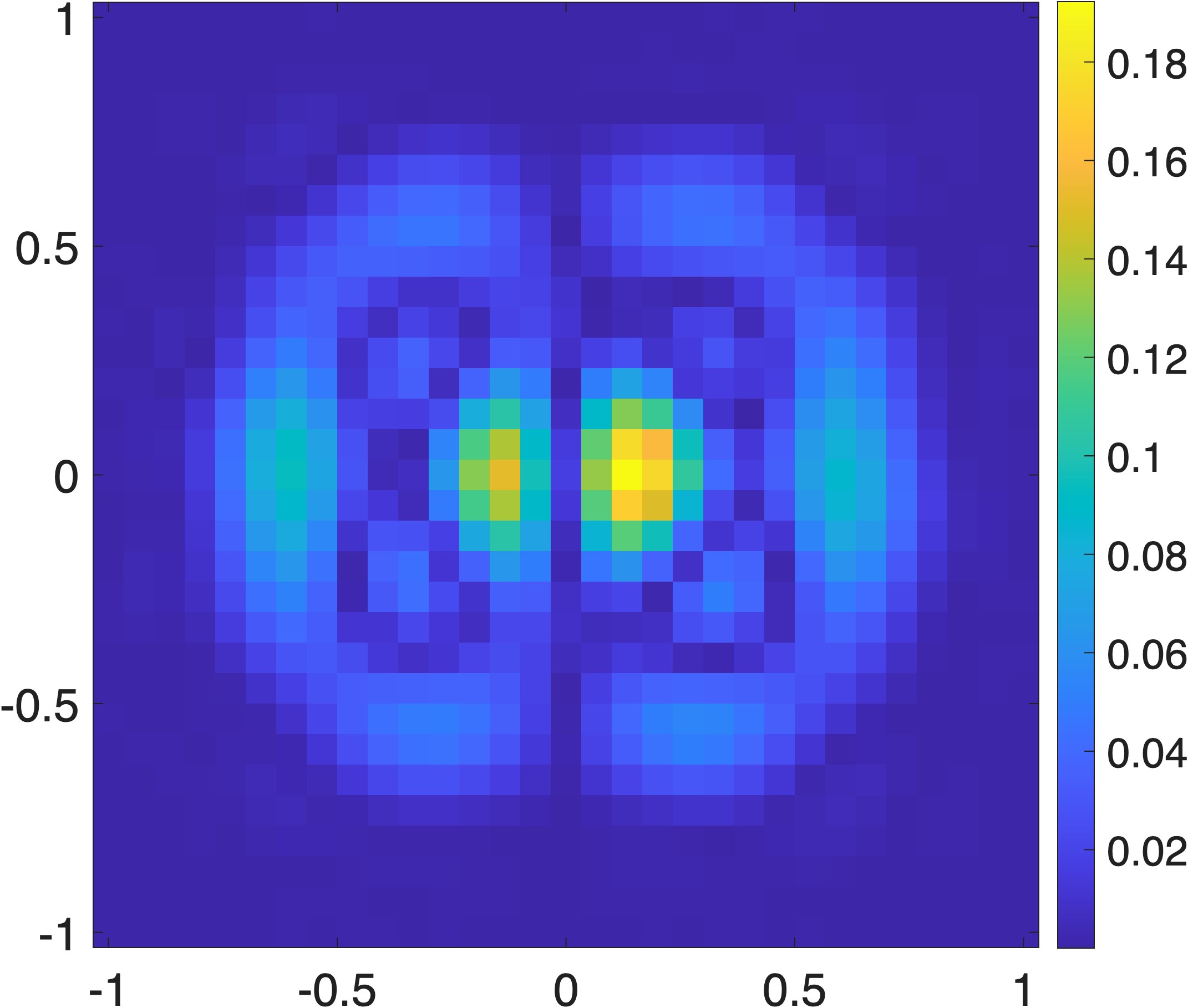}
}
\quad
\subfloat[$\frac{|u_{0,2}^{\rm true} - u_{0,2}^{\rm rec}|}{\|u_{0,2}^{\rm true}\|_{L^{\infty}(\Omega)}}$\label{fig:case1_err_u2}]{
    \includegraphics[width=0.25\textwidth]{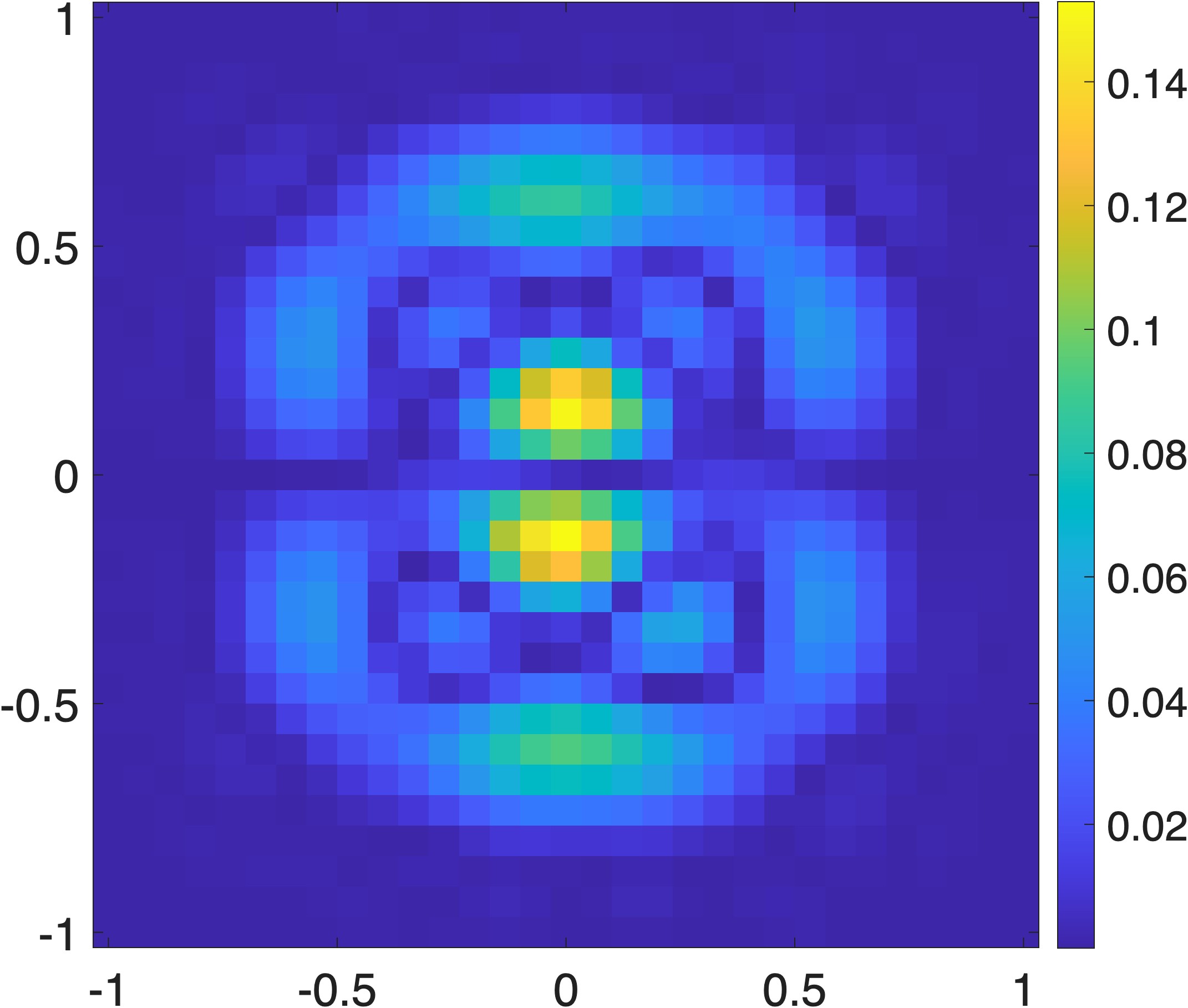}
}
\quad
\subfloat[$\frac{|p_{0}^{\rm true} - p_{0}^{\rm rec}|}{\|p_{0}^{\rm true}\|_{L^{\infty}(\Omega)}}$ \label{fig:case1_err_p0}]{
    \includegraphics[width=0.25\textwidth]{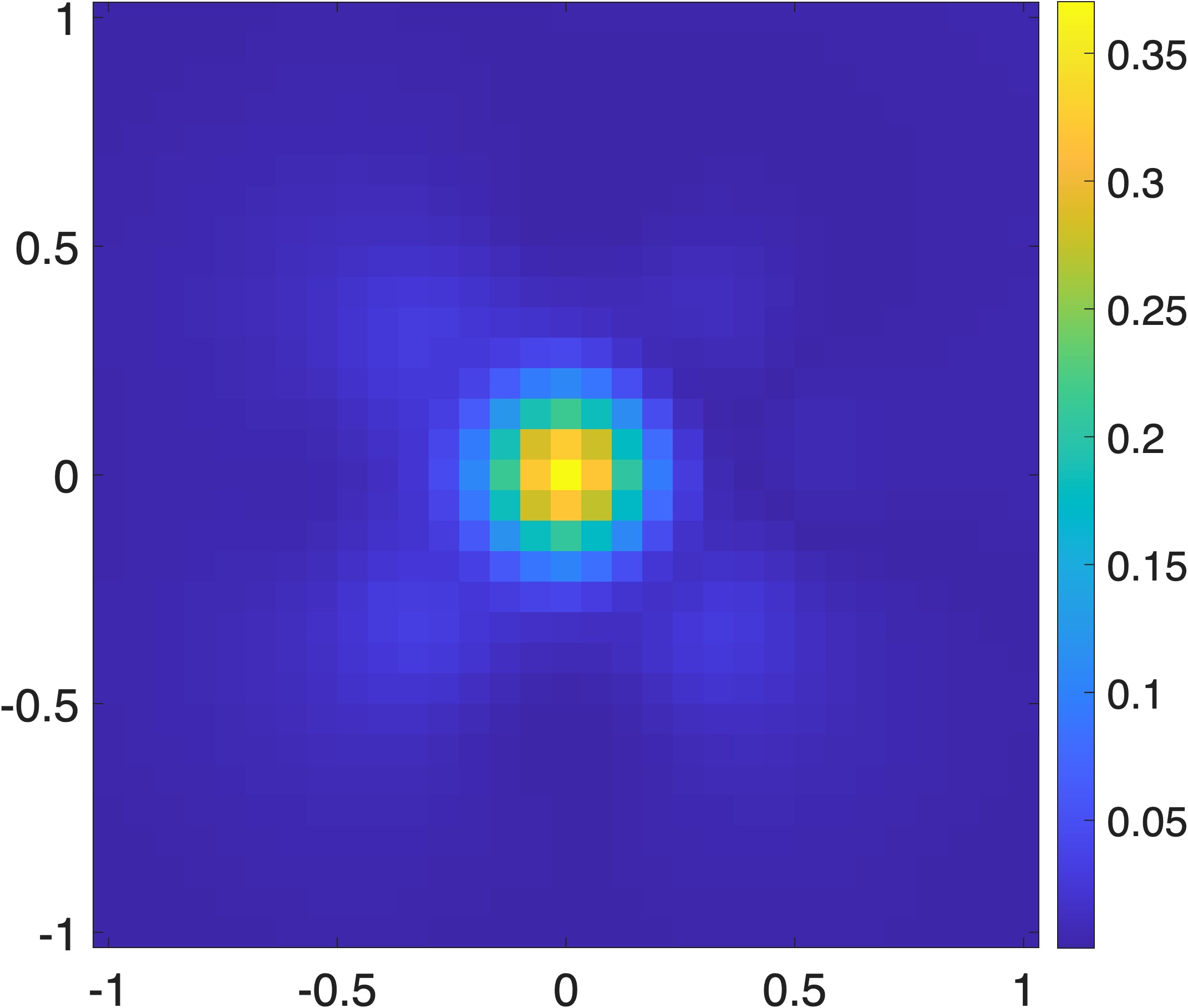}
}

\caption{Test 1. First row:  true velocity components and initial pressure. Second row: reconstructed velocity components and initial pressure. Third row: pointwise relative reconstruction errors.}
\label{fig:case1_all}
\end{figure}

To monitor the convergence of the Picard iteration, we record both the relative changes of the iterates and the relative residuals of the reduced system at each step. More precisely, for the $(k+1)$-st iterate, we define
\begin{equation} \label{5.7}
\operatorname{relU}^{(k)}
:=
\frac{\|\bU^{(k+1)}-\bU^{(k)}\|_{[L^2(\Omega)]^{2(N+1)}}}
{\|\bU^{(k+1)}\|_{[L^2(\Omega)]^{2(N+1)}}},
\qquad
\operatorname{relP}^{(k)}
:=
\frac{\|\bP^{(k+1)}-\bP^{(k)}\|_{[L^2(\Omega)]^{(N+1)}}}
{\|\bP^{(k+1)}\|_{[L^2(\Omega)]^{(N+1)}}}.
\end{equation}
In addition, we compute the relative residuals
\begin{equation}\label{5.8}
\operatorname{ResU}^{(k)}
:=
\frac{
\left\|
\mathcal L_U^{(k+1)}-\mathcal R_U^{(k+1)}
\right\|_{[L^2(\Omega)]^{2(N+1)}}
}{
\left\|\mathcal L_U^{(k+1)}\right\|_{[L^2(\Omega)]^{2(N+1)}}
+
\left\|\mathcal R_U^{(k+1)}\right\|_{[L^2(\Omega)]^{2(N+1)}}
},
\end{equation}
and
\begin{equation}\label{5.9}
\operatorname{ResP}^{(k)}
:=
\frac{
\left\|
\mathcal L_P^{(k+1)}-\mathcal R_P^{(k+1)}
\right\|_{[L^2(\Omega)]^{N+1}}
}{
\left\|\mathcal L_P^{(k+1)}\right\|_{[L^2(\Omega)]^{N+1}}
+
\left\|\mathcal R_P^{(k+1)}\right\|_{[L^2(\Omega)]^{N+1}}
},
\end{equation}
where
\begin{align*}
\mathcal L_U^{(k+1)}
&:=
\mu S_N(\Delta \bU^{(k+1)}),
&
\mathcal R_U^{(k+1)}
&:=
R_N\bU^{(k+1)}
+
\mathbf F^N(\x,\bU^{(k+1)},\nabla \bU^{(k+1)})
+
S_N(\nabla \bP^{(k+1)}),
\\
\mathcal L_P^{(k+1)}
&:=
S_N(\Delta \bP^{(k+1)}),
&
\mathcal R_P^{(k+1)}
&:=
G^N(\x,\bU^{(k+1)},\nabla \bU^{(k+1)}).
\end{align*}

We interpret $\operatorname{ResU}^{(k)}$ and $\operatorname{ResP}^{(k)}$ as relative residuals: their numerators measure the mismatch in the reduced equations, while their denominators scale this mismatch by the size of the corresponding terms. As a result, both quantities are dimensionless and remain comparable across different test cases and iterations. Figure~\ref{fig:case1_hist} displays their evolution and illustrates both the stabilization of the iterates and the extent to which the reduced system is satisfied.

\begin{figure}[ht]
\centering

\subfloat[Relative errors $\mathrm{relU}$ and $\mathrm{relP}$ versus the iteration number.\label{fig:case1_rel_hist}]{
    \includegraphics[width=0.42\textwidth]{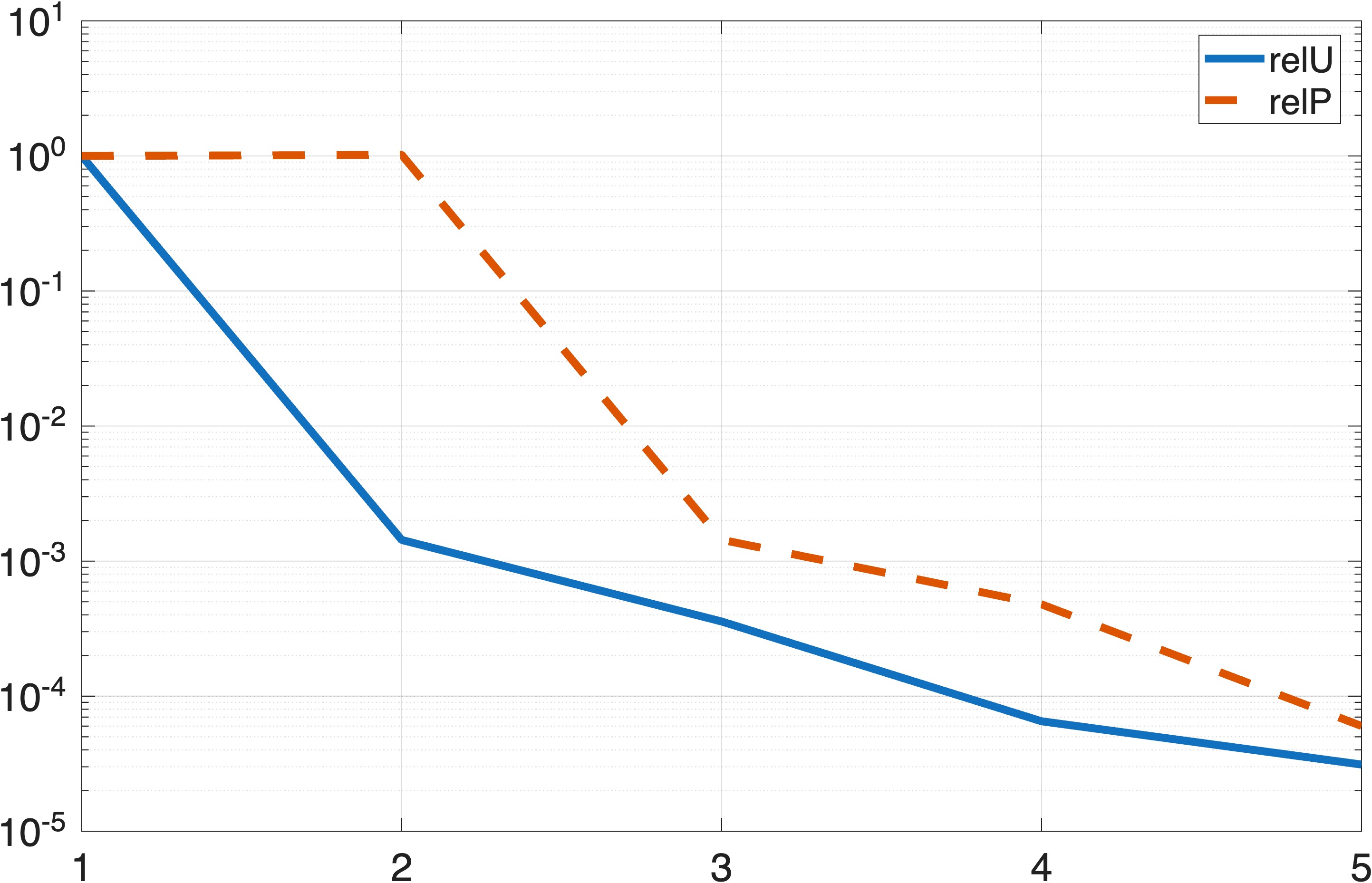}
}
\quad
\subfloat[Residuals $\mathrm{ResU}$ and $\mathrm{ResP}$ versus the iteration number.\label{fig:case1_res_hist}]{
    \includegraphics[width=0.42\textwidth]{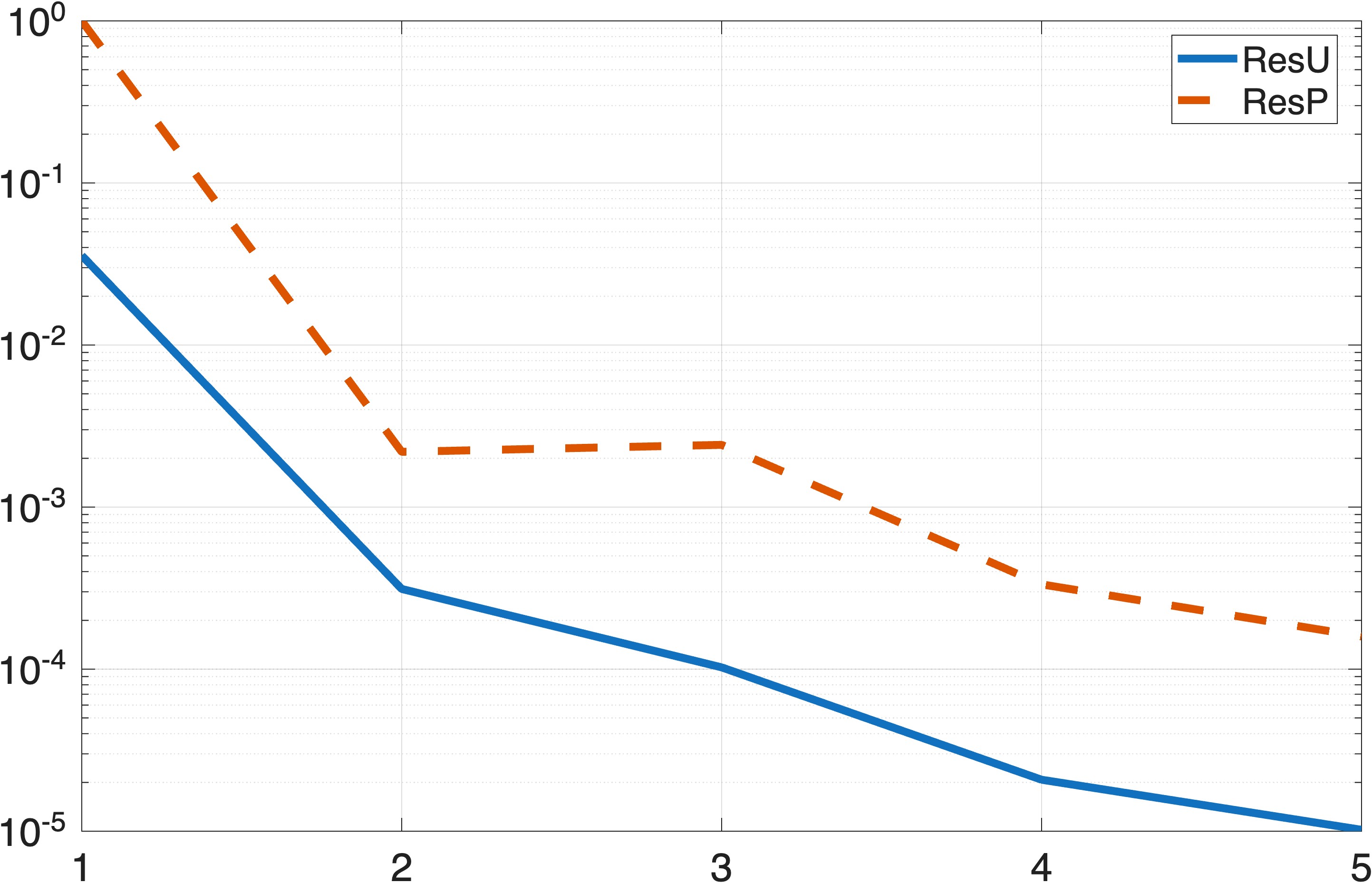}
}

\caption{Test 1. Convergence history of the Picard iteration. The left image shows the relative changes $\operatorname{relU}^{(k)}$ and $\operatorname{relP}^{(k)}$, while the right image shows the relative residuals $\operatorname{ResU}^{(k)}$ and $\operatorname{ResP}^{(k)}$.}
\label{fig:case1_hist}
\end{figure}

Figure~\ref{fig:case1_all} shows that the proposed method reconstructs the main structures of the true initial velocity and pressure with good accuracy. In particular, the reconstructed profiles of $u_{0,1}$, $u_{0,2}$, and $p_0$ agree well with the corresponding true ones, and the error plots indicate that the discrepancies remain localized and moderate throughout the computational domain. Here, the relative $L^2$ error of a reconstructed quantity $q^{\rm rec}$ with true value $q^{\rm true}$ is defined by
\[
\frac{\|q^{\rm rec}-q^{\rm true}\|_{L^2(\Omega)}}{\|q^{\rm true}\|_{L^2(\Omega)}}\times 100\%.
\]
With this definition, the relative $L^2$ errors are $14.47097\%$ for $u_{0,1}$, $13.33345\%$ for $u_{0,2}$, and $23.91453\%$ for $p_0$. Although the pressure reconstruction is slightly less accurate than the velocity reconstruction, the overall quality remains satisfactory for this first test. In addition, Figure~\ref{fig:case1_hist} demonstrates a stable and rapidly convergent iterative process: both the relative changes of the iterates and the relative residuals decrease steadily as the iteration proceeds, indicating that the Picard scheme stabilizes quickly and that the computed solution satisfies the reduced system with good accuracy.

{\bf Test 2.}
We next present the second numerical test. In this example, the body force $\bff=(f_1,f_2)$ is defined by
\[
f_1(x,y)=e^{-\frac{(x/0.30)^2+(y/0.30)^2}{1.0}},
\qquad
f_2(x,y)=e^{-\frac{(x/0.22)^2+\big((y+0.20)/0.28\big)^2}{1.0}}.
\]
The force field is normalized so that its maximum absolute component equals $1$.

The true initial velocity $\bu_0^{\rm true}(x,y)=\big(u_{0,1}^{\rm true}(x,y),u_{0,2}^{\rm true}(x,y)\big)$ is defined by
\begin{multline*}
u_{0,1}^{\rm true}(x,y)
=
0.08\Big[
-4y(1-y^2)(1-x^2)^2
\Big(
e^{-\frac{(x+0.25)^2+(y-0.05)^2}{0.14}}
-0.85\,e^{-\frac{(x-0.20)^2+(y+0.18)^2}{0.10}}
\Big) \\
+(1-x^2)^2(1-y^2)^2
\Big(
-\frac{2(y-0.05)}{0.14}e^{-\frac{(x+0.25)^2+(y-0.05)^2}{0.14}}
+0.85\,\frac{2(y+0.18)}{0.10}e^{-\frac{(x-0.20)^2+(y+0.18)^2}{0.10}}
\Big)
\Big],
\end{multline*}
and
\begin{multline*}
u_{0,2}^{\rm true}(x,y)
=
-0.08\Big[
-4x(1-x^2)(1-y^2)^2
\Big(
e^{-\frac{(x+0.25)^2+(y-0.05)^2}{0.14}}
-0.85\,e^{-\frac{(x-0.20)^2+(y+0.18)^2}{0.10}}
\Big) \\
+(1-x^2)^2(1-y^2)^2
\Big(
-\frac{2(x+0.25)}{0.14}e^{-\frac{(x+0.25)^2+(y-0.05)^2}{0.14}}
+0.85\,\frac{2(x-0.20)}{0.10}e^{-\frac{(x-0.20)^2+(y+0.18)^2}{0.10}}
\Big)
\Big].
\end{multline*}
This field is also normalized so that its maximum absolute component equals $1$. 
It is easy to verify that the initial velocity field $\bu_0^{\rm true}(x,y)=\bigl(u_{0,1}^{\rm true}(x,y),\,u_{0,2}^{\rm true}(x,y)\bigr)$ satisfies the incompressibility condition $\nabla\cdot \bu_0^{\rm true}=0$ in $\Omega$.
The true and reconstructed initial velocity components and pressure for Test~2 are displayed in Figure~\ref{fig:case2_all}, together with the corresponding pointwise relative reconstruction errors.

\begin{figure}[ht]
\centering

\subfloat[$u_{0,1}^{\rm true}$\label{fig:case2_true_u1}]{
    \includegraphics[width=0.25\textwidth]{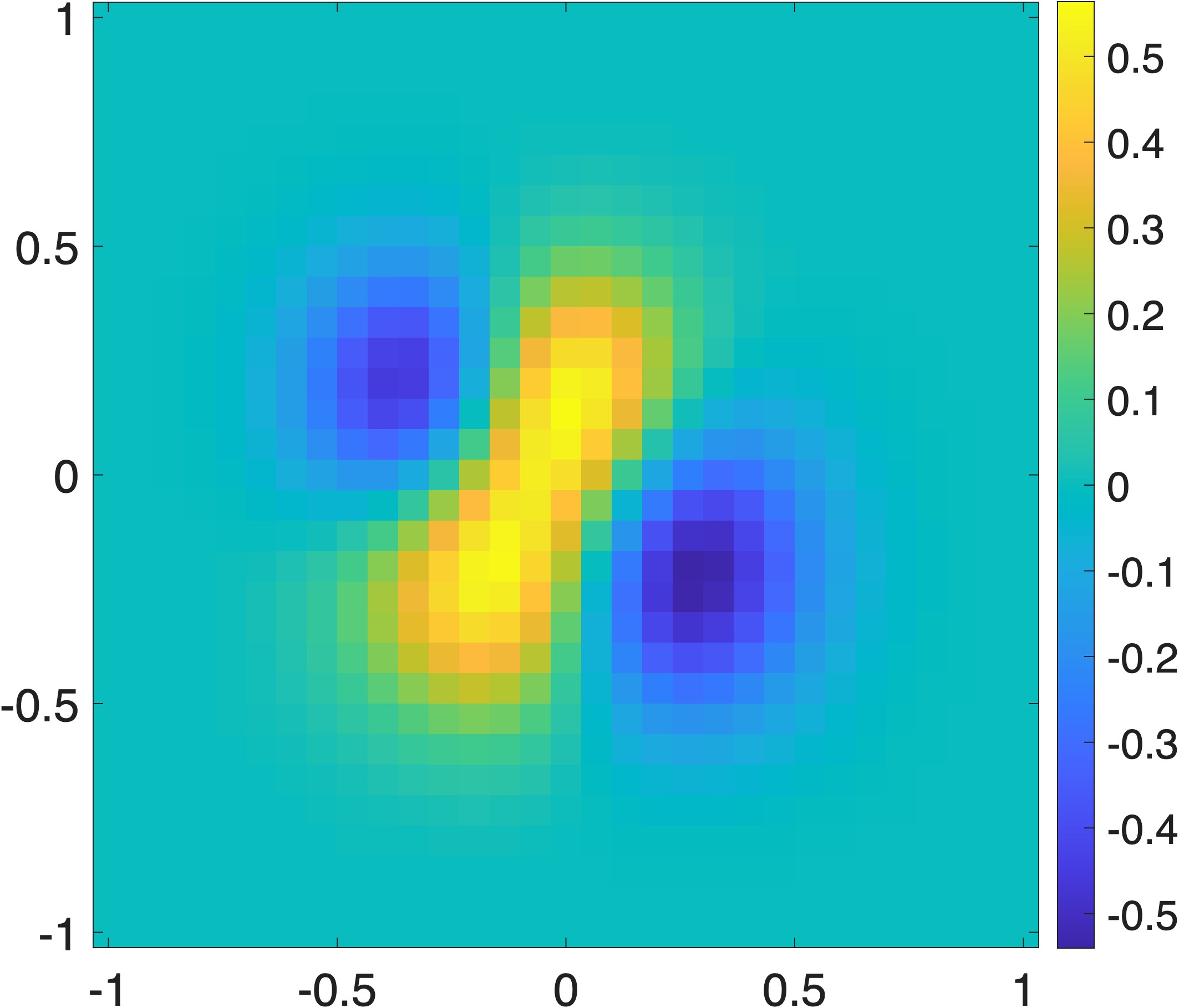}
}
\quad
\subfloat[$u_{0,2}^{\rm true}$\label{fig:case2_true_u2}]{
    \includegraphics[width=0.25\textwidth]{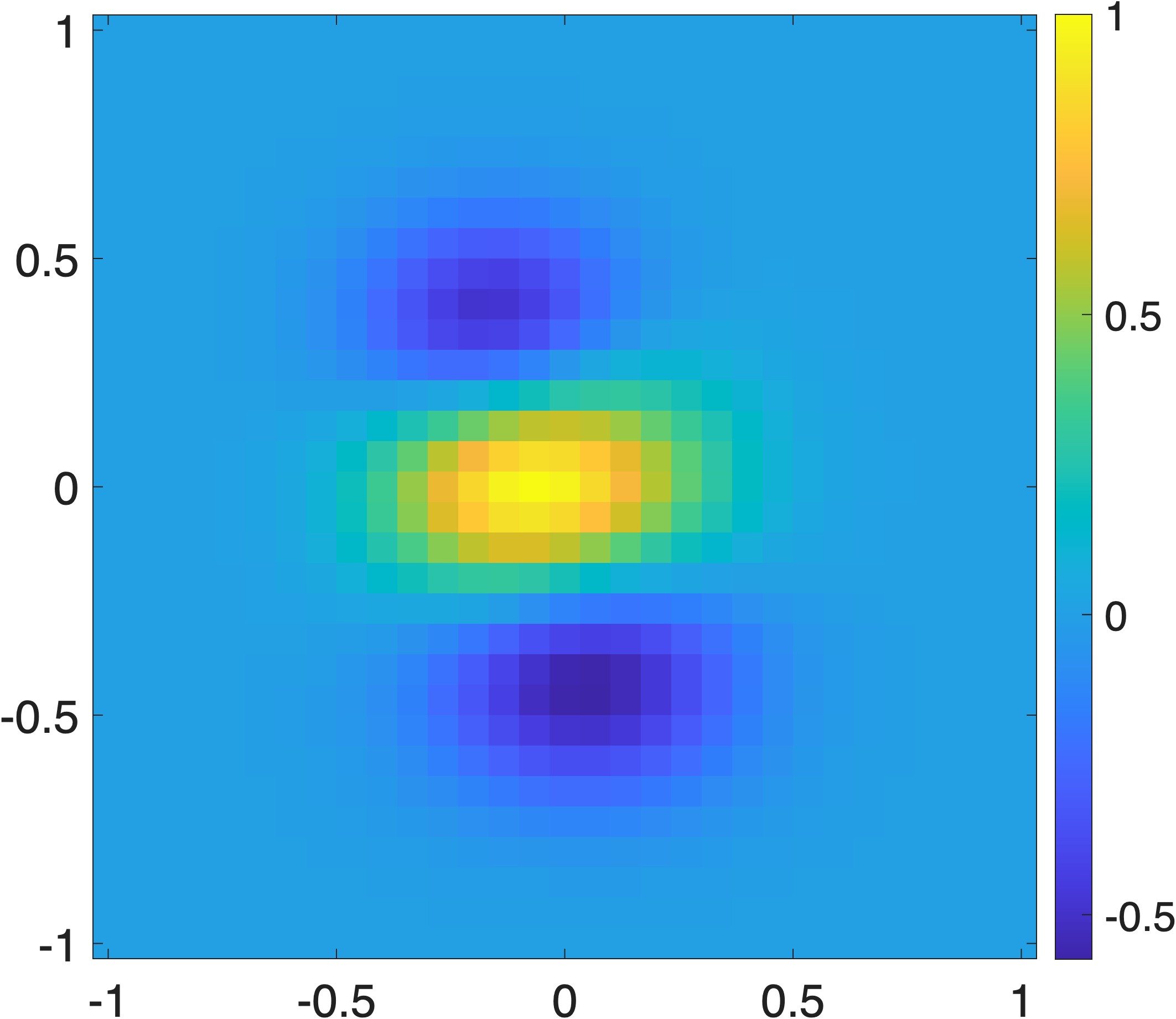}
}
\quad
\subfloat[$p_{0}^{\rm true}$\label{fig:case2_true_p0}]{
    \includegraphics[width=0.25\textwidth]{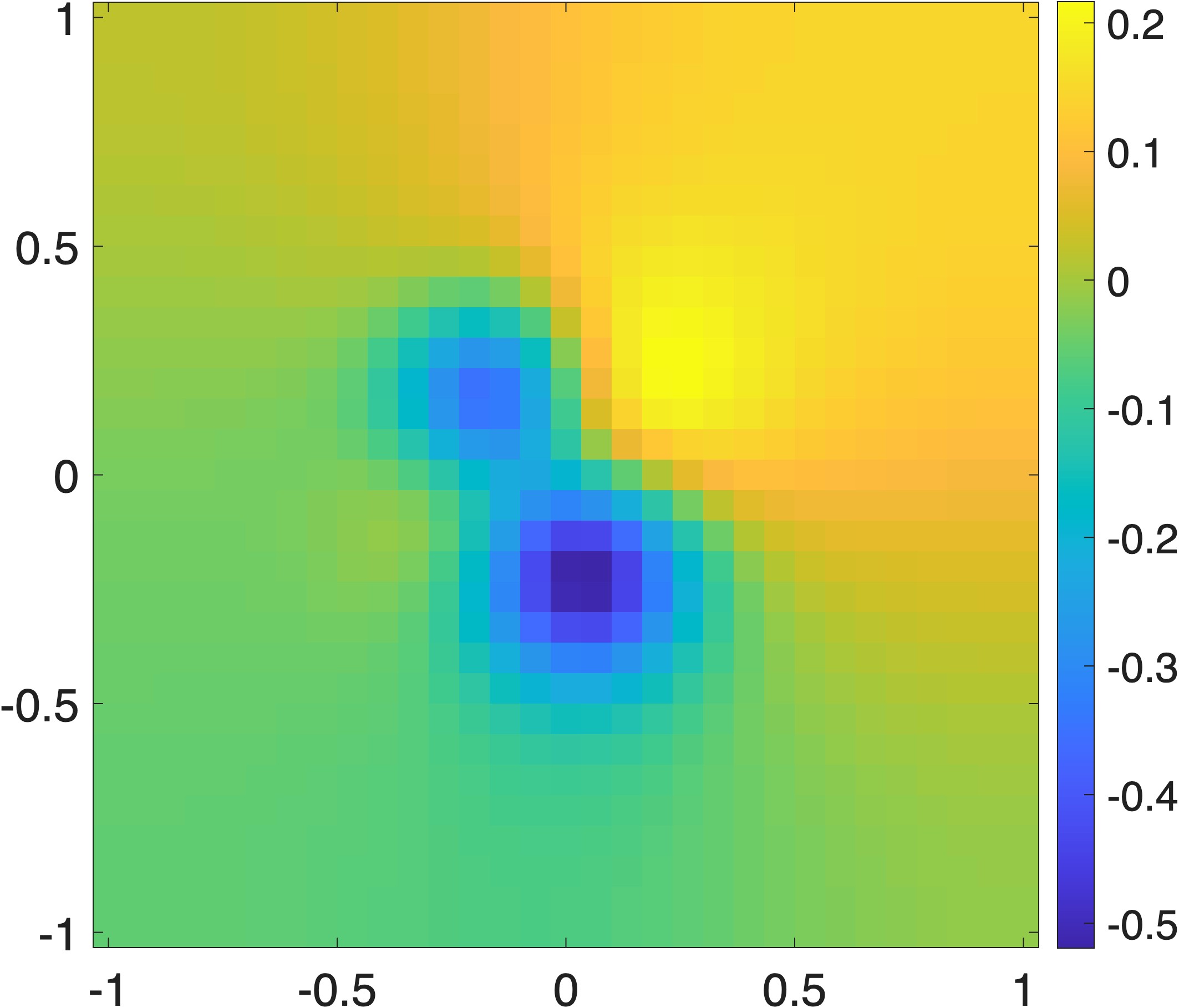}
}

\subfloat[$u_{0,1}^{\rm rec}$\label{fig:case2_rec_u1}]{
    \includegraphics[width=0.25\textwidth]{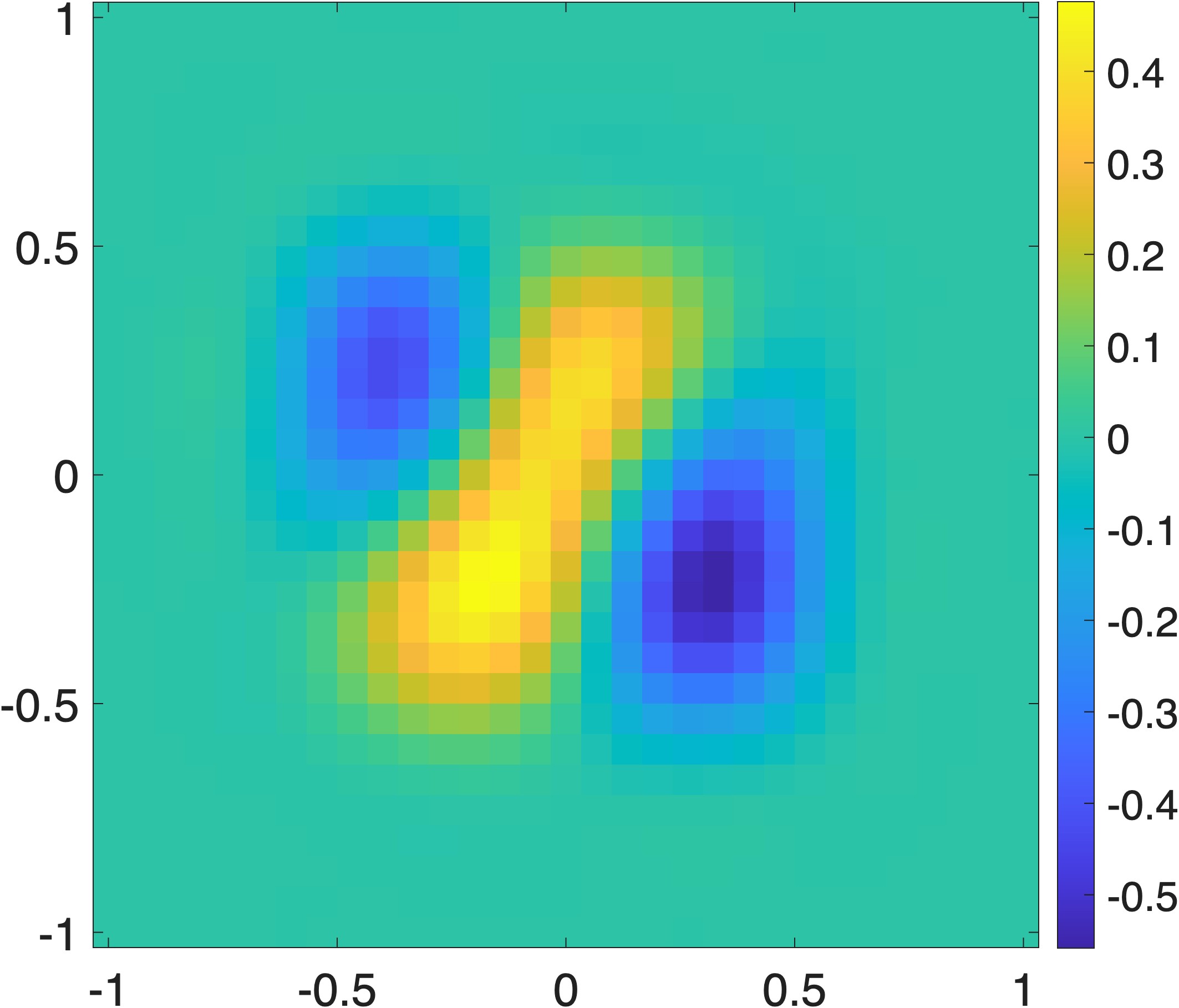}
}
\quad
\subfloat[$u_{0,2}^{\rm rec}$\label{fig:case2_rec_u2}]{
    \includegraphics[width=0.25\textwidth]{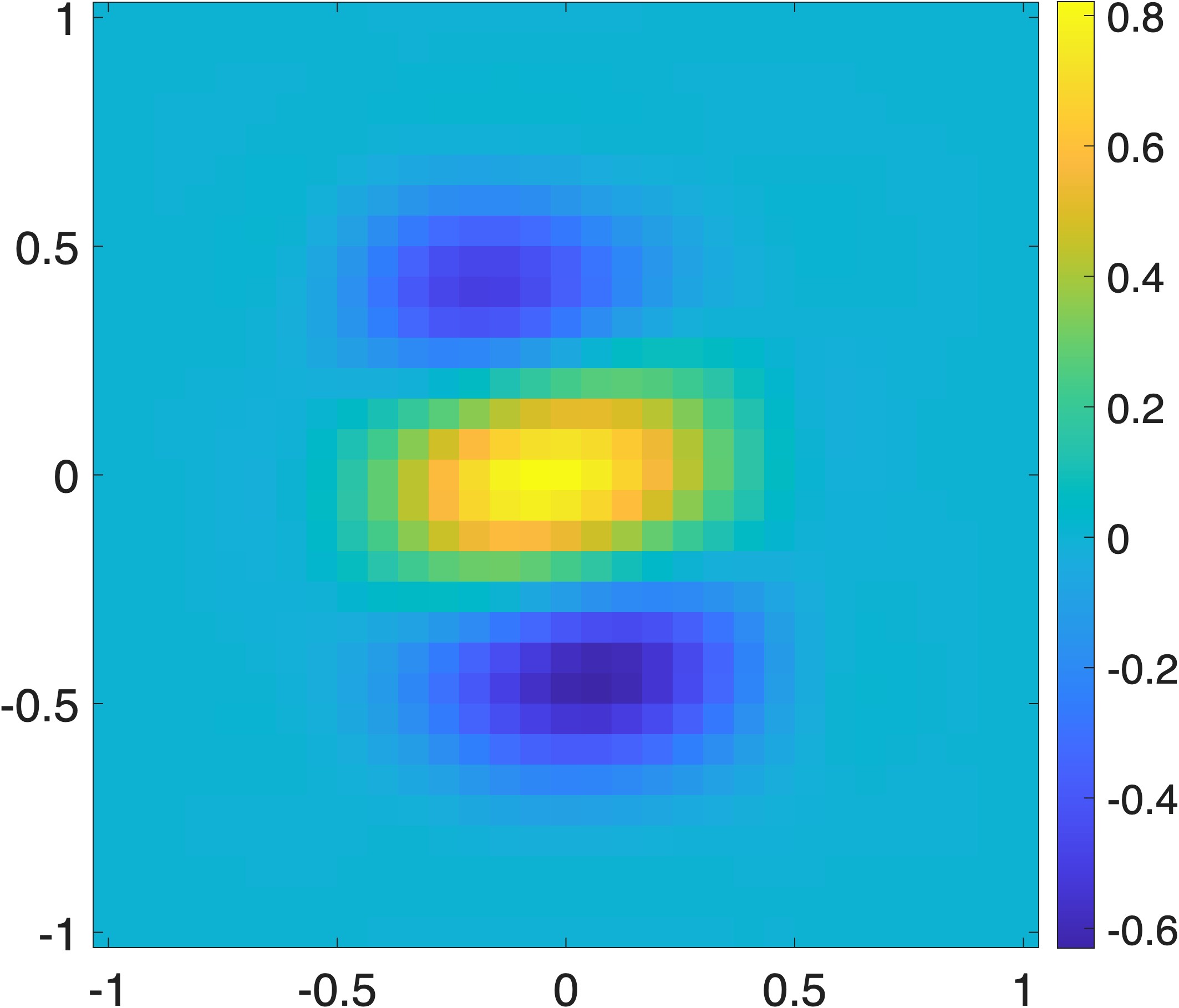}
}
\quad
\subfloat[$p_{0}^{\rm rec}$\label{fig:case2_rec_p0}]{
    \includegraphics[width=0.25\textwidth]{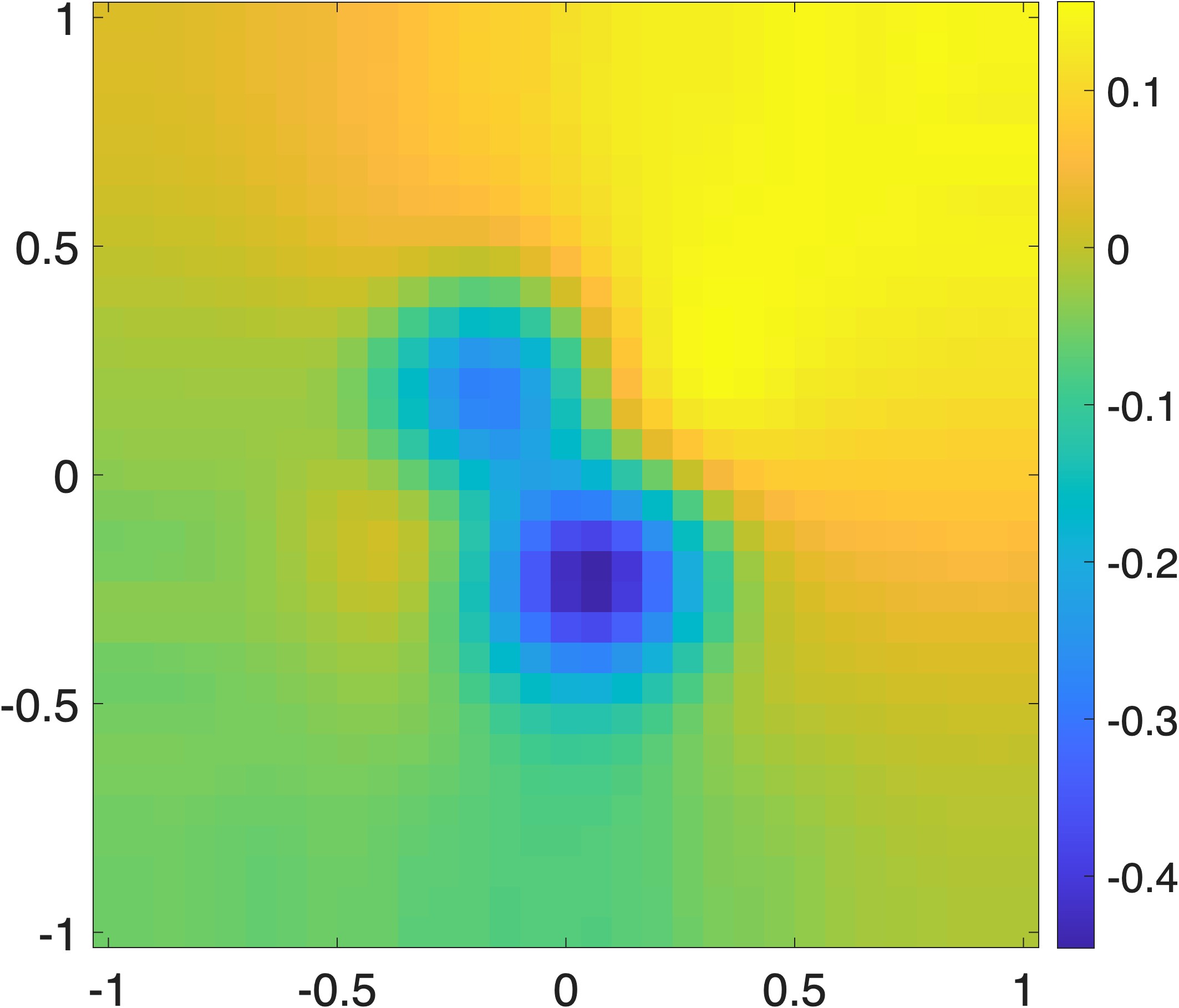}
}

\subfloat[$\frac{|u_{0,1}^{\rm true}-u_{0,1}^{\rm rec}|}{\|u_{0,1}^{\rm true}\|_{L^\infty(\Omega)}}$\label{fig:case2_err_u1}]{
    \includegraphics[width=0.25\textwidth]{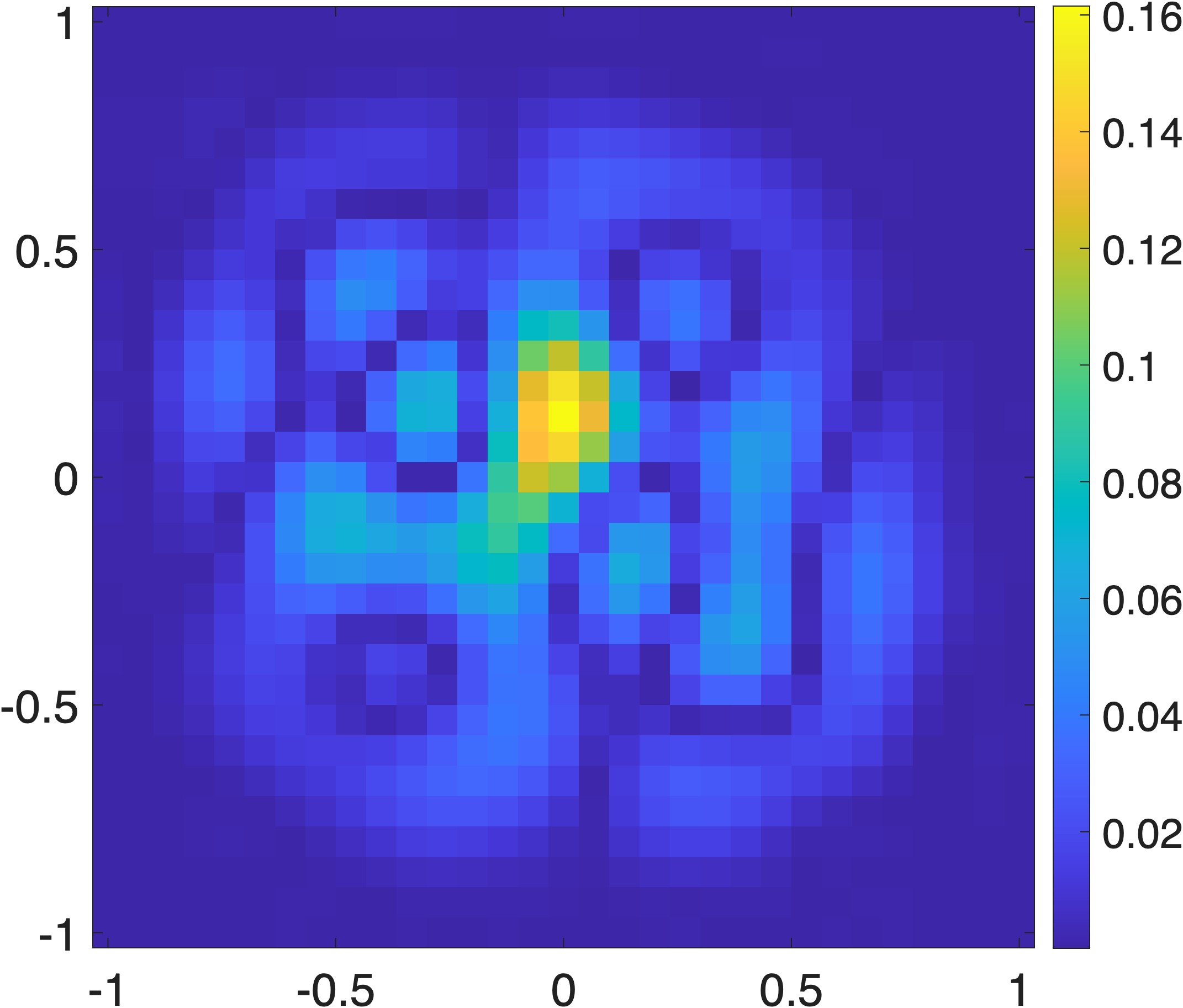}
}
\quad
\subfloat[$\frac{|u_{0,2}^{\rm true}-u_{0,2}^{\rm rec}|}{\|u_{0,2}^{\rm true}\|_{L^\infty(\Omega)}}$\label{fig:case2_err_u2}]{
    \includegraphics[width=0.25\textwidth]{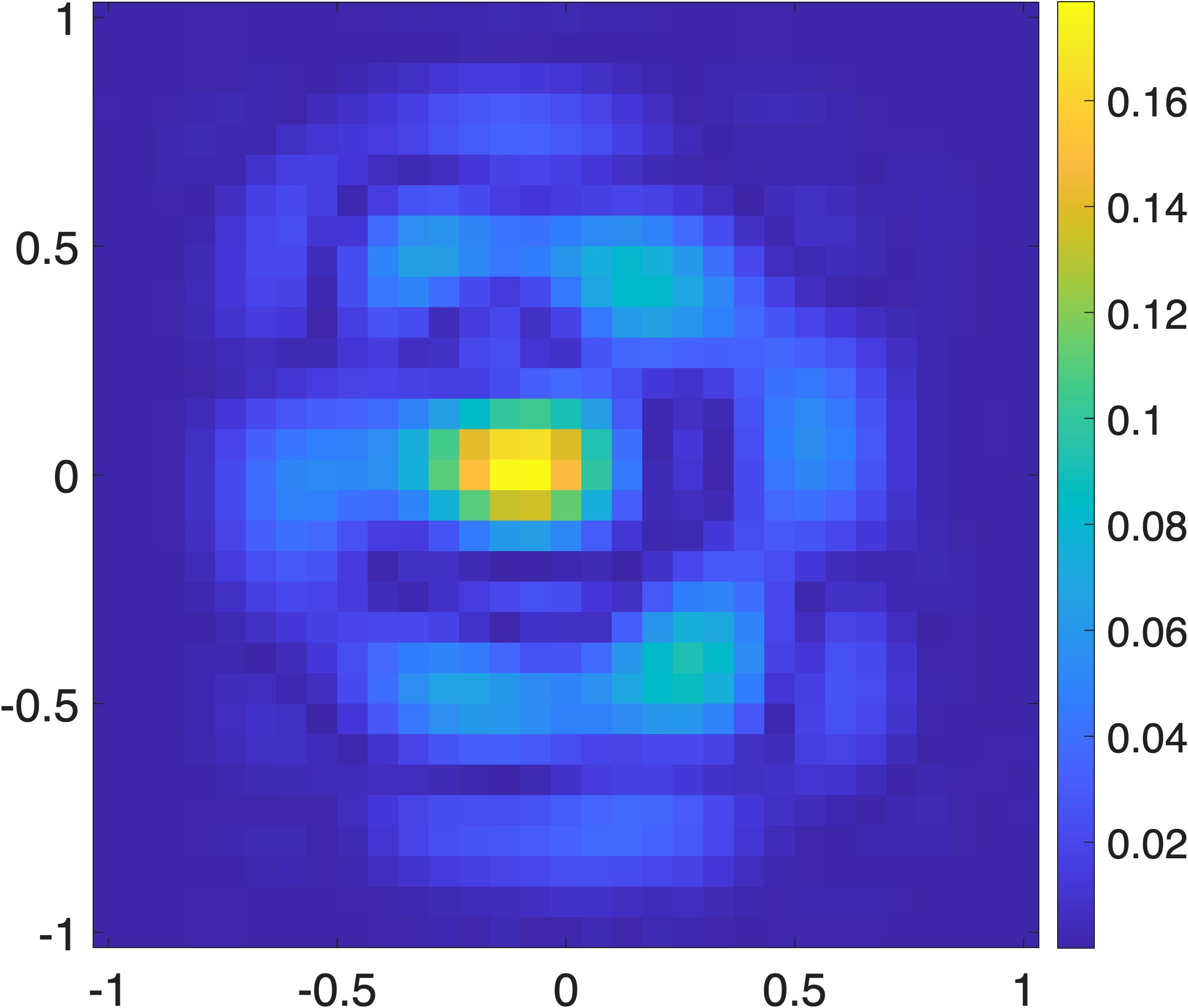}
}
\quad
\subfloat[$\frac{|p_{0}^{\rm true}-p_{0}^{\rm rec}|}{\|p_{0}^{\rm true}\|_{L^\infty}(\Omega)}$\label{fig:case2_err_p0}]{
    \includegraphics[width=0.25\textwidth]{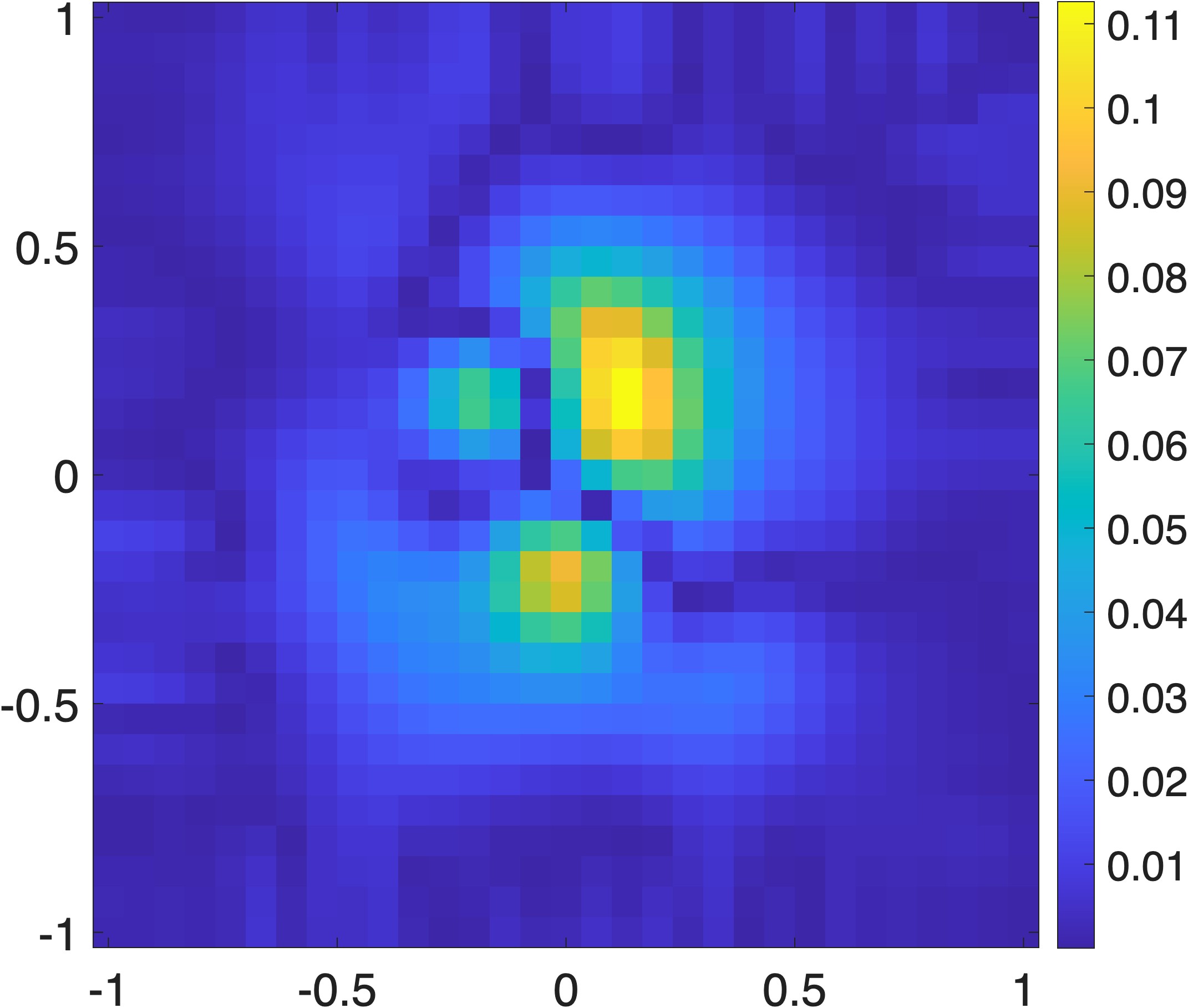}
}

\caption{Test 2. First row: true initial velocity components and pressure. Second row: reconstructed initial velocity components and pressure. Third row: pointwise relative reconstruction errors.}
\label{fig:case2_all}
\end{figure}

As in Test~1, we monitor the quantities $\operatorname{relU}^{(k)}$, $\operatorname{relP}^{(k)}$, $\operatorname{ResU}^{(k)}$, and $\operatorname{ResP}^{(k)}$ defined in \eqref{5.7}, \eqref{5.8}, and \eqref{5.9}; their values are displayed in Figure \ref{fig:case2_hist} and are used to assess both the stabilization of the iterates and the accuracy with which the reduced system is satisfied.

\begin{figure}[ht]
\centering

\subfloat[$\operatorname{relU}^{(k)}$ and $\operatorname{relP}^{(k)}$\label{fig:case2_rel_hist}]{
    \includegraphics[width=0.42\textwidth]{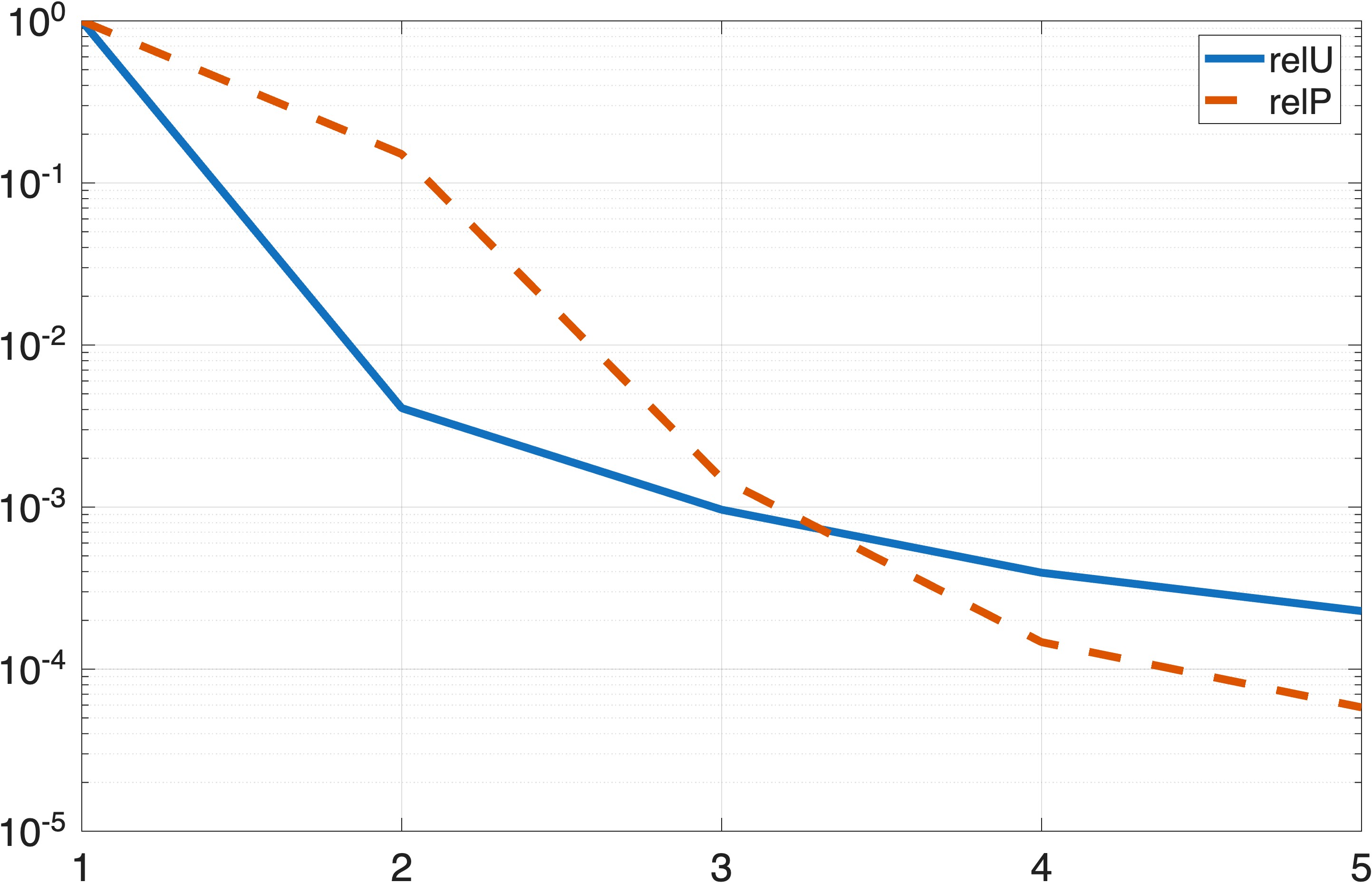}
}
\quad
\subfloat[$\operatorname{ResU}^{(k)}$ and $\operatorname{ResP}^{(k)}$\label{fig:case2_res_hist}]{
    \includegraphics[width=0.42\textwidth]{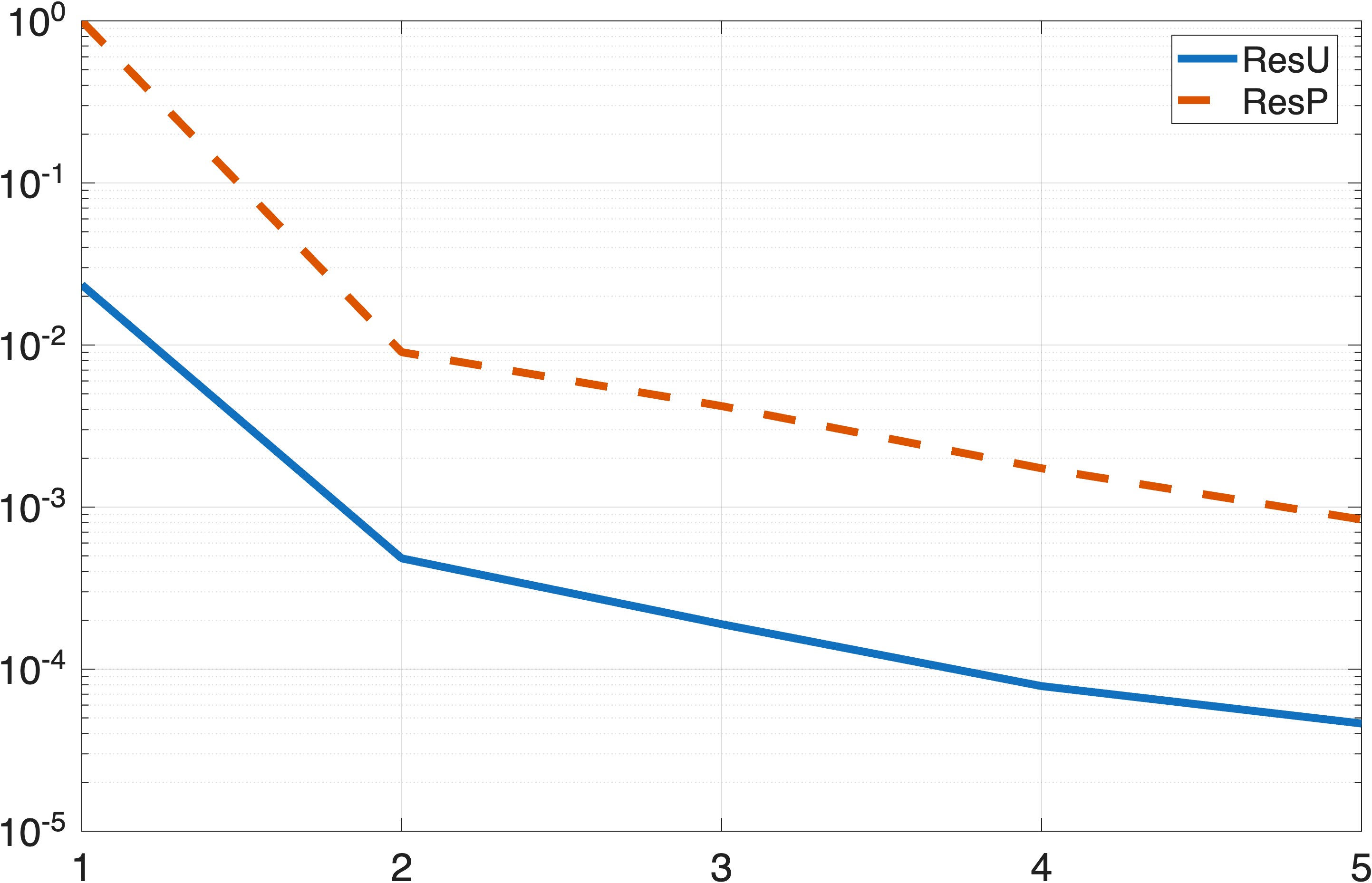}
}

\caption{Test 2. Convergence history of the Picard iteration. The left image shows the relative changes $\operatorname{relU}^{(k)}$ and $\operatorname{relP}^{(k)}$, while the right image shows the relative residuals $\operatorname{ResU}^{(k)}$ and $\operatorname{ResP}^{(k)}$.}
\label{fig:case2_hist}
\end{figure}

Figure~\ref{fig:case2_all} shows that the proposed method again produces a satisfactory reconstruction of the initial velocity and pressure for Test~2. The reconstructed profiles capture the main geometric features of the true solution, including the locations and shapes of the dominant positive and negative regions, while the error plots indicate that the discrepancies remain moderate over the computational domain. Using the relative $L^2$ error $\|q^{\rm rec}-q^{\rm true}\|_{L^2(\Omega)}/\|q^{\rm true}\|_{L^2(\Omega)}\times 100\%$, the reconstruction errors are $17.64913\%$ for $u_{0,1}$, $15.74004\%$ for $u_{0,2}$, and $18.65730\%$ for $p_0$. These values confirm that the method maintains good accuracy even for this more complicated test. In addition, Figure~\ref{fig:case2_hist} shows a stable convergence behavior of the Picard iteration: both the relative changes of the iterates and the relative residuals decrease steadily, indicating that the iteration stabilizes reliably and that the computed solution satisfies the reduced system with good accuracy.

{\bf Test 3.}
We next present the third numerical test. In this example, the body force $\bff=(f_1,f_2)$ is defined by
\[
f_1(x,y)=e^{-\frac{\big((x+0.04)/0.28\big)^2+\big((y-0.02)/0.32\big)^2}{0.95}},
\qquad
f_2(x,y)=e^{-\frac{\big((x-0.03)/0.24\big)^2+\big((y+0.18)/0.30\big)^2}{0.90}}.
\]
The force field is normalized so that its maximum absolute component equals $1$.

The true initial velocity $\bu_0^{\rm true}(x,y)=\big(u_{0,1}^{\rm true}(x,y),u_{0,2}^{\rm true}(x,y)\big)$ is defined by
\begin{multline*}
u_{0,1}^{\rm true}(x,y)
=
0.08\Big[
\Big(
-4y(1-y^2)(1-x^2)^2
-\frac{2(y-0.35)}{0.06}(1-x^2)^2(1-y^2)^2
\Big)
e^{-\frac{(x+0.45)^2+(y-0.35)^2}{0.06}} \\
-0.85
\Big(
-4y(1-y^2)(1-x^2)^2
-\frac{2(y+0.35)}{0.05}(1-x^2)^2(1-y^2)^2
\Big)
e^{-\frac{(x-0.45)^2+(y+0.35)^2}{0.05}}
\Big],
\end{multline*}
and
\begin{multline*}
u_{0,2}^{\rm true}(x,y)
=
-0.08\Big[
\Big(
-4x(1-x^2)(1-y^2)^2
-\frac{2(x+0.45)}{0.06}(1-x^2)^2(1-y^2)^2
\Big)
e^{-\frac{(x+0.45)^2+(y-0.35)^2}{0.06}} \\
-0.85
\Big(
-4x(1-x^2)(1-y^2)^2
-\frac{2(x-0.45)}{0.05}(1-x^2)^2(1-y^2)^2
\Big)
e^{-\frac{(x-0.45)^2+(y+0.35)^2}{0.05}}
\Big].
\end{multline*}
This field is also normalized so that its maximum absolute component equals $1$. One can check that $\bu_0^{\rm true}(x,y)=\bigl(u_{0,1}^{\rm true}(x,y),\,u_{0,2}^{\rm true}(x,y)\bigr)$ is divergence-free, that is, $\nabla\cdot \bu_0^{\rm true}=0$ in $\Omega$.
The true and reconstructed initial velocity components and pressure for Test~3 are displayed in Figure~\ref{fig:case3_all}, together with the corresponding pointwise relative reconstruction errors.

\begin{figure}[ht]
\centering

\subfloat[$u_{0,1}^{\rm true}$\label{fig:case3_true_u1}]{
    \includegraphics[width=0.25\textwidth]{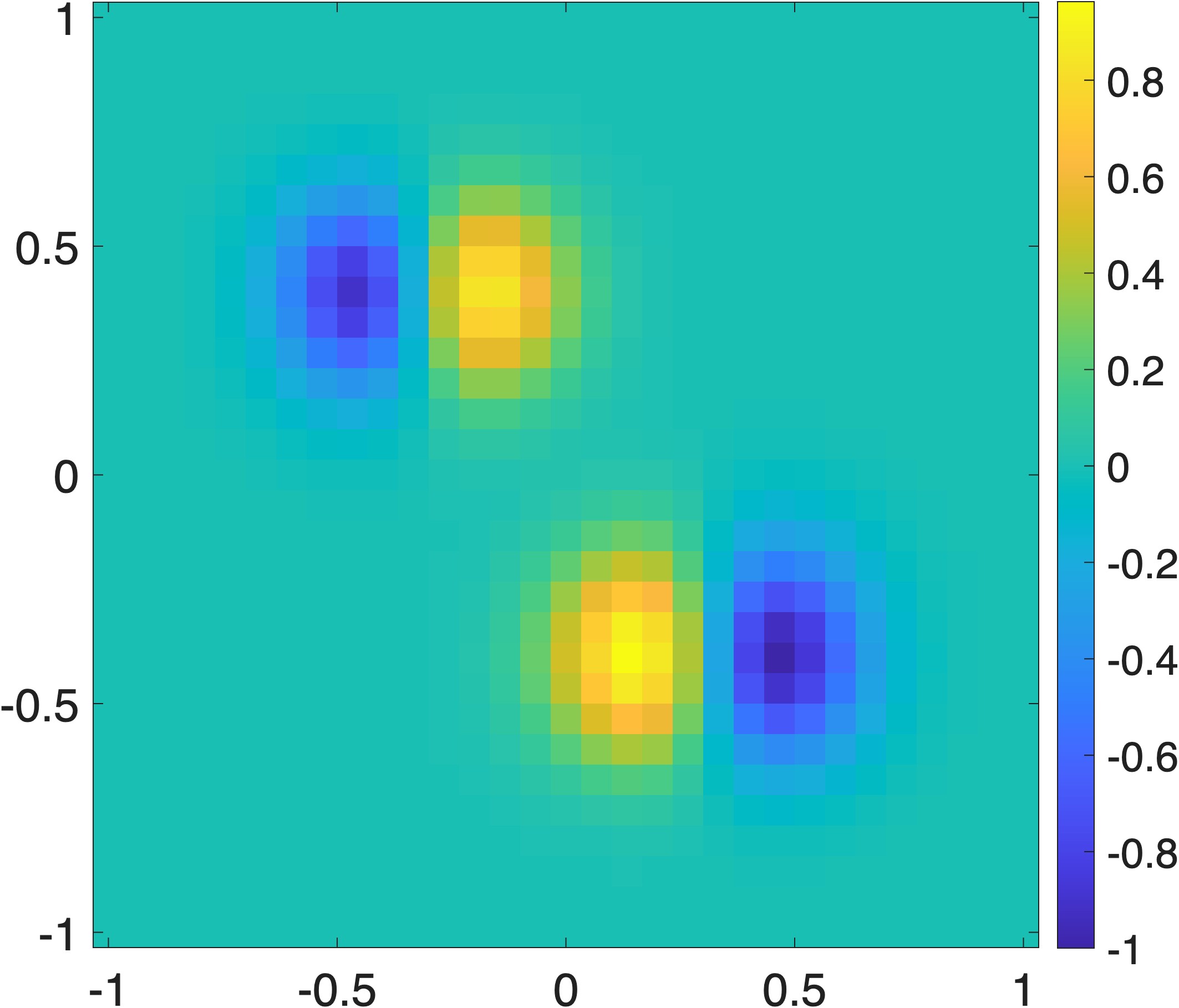}
}
\quad
\subfloat[$u_{0,2}^{\rm true}$\label{fig:case3_true_u2}]{
    \includegraphics[width=0.25\textwidth]{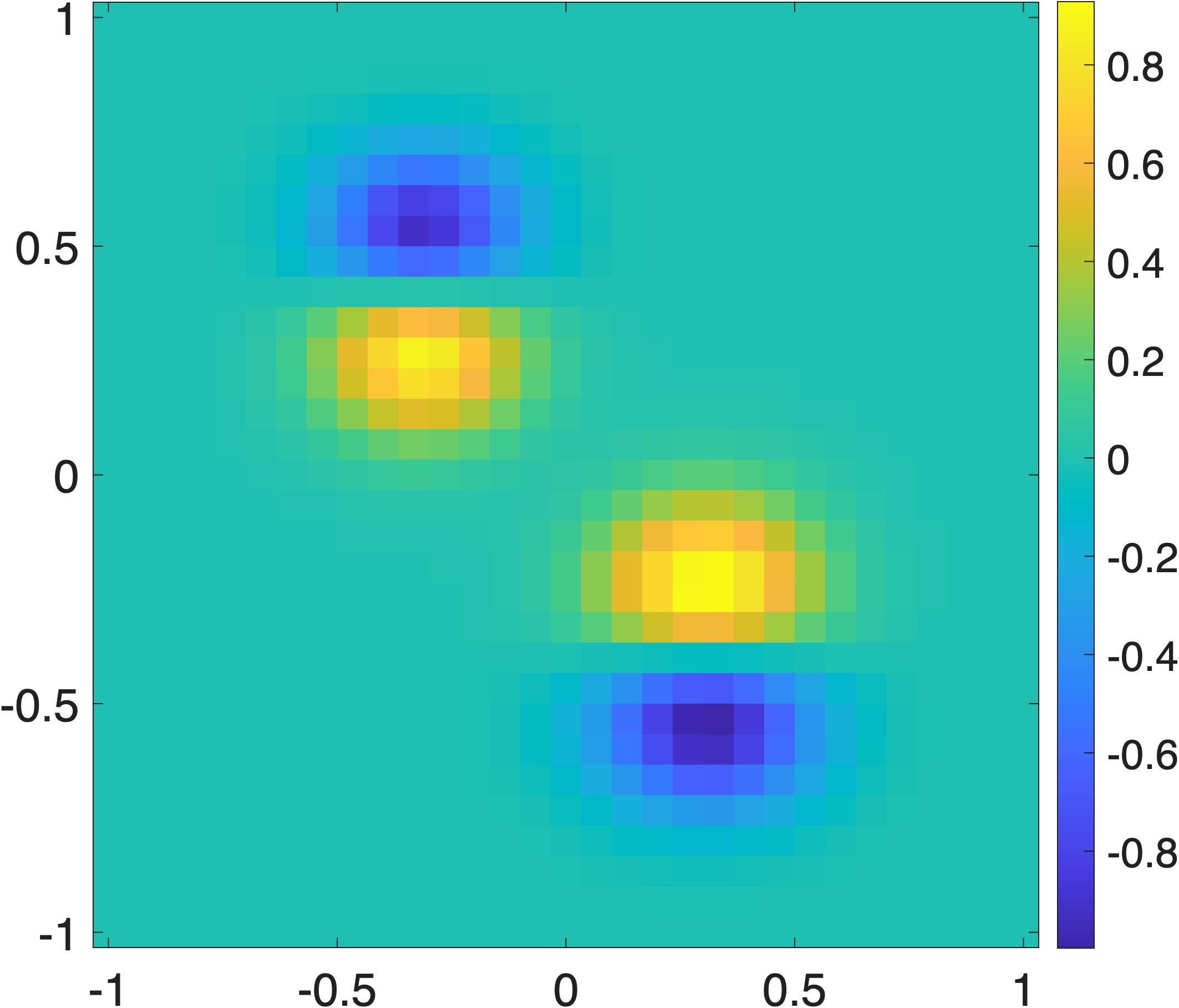}
}
\quad
\subfloat[$p_{0}^{\rm true}$\label{fig:case3_true_p0}]{
    \includegraphics[width=0.25\textwidth]{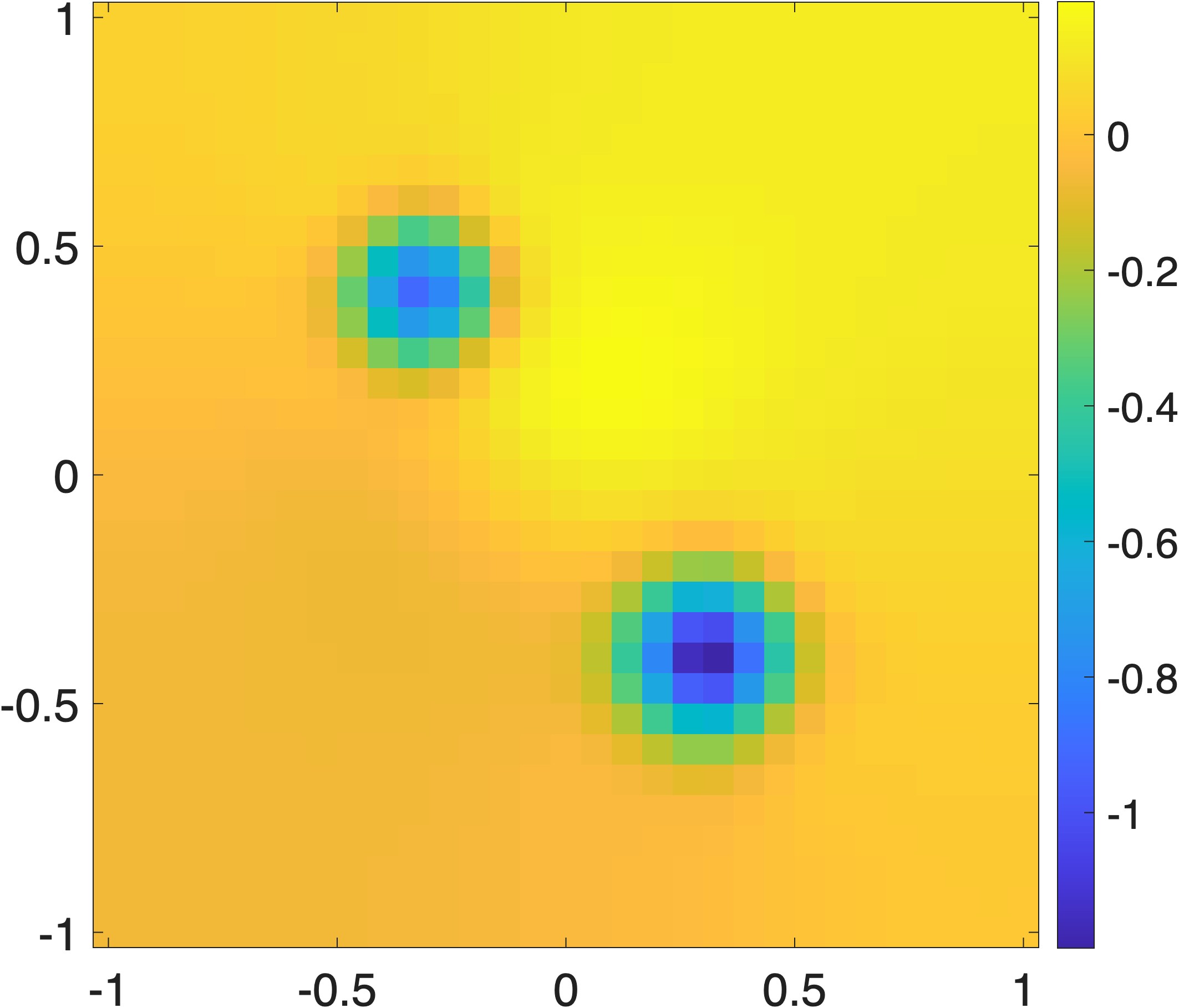}
}

\subfloat[$u_{0,1}^{\rm rec}$\label{fig:case3_rec_u1}]{
    \includegraphics[width=0.25\textwidth]{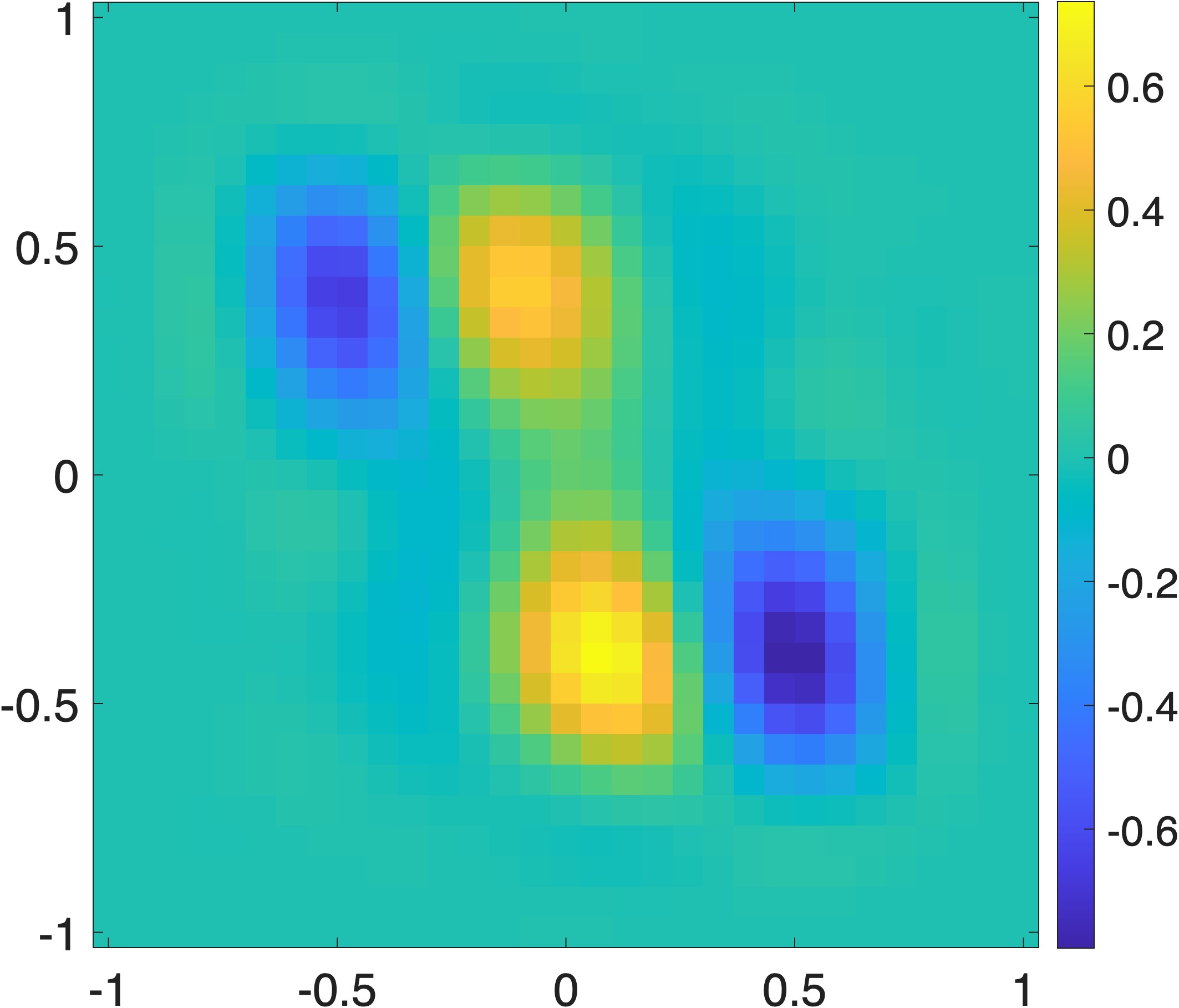}
}
\quad
\subfloat[$u_{0,2}^{\rm rec}$\label{fig:case3_rec_u2}]{
    \includegraphics[width=0.25\textwidth]{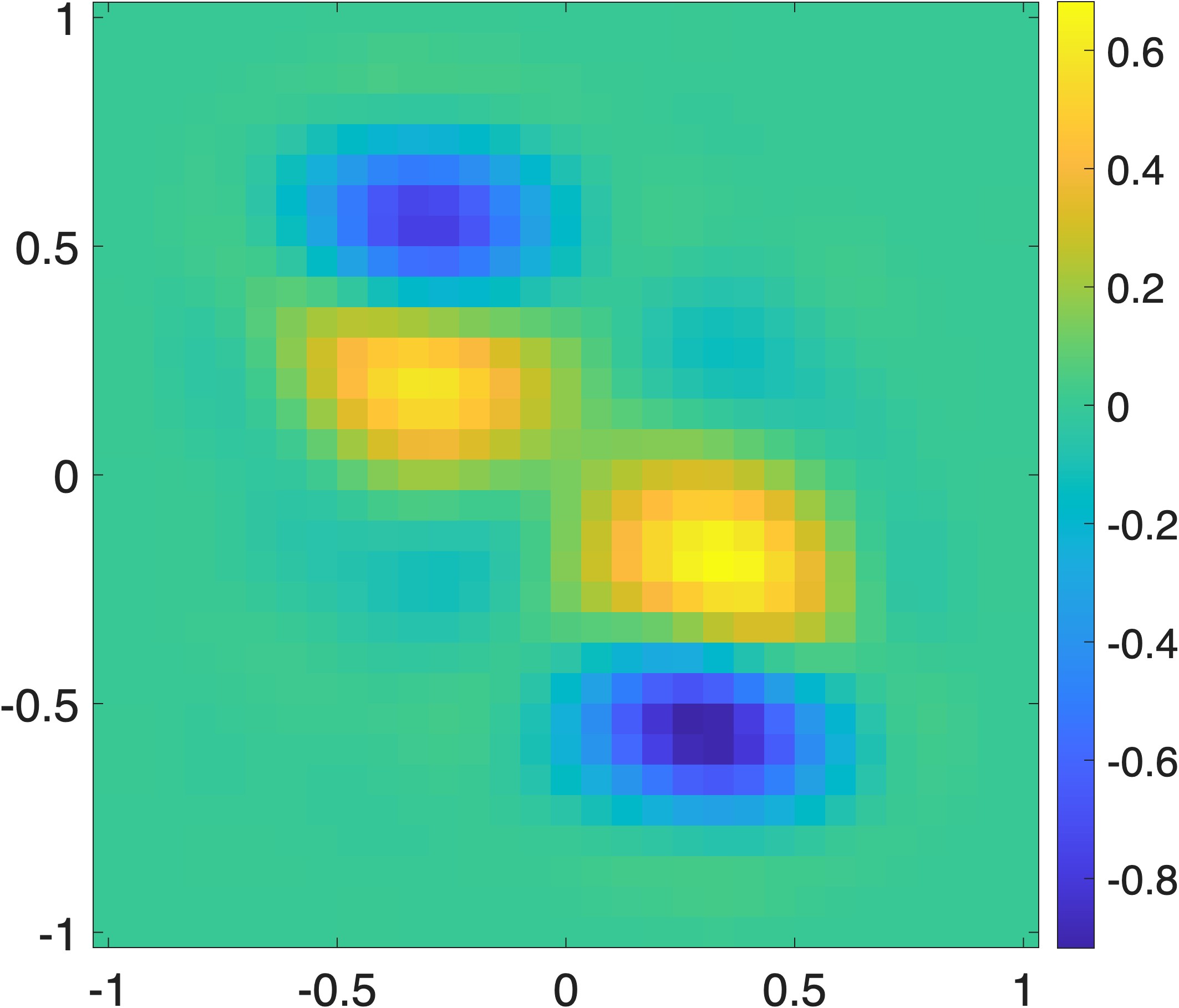}
}
\quad
\subfloat[$p_{0}^{\rm rec}$\label{fig:case3_rec_p0}]{
    \includegraphics[width=0.25\textwidth]{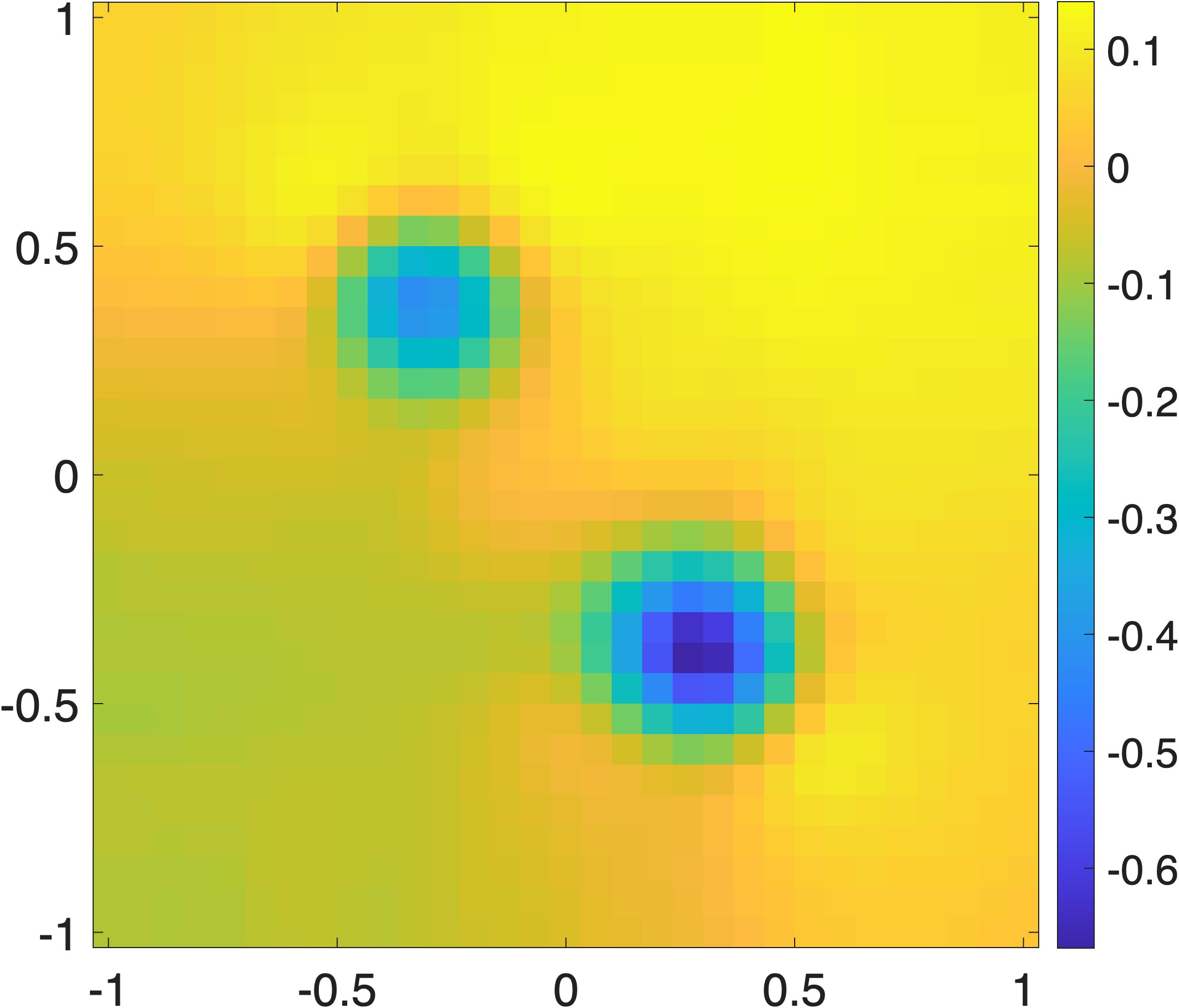}
}

\subfloat[$\frac{|u_{0,1}^{\rm true}-u_{0,1}^{\rm rec}|}{\|u_{0,1}^{\rm true}\|_{L^\infty(\Omega)}}$\label{fig:case3_err_u1}]{
    \includegraphics[width=0.25\textwidth]{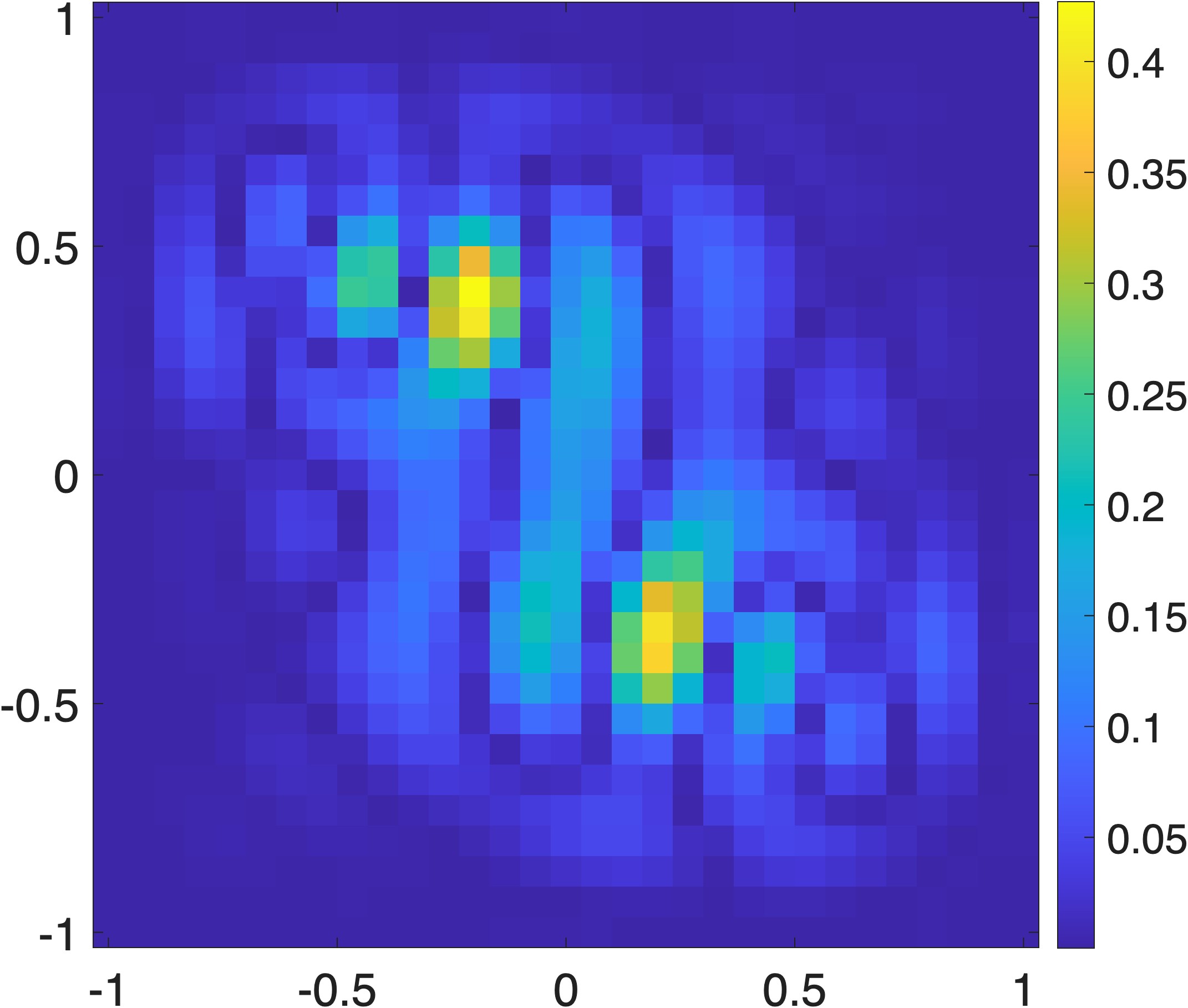}
}
\quad
\subfloat[$\frac{|u_{0,2}^{\rm true}-u_{0,2}^{\rm rec}|}{\|u_{0,2}^{\rm true}\|_{L^\infty(\Omega)}}$\label{fig:case3_err_u2}]{
    \includegraphics[width=0.25\textwidth]{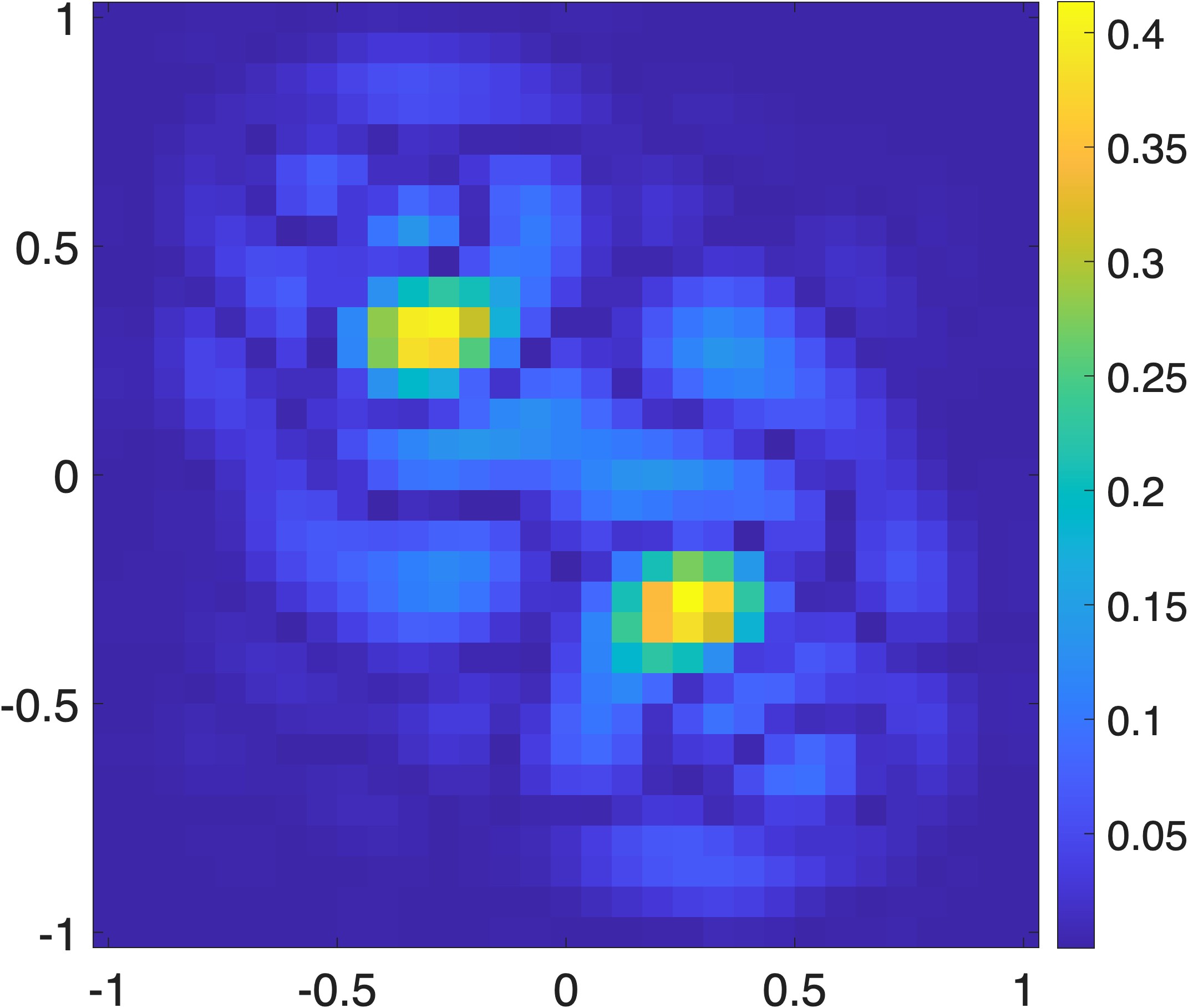}
}
\quad
\subfloat[$\frac{|p_{0}^{\rm true}-p_{0}^{\rm rec}|}{\|p_{0}^{\rm true}\|_{L^\infty(\Omega)}}$\label{fig:case3_err_p0}]{
    \includegraphics[width=0.25\textwidth]{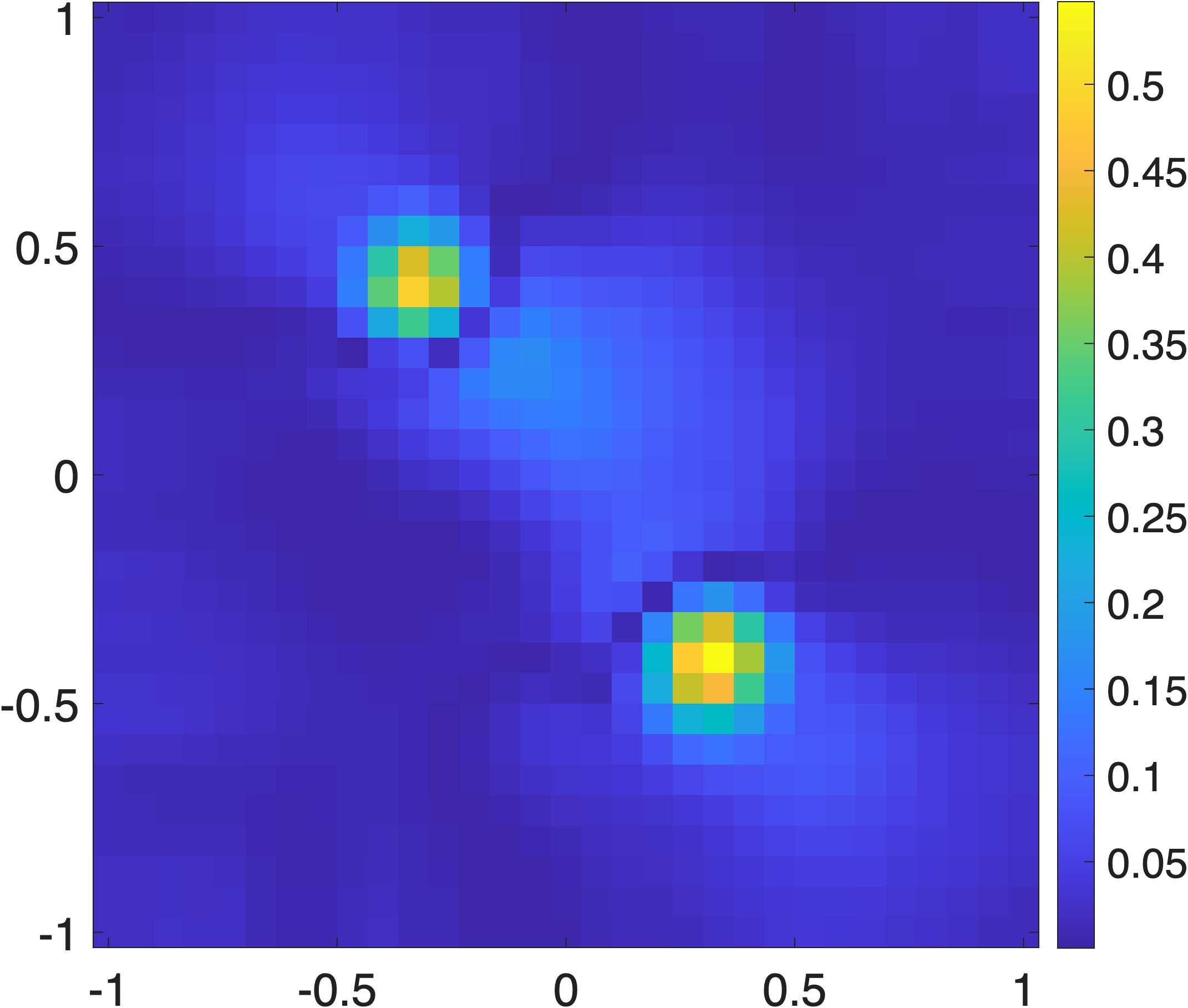}
}

\caption{Test 3. First row: true initial velocity components and pressure. Second row: reconstructed initial velocity components and pressure. Third row: pointwise relative reconstruction errors.}
\label{fig:case3_all}
\end{figure}

As in Test~1, we monitor the quantities $\operatorname{relU}^{(k)}$, $\operatorname{relP}^{(k)}$, $\operatorname{ResU}^{(k)}$, and $\operatorname{ResP}^{(k)}$ defined in \eqref{5.7}--\eqref{5.9}. The corresponding results for Test~3 are shown in Figure~\ref{fig:case3_hist}.

\begin{figure}[ht]
\centering

\subfloat[$\operatorname{relU}^{(k)}$ and $\operatorname{relP}^{(k)}$\label{fig:case3_rel_hist}]{
    \includegraphics[width=0.42\textwidth]{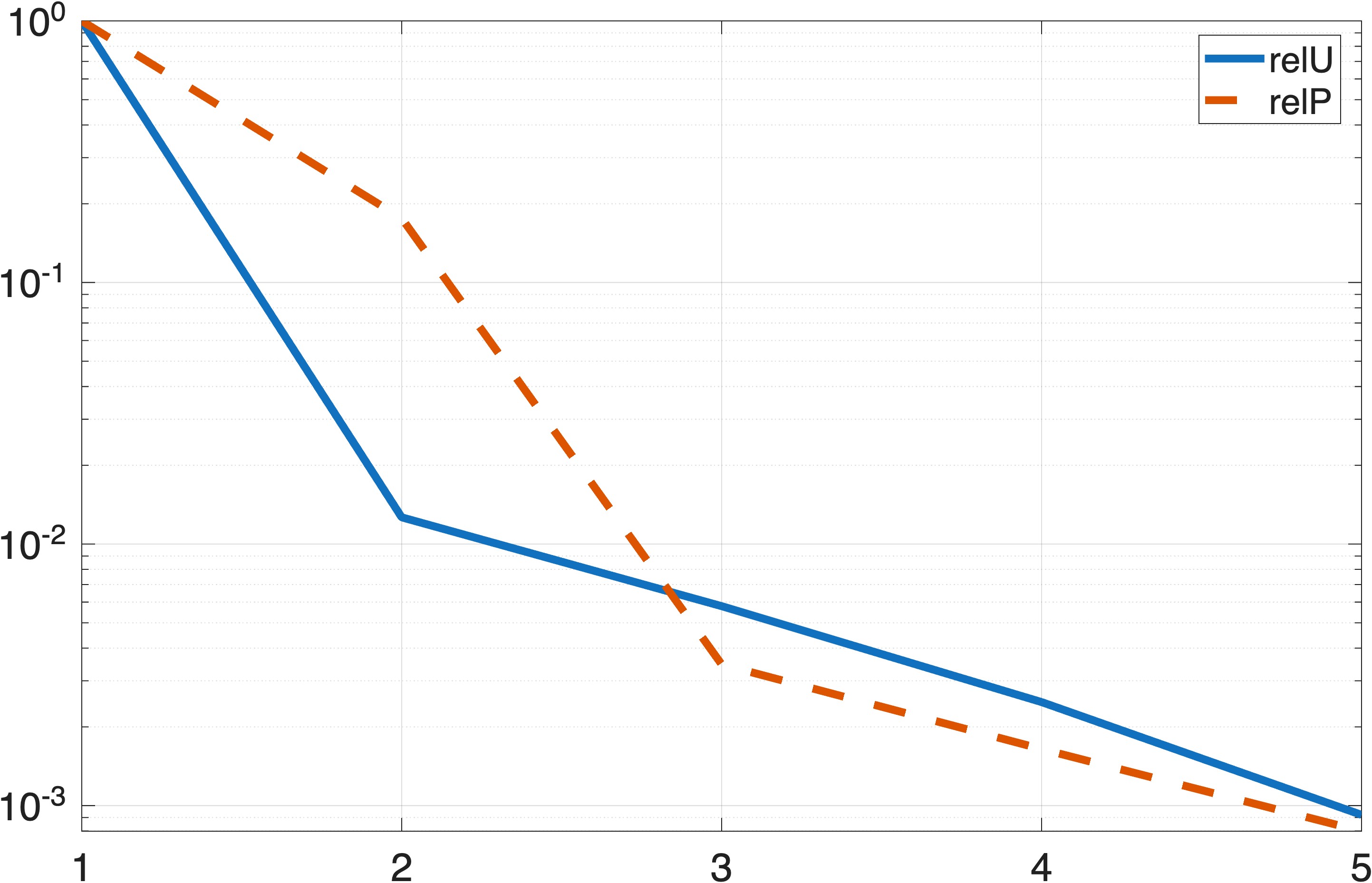}
}
\quad
\subfloat[$\operatorname{ResU}^{(k)}$ and $\operatorname{ResP}^{(k)}$\label{fig:case3_res_hist}]{
    \includegraphics[width=0.42\textwidth]{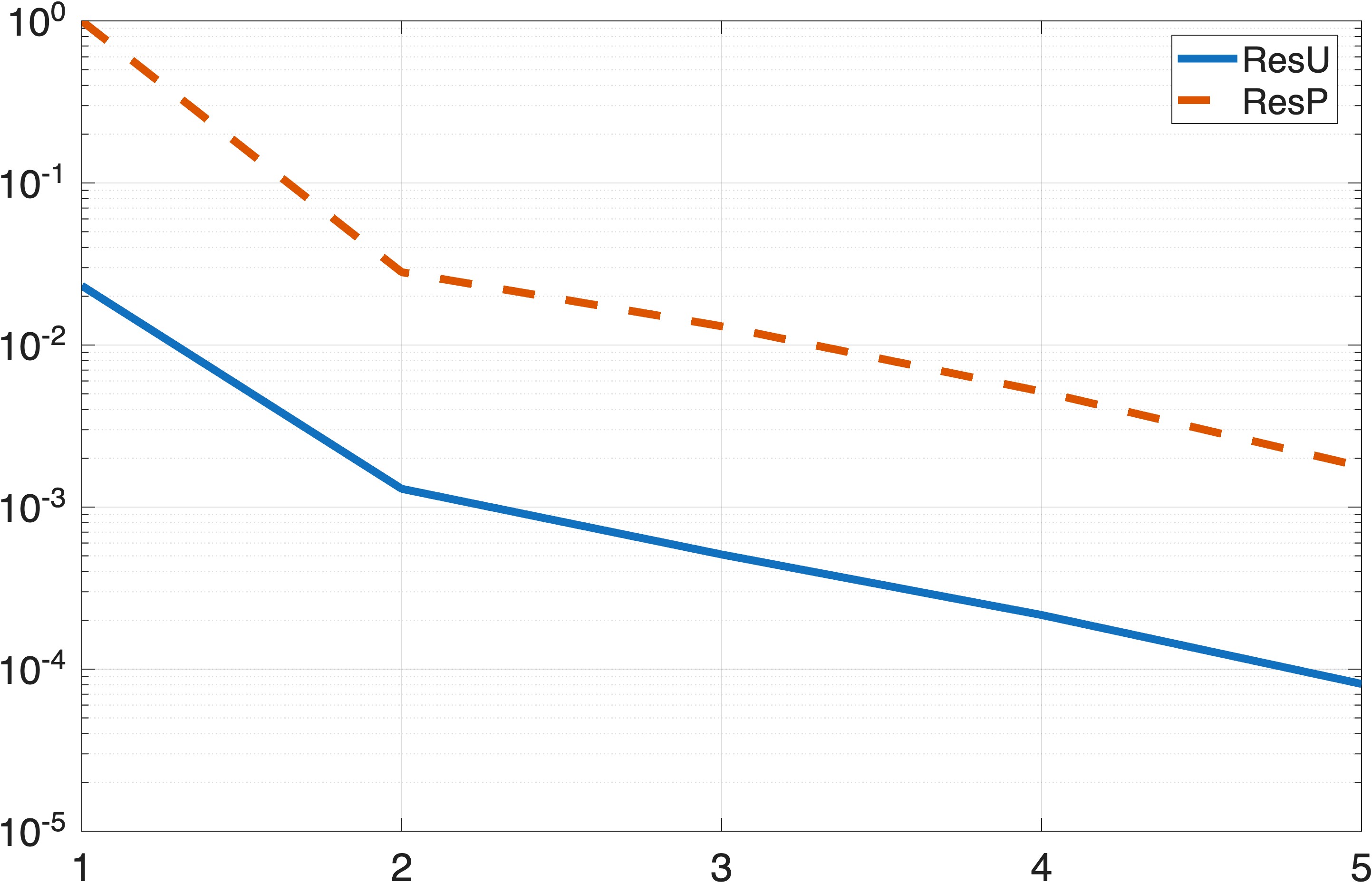}
}

\caption{Test 3. Convergence history of the Picard iteration. The left image shows the relative changes $\operatorname{relU}^{(k)}$ and $\operatorname{relP}^{(k)}$, while the right image shows the relative residuals $\operatorname{ResU}^{(k)}$ and $\operatorname{ResP}^{(k)}$.}
\label{fig:case3_hist}
\end{figure}

The numerical results for Test~3 remain satisfactory despite the increased complexity of this example. Figure~\ref{fig:case3_all} shows that the proposed method successfully reconstructs the main structures of the true initial velocity and pressure, including the locations, signs, and overall shapes of the dominant features. In particular, the reconstructed solution captures well the two separated localized structures in the velocity field as well as the main features of the pressure distribution. The corresponding pointwise relative error plots indicate that the largest discrepancies are concentrated near the regions of strongest variation, while the errors remain moderate in most parts of the domain. Using the relative $L^2$ error $\|q^{\rm rec}-q^{\rm true}\|_{L^2(\Omega)}/\|q^{\rm true}\|_{L^2(\Omega)}\times 100\%$, the reconstruction errors are $35.98964\%$ for $u_{0,1}$, $31.06575\%$ for $u_{0,2}$, and $42.79563\%$ for $p_0$. Although these errors are larger than those in Tests~1 and~2, the reconstructions still recover the essential qualitative features of the exact solution. In addition, Figure~\ref{fig:case3_hist} shows a stable convergence behavior of the Picard iteration: both the relative changes of the iterates and the relative residuals decrease steadily as the iteration proceeds, which indicates that the iteration stabilizes reliably and that the computed solution satisfies the reduced system with good accuracy.

\section{Concluding remarks} \label{sec:conclusion}

In this paper, we study an inverse initial-data problem for the incompressible Navier--Stokes system. The goal was to reconstruct the initial velocity $\bu_0$ and the initial pressure $p_0$ from lateral boundary observations, without assuming that the time-independent body force $\bff$ is known a priori. To overcome this difficulty, we first eliminated $\bff$ by differentiating the Navier--Stokes system with respect to time. We then applied the Legendre polynomial--exponential time-dimensional reduction method to derive a finite-dimensional coupled elliptic system for the coefficient vectors $\bU^N$ and $P^N$. For this reduced system, we introduced a Carleman contraction framework and proved that the associated operator $\mathcal T_{\lambda,\epsilon}$ is contractive on a suitable admissible set when $\lambda$ is sufficiently large. As a consequence, the reduced system admits a unique fixed point, which provides a stable approximation of the desired coefficient vectors and, in turn, of the initial data.

We also proposed a practical numerical algorithm based on this contraction mapping principle and used it to reconstruct the initial velocity and pressure from synthetic boundary data. The numerical examples show that the proposed method recovers the main qualitative features of the true initial data with satisfactory accuracy, while the convergence histories confirm the stability of the Picard iteration and the accuracy with which the reduced system is satisfied. These results indicate that the combination of time-dimensional reduction and the Carleman--Picard method is a promising approach for inverse problems for nonlinear fluid models. Possible directions for future work include the treatment of partial boundary data, more detailed convergence and stability estimates for the fully discrete scheme, and extensions to more general fluid systems and three-dimensional geometries.

%\section*{Conflict of interest}
%The authors declare no conflict of interest.
%
%
%
\section*{Data availability statement}
The data that support the findings of this study are available from the corresponding author upon reasonable request.

%\section*{Ethics statement}
%This study did not involve human participants, human data, animal experiments, or clinical trials.
%

%\bibliographystyle{plain}
%\bibliography{../../../../mybib}
%%\bibliography{mybib}

\begin{thebibliography}{10}
\bibitem{Liu2024}
G.~Liu.
\newblock Determining the viscosity function from the boundary measurements for
  the stokes and the navier--stokes equations.
\newblock {\em Inverse Problems}, 40(12):125011, dec 2024.
\end{thebibliography}
\end{document}